\documentclass[11pt,letterpaper]{amsart}
\usepackage[utf8]{inputenc}
\usepackage[T1]{fontenc}
\usepackage[margin=1in]{geometry}
\usepackage{amsmath}
\usepackage{amssymb}
\usepackage{amsthm}
\usepackage{amscd}
\usepackage{enumerate}
\usepackage{xcolor}
\usepackage{subcaption}
\usepackage{mathrsfs}
\usepackage{tikz-cd}
\usepackage{mathtools}

\usepackage{graphicx}
\usepackage{amsmath,amsthm,amsfonts,amscd,amssymb,comment,eucal,latexsym,mathrsfs}
\usepackage{stmaryrd}
\usepackage[all]{xy}

\usepackage{epsfig}

\usepackage[all]{xy}
\xyoption{poly}
\usepackage{fancyhdr}
\usepackage{wrapfig}
\usepackage{epsfig}

\theoremstyle{plain}
\newtheorem{thm}{Theorem}[section]

\newtheorem{prop}[thm]{Proposition}
\newtheorem{lem}[thm]{Lemma}
\newtheorem{cor}[thm]{Corollary}
\newtheorem{conj}[thm]{Conjecture}

\theoremstyle{definition}
\newtheorem{defn}[thm]{Definition}
\newtheorem{remark}[thm]{Remark}
\newtheorem{example}[thm]{Example}

\newtheorem{question}[thm]{Question}

 \def\cB{{\mathcal{B}}}   \def\cE{{\mathcal{E}}} \def\cF{{\mathcal{F}}} \def\cG{{\mathcal{G}}}      \def\cM{{\mathcal{M}}}  \def\cO{{\mathcal{O}}} \def\cP{{\mathcal{P}}} \def\cQ{{\mathcal{Q}}} \def\cR{{\mathcal{R}}} \def\cS{{\mathcal{S}}}  \def\cU{{\mathcal{U}}}   \def\cX{{\mathcal{X}}} \def\cY{{\mathcal{Y}}} \def\cZ{{\mathcal{Z}}}

\newcommand\Cay{\operatorname{Cay}}

\newcommand\id{\operatorname{id}}

\DeclareMathOperator{\SO}{SO}

\newcommand\Sym{\operatorname{Sym}}

\newcommand\SL{\operatorname{SL}}

\makeatletter
\newtheorem*{rep@theorem}{\rep@title}
\newcommand{\newreptheorem}[2]{%
\newenvironment{rep#1}[1]{%
 \def\rep@title{#2 \ref{##1}}%
 \begin{rep@theorem}}%
 {\end{rep@theorem}}}
\makeatother
\newreptheorem{theorem}{Theorem}

\begin{document}

\author{Mikołaj Fr\k{a}czyk}
\address{\parbox{\linewidth}{Faculty of Mathematics and Computer Science, Jagiellonian University, ul. Łojasiewicza 6, 30-348 Krak{\'o}w, Poland
}}
\email{mikolaj.fraczyk@uj.edu.pl}
\urladdr{https://sites.google.com/view/mikolaj-fraczyk/home}

\title[]{Infinitesimal containment and sparse factors of iid}
\begin{abstract}

We introduce infinitesimal weak containment for measure-preserving actions of a countable group $\Gamma$: an action $(X,\mu)$ is infinitesimally contained in $(Y,\nu)$ if the statistics of the action of $\Gamma$ on small measure subsets of $X$ can be approximated inside $Y$. We show that the Bernoulli shift $[0,1]^\Gamma$ is infinitesimally contained in the left-regular action of $\Gamma$. For exact groups, this implies that sparse factor-of-iid subsets of $\Gamma$ are approximately hyperfinite. We use it to quantify a theorem of Chifan--Ioana on measured subrelations of the Bernoulli shift of an exact group. For the proof of infinitesimal containment we define \emph{entropy support maps}, which take a small subset $U$ of $\{0,1\}^I$ and assign weights to coordinates above every point of $U$, according to how ''important'' they are for the structure of the set.

\end{abstract}
\maketitle

\section{Introduction}
\subsection{Motivation from measured group theory}
Our goal is to study the dynamics of very small measure subsets in probability measure preserving actions of a countable group $\Gamma$. Our motivation comes from the problem of estimating the cost of p.m.p. actions \cite{Gab1}, where wanted to identify new obstructions to some strategies of proving fixed price. We are not the only ones to make that connection, a similar motivation to study the sparse factor of iid subsets was given in \cite{pete2025nonamenable} and \cite{csokapetemester}.
Let us illustrate it on the example of fixed price problem for fundamental groups of hyperbolic 3-manifolds \cite{Gab1, abert2012rank}. 

The rank of a countable group $\Lambda$, denoted ${\rm rk}(\Lambda)$ is the minimal number of generators of $\Lambda$. The relation between the rank and the index of finite index subgroups was investigated in \cite{lackenby2005expanders}, where Lackenby considered the rank gradient
$\lim_{n\to\infty}\frac{\rm rk(\Lambda_i)}{[\Lambda:\Lambda_i]},$ where $\Lambda_i$ is a sequence of finite index subgroups with $[\Lambda:\Lambda_i]\to\infty.$ 
Now, let $\Gamma$ be the fundamental group of a compact hyperbolic $3$-manifold $M$. Should we expect that the rank of finite index subgroups of $\Gamma$ grows linearly or sub-linearly in the index? This question was highlighted in the work of Abert-Nikolov \cite{abert2012rank}, where it was related to the fixed price question \cite{Gab1} and Heegaard genus conjecture \cite{lackenby2002heegaard}. For context, the rank of fundamental groups of hyperbolic surfaces grows linearly in the volume, hence linearly in the index of a subgroup, while for higher rank lattices it was recently shown \cite{fraczyk2023poisson} that the rank grows sub-linearly in the index. The status of lattices in rank one groups still remains unresolved, with the exception of the surface groups.

The notion of rank has a measured analogue called cost, which turned out to be a key tool to study the rank growth. Let $\Gamma\curvearrowright (X,\mu)$ be a probability measure preserving action on a standard Borel probability space $X$, which we will abbreviate as p.m.p. action. The \emph{orbit equivalence relation} $\cR$ is the measurable subset $\{(x,\gamma x)\mid x\in X, \gamma\in \Gamma\}\subset X\times X.$ A subset $\cG\subset \cR$ is a \emph{graphing} of $\cR$ if for any pair $(x,y)\in \cR$ there exists a finite chain $(x_0,x_1,\ldots, x_k)$ such that $x=x_0, y=x_k$ and $(x_i,x_{i+1})$ or $(x_{i+1},x_i)\in \cG$ for all $i=0,\ldots, k-1.$ One might think of $\cG$ as a collection of bridges between the points of $X$ that allow us to move between any two elements in the same equivalence class of $\cR$, or as a measurable collection of directed graphs on equivalence classes of $\cR$ that are connected as undirected graphs. The out-degree of a point $x\in X$ is defined as $\overrightarrow{\deg}_{\cG}(x):=|\{y\in X\mid (x,y)\in \cG\}$. We can invert the graphing by flipping the coordinates $\cG^{-1}:=\{(y,x)\in \cR\mid (x,y)\in \cG\}.$ The in-degree $\overleftarrow{\deg}_{\cG}(x)$ of $x$ in $\cG$ is the out-degree in $\cG^{-1}.$ The total degree is the sum of in-and-out-degrees.

The cost of the graphing is given by $${\rm cost}(\cG)=\int \overrightarrow{\deg}_{\cG}(x) d\mu(x)$$ and the cost of the  action $(X,\mu)$ is 
${\rm cost}(X,\mu):=\inf_{\cG}{\rm cost}(\cG),$ where the infimum is taken over all graphings $\cG$ of $\cR$. If $S\subset \Gamma$ generates $\Gamma$, then $\cG=\{(x,sx)\mid s\in S\}$ is a graphing so $\rm cost(X,\mu)\leq |S|.$ A group is said to have the fixed price property if any two essentially free actions have the same cost. It is known that free groups, surface groups \cite{Gab1}, higher rank lattices \cite{fraczyk2023poisson} and products of infinite countable groups \cite{khezeli2025products,bevilacqua2025metric} have fixed price. On the other hand, it is still open whether the hyperbolic 3-manifold groups have fixed price. Similarly for lattices in rank one real Lie groups, except those isogenous to $\SO(2,1)$. No examples of groups that provably fail fixed price are known. 

Going back to the original problem on ranks of subgroups, in \cite{abert2012rank}, it was observed that the asymptotic ratio of the rank over index can be recovered from the cost of a suitable profinite p.m.p. action. Let $\Gamma_n\subset \Gamma$ be a sequence of finite index subgroups with $\Gamma_{n+1}\subset\Gamma_n.$ We can form a p.m.p. action os the inverse limit $\lim_{\leftarrow}\Gamma/\Gamma_n $ of coset actions $\Gamma/\Gamma_n$. %This space is the profinite completion $\widehat{\Gamma}$ of $\Gamma$ divided by the intersection of closures of $\Gamma_n$ in $\widehat{\Gamma}$, with the normalized quotient Haar measure. 
This is a profinite action that extends to a transitive action by the profinite completion of $\Gamma$. By \cite{abert2012rank}, $${\rm cost}(\lim_{\leftarrow}\Gamma/\Gamma_n,\mu)-1=\lim_{n\to\infty}\frac{{\rm rk}(\Gamma_n)}{[\Gamma:\Gamma_n]}.$$ If the cost of $(\lim_{\leftarrow}\Gamma/\Gamma_n,\mu)$ were $1$, we could conclude that the rank of $\Gamma_n$ grows sublinearly in the index. Unfortunately, for sequences of arithmetic interest, like e.g. the principal congruence subgroups of arithmetic lattices in $\SO(3,1)$, there are currently no methods to compute the cost of the relevant profinite actions. This is where the fixed price property could help by reducing the problem to computing the cost of a more "friendly" action. Indeed, hyperbolic $3$-manifold groups always do admit cost one actions. Let us explain how to construct them. By Agol's virtual fibering theorem \cite{agol2008criteria}, there is a finite index subgroup $\Gamma'\subset \Gamma$, a surface group $N\triangleleft \Gamma'$, normal in  $\Gamma'$ such that $\Gamma'=N\rtimes \langle s\rangle $, where $s\in \Gamma'$ induces a pseudo-Anosov automorphism on $\Gamma'.$ For simplicity assume $\Gamma=\Gamma'$, otherwise we would need to induce the p.m.p. action from $\Gamma'$ to $\Gamma$ to construct an example.
Let $\rho\colon \Gamma\to \Gamma/N\simeq \mathbb Z$ be the quotient map and let $(Y,\nu)$, $(Z,\tau)$ be ergodic p.m.p. actions of $\Gamma$ such that $(Y,\nu)$ is essentially free and $(Z,\tau)$ factors through $\rho$ and is essentially free as an action of $\mathbb Z=\Gamma/N$. We claim that the action $(Y\times Z,\nu\times \tau)$ has cost one. Fix a finite generating set $W\subset N$. For any positive measure subset $U\subset Z$ we construct a graphing $$\cG_U=\{((y,z),(wy,z))|y\in Y, z\in U, w\in W\}\cup \{((y,z),(sy,sz)), y\in Y, z\in Z\},$$ as illustrated in Figure \ref{fig1}. 
The fact that it is really a graphing is a pleasant exercise left to the reader. 
The graphs induced by the graphing $\cG_U$ on the orbits have a particular shape.  As $\tau(U)$ goes to $0$ we see a sparse set of cosets of $N$, each connected by a fixed generating set of $N$ (the high degree part). In between, we see long bridges built from $s$ (the low degree part). 
The cost of the graphing $\cG_U$ is $|W|\tau(U)+1,$ so by making $\tau(U)$ smaller we can drive it as close to $1$ as we want. This means that $(Y\times Z,\nu\times \tau)$ has cost one, because the cost of any essentially free action is at least $1$. A graphing with cost close to $1$ will be called \emph{cheap graphing}.
\begin{figure}
\includegraphics[height=10cm, trim={2cm 5.5cm 2cm 0cm},clip]{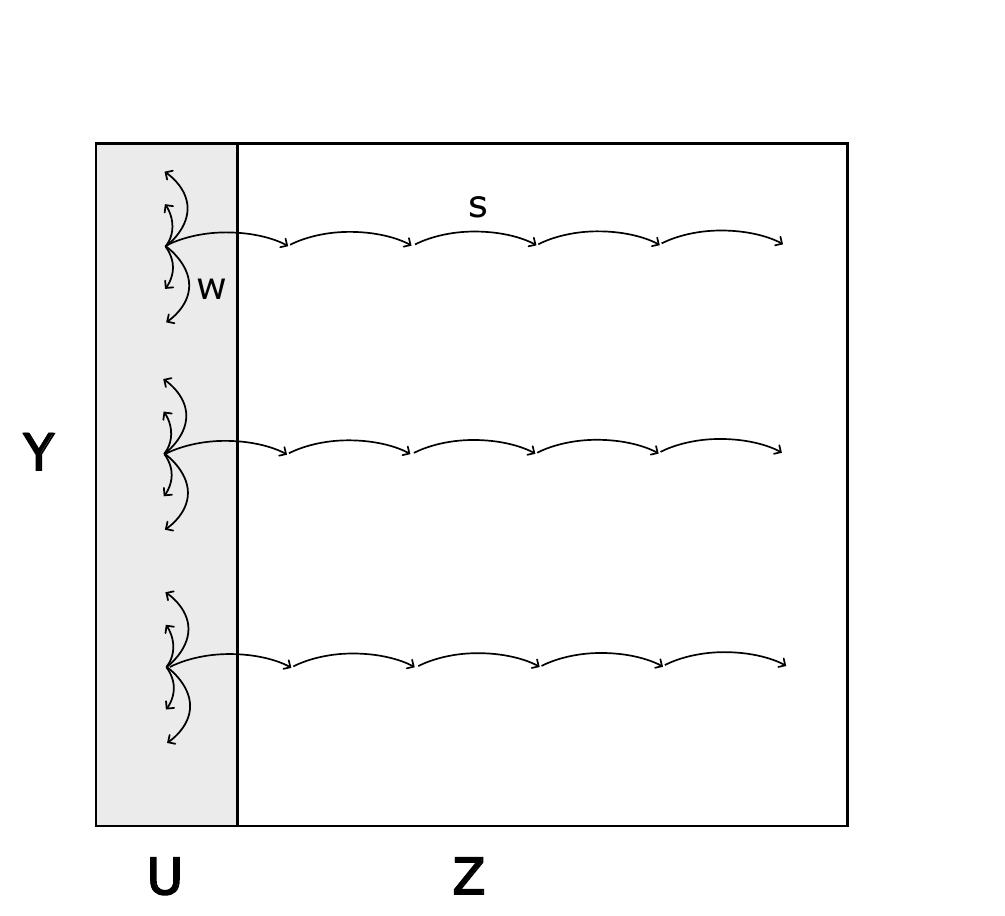}
\caption{Graphing $\cG_U$}
\label{fig1}
\end{figure}

Fixed price property predicts that we should somehow be able to "transplant" these cheap graphings to any free p.m.p. action of $\Gamma$. Let us highlight two features of the graphing $\cG_U$ that we have just constructed, which are shared by all cheap graphings of p.m.p. actions.
\begin{enumerate}
    \item There is a small part of the space where the graphing has high ($\geq 3)$ total degree. We expect that it has a special structure. In our last example this was $Y\times U,$ preserved by the relatively large subgroup $N\triangleleft \Gamma.$
    \item The rest of the space witnesses small total degree, equal $2$ or $1$.
\end{enumerate}

Motivated by the example of $\cG_U$, we believe that the geometry of the high degree part plays a key role in the final connectivity properties of any cheap graphing. Hence, whether we want to prove or disprove the existence of cheap graphing, it is important to understand the geometry of small subsets of p.m.p. actions.

By \cite{AbertWeiss2013BernoulliWeaklyContained}, the Bernoulli shift $\Gamma\curvearrowright ([0,1]^\Gamma, {\rm Leb}^\Gamma)$ has maximal cost among all essentially free p.m.p. actions. Since we know $\Gamma$ has some cost one actions, the fixed price property holds for $\Gamma$ if and only if ${\rm cost}([0,1]^\Gamma,{\rm Leb}^\Gamma)=1$, i.e. if we are able to find cheap graphings generating the orbit equivalence relation on $([0,1]^\Gamma, {\rm Leb}^\Gamma)$. Could these graphings look approximately like  $\cG_U$, that is, be a sparse family of non-amenable subgraphs connected by long bridges? Our Theorem \ref{mthm-approxhyp} implies that the answer is no. In fact, we can do it for a large class of groups, including all Gromov hyperbolic groups. We hope our results can open new ways to lower bound the cost of p.m.p. actions for hyperbolic groups.

%Does the limit $$\lim_{n\to\infty}\frac{{\rm rk}(\Gamma_n)}{[\Gamma:\Gamma_n]}$$ exists for any sequence of normal subgroups $\Gamma_n$ with $\Gamma_{n+1}\subset \Gamma_n$ and $\bigcap_{n=1}^\infty \Gamma_n=\{1\}$? 

\subsection{Infinitesimal containment} Let $\Gamma$ be a countable group. The weak containment was introduced by Kechris \cite{Kechris2010GlobalAspects}. It is a relation between probability measure preserving actions of $\Gamma$. 
\begin{defn}\cite{Kechris2010GlobalAspects,BurtonKechris2020WeakContainment} $(X,\mu)$ is weakly contained in $(Y,\nu)$ if for any finite set $F\subset \Gamma$, $k\in \mathbb N$, any finite measure subsets $A_1,\ldots, A_k$ and $\varepsilon>0$ we can find subsets $B_1,\ldots, B_k\subset Y$ such that 
$$|\mu(A_i\cap \gamma A_j)-\nu(B_i\cap \gamma B_j)|\leq \varepsilon\max_{i=1,\ldots,k}\{\mu(A_i)\}.$$
\end{defn}
If an action $(X,\mu)$ is weakly contained in $(Y,\nu)$, we write $(X,\mu)\xhookrightarrow{w}(Y,\nu)$.
Weak containment introduces a useful order on the set of p.m.p. actions. Certain invariants, most notably the cost of an action \cite{GaboriauSurvey}, are monotone with respect to weak containment \cite{AbertWeiss2013BernoulliWeaklyContained}.
We introduce two weaker types of containment, called the \emph{projective} and the \emph{infinitesimal containment}. Projective containment is arguably better suited for comparisons between infinite measure preserving actions while the infinitesimal containment is designed to capture the dynamics of small subsets in a p.m.p. action. %We use it to study sparse factor of iid subsets of Cayley graphs and, more broadly, sparse subsets of countable groups arising as factors of p.m.p. actions. Already in the case of factors of iid, infinitesimal containment goes beyond the scope of current tools like e.g. the entropy inequalities or spectral theory and provides strong geometric constraints on sparse factor of iid subsets (Theorem \ref{mthm-approxhyp}). 

In what follows, $(X,\mu),(Y,\nu),(Z,\tau)$ stand for measure preserving actions of a countable group $\Gamma$ on standard Borel spaces. We do not assume these are finite measure, but we'll assume they are $\sigma$-finite. All sets appearing in the sequel will be measurable.

\begin{defn} $(X,\mu)$ is projectively contained in $(Y,\nu)$ if for any finite set $F\subset \Gamma$, $k\in \mathbb N$, any finite measure subsets $A_1,\ldots, A_k$ and $\varepsilon>0$ we can find $\lambda>0$ and subsets $B_1,\ldots, B_k\subset Y$ such that 
$$|\mu(A_i\cap \gamma A_j)-\lambda\nu(B_i\cap \gamma B_j)|\leq \varepsilon\max_{i=1,\ldots,k}\{\mu(A_i)\}.$$
\end{defn}
We will write $(X,\mu)\xhookrightarrow{proj} (Y,\nu)$ to indicate that $(X,\mu)$ is projectively contained in $(Y,\nu)$. To recover the definition of weak containment we would have to fix $\lambda=1$. As opposed to weak containment, we can find instances of infinite measure preserving actions projectively contained in a p.m.p. action.  Throughout the paper all actions of $\Gamma$ on discrete spaces are always endowed with the counting measure. We suppress the measure from notation, or write $|\cdot|$ if needed. %The following examples of projective containment are verified in Section \ref{sec-containment}.
\begin{example}
\begin{enumerate}
    \item if $(X,\mu)$ is essentially free, then $\Gamma\xhookrightarrow{proj} (X,\mu).$ 
    \item Let $\Gamma\curvearrowright \Xi$ be an action on a countable set. Then $\Xi\xhookrightarrow{proj} ([0,1]^\Xi,{\rm Leb}^\Xi).$
    \item (Theorem \ref{thm-exactprojcont}) If $\Gamma$ is an exact group, then $(X,\mu)\xhookrightarrow{proj} \Gamma$ if and only if the action $(X,\mu)$ is amenable \cite{Zimmer3}.
\end{enumerate}    
\end{example}
   Weak containment is often compared to the weak containment between unitary representations in the sense of Zimmer \cite[Section 1]{BurtonKechris2020WeakContainment}. This is the version of weak containment where one approximates matrix coefficients of the first representation by the matrix coefficients of the second.
   It is well known that weak containment of p.m.p. actions $(X,\mu)\xhookrightarrow{w}(Y,\nu)$ implies the weak containment of unitary representations $L^2(X,\mu)\xhookrightarrow{w} L^2(Y,\nu)$ \cite[Prop 10.5]{Kechris2010GlobalAspects}. It is also the case that $(X,\mu)\xhookrightarrow{proj}(Y,\nu)$ implies $L^2(X,\mu)\xhookrightarrow{w}L^2(Y,\nu).$

\begin{remark}
During final stages of the writeup of this paper we learned that our results on the particular type of projective containment $(Y,\nu)\xhookrightarrow{proj}\Gamma$, have significant overlap with the recent paper \cite{bevilacqua2025metric} of Bevilacqua and Bowen. They introduce the notion of limit regular actions, which in our language are exactly the actions projectively contained in $\Gamma$. Our Theorem \ref{thm-exactprojcont} stating that for exact groups $\Gamma$ the containment $(Y,\nu)\xhookrightarrow{proj}\Gamma$ holds if and only if $(Y,\nu)$ is amenable is equivalent to the corresponding statements on limit regular actions in \cite{bevilacqua2025metric}. We believe most of our results in section \ref{sec-containment} on projective containment in the action of the group on itself could be extracted from \cite{bevilacqua2025metric}. Still, some of the tools we develop to prove Theorem \ref{thm-exactprojcont} are valid for general projective containment, so we decided to leave our original argument unchanged.
\end{remark}
   
We can further relax the definition of the projective containment by requiring that the approximation only holds for small enough subsets. In this way we arrive at the \emph{infinitesimal containment}, which is the main focus of this work.  
\begin{defn}\label{def-infcont}
$(X,\mu)$ is infinitesimally contained in $(Y,\nu)$ if for any finite set $F\subset \Gamma$, $k\in \mathbb N$, $\varepsilon>0$ there exists $\delta>0$ with the following property. For any positive measure sets $A_1,\ldots, A_k\subset X$ satisfying $\mu(A_i)\leq \delta$ we can find $\lambda>0$ and subsets $B_1,\ldots, B_k\subset Y$ such that 
$$|\mu(A_i\cap \gamma A_j)-\lambda\nu(B_i\cap \gamma B_j)|\leq \varepsilon\max_{i=1,\ldots k} \mu(A_i).$$
\end{defn}

 We write $(X,\mu)\xhookrightarrow{inf} (Y,\nu)$ to indicate that $(X,\mu)$ is infinitesimally contained in $(Y,\nu)$. It is quite difficult to find even a single non-trivial instance of infinitesimal containment (i.e. where it would not follow from weak or projective containment). 
\begin{thm}\label{mthm-infcontBernouilli}
The Bernoulli shift $([0,1]^\Gamma, {\rm Leb}^\Gamma)$ is infinitesimally contained in $\Gamma.$
\end{thm}

When it comes to the right-hand side of the containment, Theorem \ref{mthm-infcontBernouilli} is strongest possible. Being infinitesimally contained in $\Gamma$ implies infinitesimal containment in any essentially free measure preserving action. As a corollary we obtain an infinitesimal analogue of the Abert-Weiss theorem \cite{AbertWeiss2013BernoulliWeaklyContained}.
\begin{cor}
The Bernoulli shift $([0,1]^\Gamma, {\rm Leb}^\Gamma)$ is infinitesimally contained in any essentially free measure preserving action of $\Gamma.$   
\end{cor}

On the other hand, Bernoulli shift is formally the "easiest" case of infinitesimal containment for a p.m.p. action. Let us explain. Suppose $(X,\mu)\xhookrightarrow{inf}\Gamma$ and $(X,\mu)$ is essentially free. By \cite{AbertWeiss2013BernoulliWeaklyContained}, Bernoulli shift is weakly contained in  $\Gamma\curvearrowright(X,\mu)$. By the transitivity properties established in Lemma \ref{lem-transitivity}, $([0,1]^\Gamma,{\rm Leb}^\Gamma)\xhookrightarrow{w}(X,\mu)\xhookrightarrow{inf}\Gamma$ yields $([0,1]^\Gamma,{\rm Leb}^\Gamma)\xhookrightarrow{inf}\Gamma$. %The transitivity properties proved in Lemma \ref{lem-transitivity} imply that infinitesimal containment $(X,\mu)\xhookrightarrow{inf}(Y,\nu)$ can be pulled back to $([0,1]^\Gamma, {\rm Leb}^\Gamma)\xhookrightarrow{inf}\Gamma.$ 
This means that if any essentially free p.m.p. action is infinitesimally contained in $\Gamma$, then we must have $([0,1]^\Gamma, {\rm Leb}^\Gamma)\xhookrightarrow{inf}\Gamma.$ 
Fortunately, this is true thanks to Theorem \ref{mthm-infcontBernouilli}, so the theory is not vacuous. By the same transitivity properties, any action weakly contained in Bernoulli shift is infinitesimally contained in $\Gamma$. This begs the question whether there are other such actions. Is there a profinite action of $\Gamma$ infinitesimally contained in $\Gamma$?  We cautiously hope that there are many more interesting cases of infinitesimal containment between p.m.p. actions and actions on $\Gamma$ on countable sets. We list some questions in Section \ref{sec-questions}.

Theorem \ref{mthm-infcontBernouilli} is generalized to Bernoulli shifts $([0,1]^{\Gamma/H},{\rm Leb}^{\Gamma/H})$ where $H\subset \Gamma$ is a subgroup. This is the contents of Theorem \ref{thm-infcontcosetB}.

\subsection{Sparse factor of iid subsets}
 We say that a random subset $F$ of $\Gamma$ is a factor of iid if it agrees in distribution with the set $\{\gamma\in\Gamma\mid \gamma^{-1} x\in  U\}$ where $U\subset [0,1]^\Gamma$ is a positive measure subset and $x\in [0,1]^\Gamma$ is ${\rm Leb}^\Gamma$-random. In general, any additional structure on $\Gamma$ is a factor of iid if it depends measurably on a uniform random point $w\in [0,1]^\Gamma$. There is an extensive literature on what can or cannot be done as a factor of iid \cite{HatamiLovaszSzegedy2014LocalGlobal} and these questions connect to the limitation of local algorithms \cite{GamarnikSudan2017Limits} running on Cayley graphs. Our contribution is towards describing the \emph{sparse} factor of iid subsets of $\Gamma.$ The density of a factor of iid subset 
$\{\gamma\in\Gamma\mid \gamma^{-1} x\in  U\}$ is defined as the measure  $\rm Leb^\Gamma(U)$, and sparse subsets are those where this measure is very small. In a recent work \cite{pete2025nonamenable} Rokob and Pete give several constructions of sparse factor of iid subsets using so called Poisson zoos, which are sparse unions of independently sampled translates of finite subsets.

{In an even more reecnt work Cs\'oka, Mester and Pete \cite{csokapetemester} use entropy inequalities on trees (see \cite{BackhauszGerencserHarangi2019EntropyFIID}) to show that sparse factor of iid subsets of a regular tree have average degree close to $2$. Having expected degree close to $2$ implies that these sets can be cheaply cut into finite pieces -- by removing a relatively small proportion of the set we are left with only finite connected components. This places strong restrictions on the geometry of sparse factor of iid subsets of a tree.} For general Cayley graphs the property of being able to cheaply cut the set into finite pieces is called \emph{approximate hyperfiniteness}.

\begin{defn}
Let $\Gamma$ be a countable group and let $S\subset \Gamma$ be a finite set. An invariant random subset\footnote{This means that the distribution of $E$ is a $\Gamma$-invariant probability measure on $\{0,1\}^\Gamma$.} $E\subset \Gamma$ is called $(S,k,\varepsilon)$-hyperfinite if there exists an invariant random coupling $C\subset E$, ${\mathbb P}(1\in C)\leq \varepsilon \mathbb P(1\in E)$ such that the connected components of the Cayley graph $\rm Cay(\Gamma,S)$ restricted to $E\setminus C$ are all finite of size at most $k$. A sequence of factor of iid subsets is called approximately hyperfinite if they are eventually $(S,k,\varepsilon)$-hyperfinite for any $\varepsilon>0$, any finite $S\subset \Gamma$ and some $k>0$. 
\end{defn}

Note that here we are making the definition independent of the choice of the Cayley graph, by requiring that approximately hyperfinite sequences can be cheaply cut for any generating set. We prove that in arbitrary exact groups, sparse factor of iid subsets are approximately hyperfinite. {This answers  \cite[Question 1.4]{csokapetemester}}.

\begin{thm}\label{mthm-approxhyp}
Let $\Gamma$ be an exact group. For any $\varepsilon>0$ and any finite subset $S\subset \Gamma$ there exist $k>0$ and $\delta>0$ such that all factor of iid subsets of density at most $\delta$ are $(S,k,\varepsilon)$-hyperfinite. In particular, any sequence of factor of iid subsets with vanishing density is approximately hyperfinite.
\end{thm}

We deduce Theorem \ref{mthm-approxhyp} from Theorem \ref{mthm-infcontBernouilli}. While we do not use entropy inequalities, the proof of Theorem \ref{mthm-approxhyp}  draws some inspiration from the proofs of entropy inequalities on Cayley graphs \cite{csoka2020entropy}. Theorem \ref{mthm-approxhyp} can fail for non-exact groups, since these contain sequences of small-scale expanders \cite{BNSWW13,elek2021uniform} so a low density union of independent random translates of such small-scale expanders would be a sparse factor of iid subset which fails to be approximately hyperfinite. We refer to the proof of \cite[Theorem 5.1]{jardon2025exactness} for the argument why small scale expanders cannot be approximately hyperfinite. This example was pointed to us by Tom Hutchcroft over a coffer break in September 2024.

\subsection{Chifan-Ioana's theorem} Infinitesimal containment can be used to strengthen Chifan-Ioana's theorem on subrelations of the orbit equivalence relation of $\Gamma\curvearrowright ([0,1]^\Gamma,{\rm Leb}^\Gamma).$ We learned about the connection between sparse factor of iid subsets and Chifan-Ioana's theorem from Gabor Pete and Tom Hutchcroft.  We recall the statement of \cite{ChifanIoana2010BernoulliSubequivalence}, restricted to the case of Bernoulli shift on $\Gamma$.

\begin{thm}[{\cite{ChifanIoana2010BernoulliSubequivalence}}]
Let $\cS$ be a measured sub-equivalence relation of the orbit equivalence relation $\cR$ on $([0,1]^\Gamma, {\rm Leb}^\Gamma)$. Then $[0,1]^\Gamma$ can be split into countably many sets $[0,1]^\Gamma=\bigsqcup_{i=0}^\infty  X_i$ such that $\cS|_{X_0}$ is hyperfinite, ${\rm Leb}^\Gamma(X_i)>0$ for $i\geq 1$ and all $X_i,i\geq 1$ are strongly $\cS$-ergodic. 
\end{thm}

In particular, it follows that the smooth part of the $\cS$-ergodic decomposition must be hyperfinite. This statement does not imply however, that the small $\cS$-ergodic components must be approximately hyperfinite. We are able to prove that this is indeed the case in the case when $\Gamma$ is exact. 

\begin{cor}
For any $\varepsilon>0$ and any finite $S\subset \Gamma$ there exists $\delta,k>0$ with the following property. For any measured sub-equivalence relation $\cS$ of the orbit equivalence relation $\cR$ on $[0,1]^\Gamma$ the $\cS$-ergodic components of measure less than $\delta$ are $(S,k,\varepsilon)$-hyperfinite. In particular, the number of $\cS$-ergodic components that fail to be  $(S,k,\varepsilon)$-hyperfinite is bounded by $\delta^{-1}.$
\end{cor}

{This formulation of quantitative Chifan-Ioana's theorem is marginally stronger than the property (qI) defined by Cs\'oka-Mester-Pete \cite{csokapetemester} and proved there for free groups}. The difference is that in \cite{csokapetemester} one does not require an upper bound $k$ on the sizes of connected components. For exact groups it is not difficult to show that the two versions are equivalent. For non-exact groups the difference could be substantial. Our approach also extends Chifan-Ioana's original theorem to weak factors of the Bernoulli shift. 

\begin{thm}\label{mthm-wCI}
Let $\Gamma$ be an exact group and let $(X,\mu)$ be a p.m.p. action of $\Gamma$, weakly contained in $([0,1]^\Gamma, {\rm Leb}^\Gamma)$. Let $\cS$ be a measured sub-equivalence relation of the orbit equivalence relation $\cR$ on $(X,\mu)$. Then $X$ can be split into countably many sets $X=\bigsqcup_{i=0}^\infty  X_i$ such that $\cS|_{X_0}$ is hyperfinite and all $X_i,i\geq 1$ are strongly $\cS$-ergodic. 
\end{thm}

Families of non-obvious examples actions weakly contained in the Bernoulli shift were constructed in \cite{hayes2019weak,hayes2021harmonic,hayes2021max}.
%{\color{red}The free group case of this theorem has been independently proved by Cs\'oka-Mester-Pete \cite{csokapetemester}}. 
We actually prove that the conclusions of Chifan-Ioana's theorem hold for all p.m.p. actions which are infinitesimally contained in $\Gamma$. This is Theorem \ref{thm-QCIinf}. At present we do not know any such action which is not weakly contained in the Bernoulli shift. We suspect that Theorem \ref{mthm-wCI} can fail for non-exact groups.

\subsection{Unimodular random subsets and thinnings}\label{sec-urs}
There is an unconditional statement about sparse factor of iid subsets which works for non-exact groups as well. Roughly speaking, the local statistics of sparse factor of iid subsets of $\Gamma$ are approximated by the local statistics of finite subsets of $\Gamma.$ To state it formally we need to introduce \emph{unimodular random subsets}  of countable groups and \emph{thinnings} of a measure preserving action. {Unimodular random subsets were introduced by Hutchcroft in \cite{hutchcroft_2020, hutchcroft2024percolation} under the name of locally unimodular random graphs. We opted for a different name since we do not consider the additional graph structure but are instead interested in how these sets sit in $\Gamma$.} Unimodular random subsets are closely tied to cross sections considered in \cite{bjorklund2025intersection}, since they appear naturally as return times to a cross section. We can also point out to \cite{thorisson1999point} and \cite[Example 2.7]{baccelli2021unimodular} as early definitions of unimodular random subsets of $\mathbb R^n$. Since the related notions circulated in the community for some time, we don't claim any originality in the definition or establishing the basic properties. At the same time we are not aware of any earlier work introducing unimodular random subsets through the lens of measured equivalence relations. 

Unimodular random subsets simulaneously generalize subgroups, finite subsets and connected components of a percolation on a group. While it is tempting to abbreviate unimodular random subsets as URS, this name is already taken by the uniformly recurrent subgroups \cite{glasner2015uniformly}. 

Let $\cM(\Gamma),\cM_o(\Gamma)$ denote respectively the space of non-empty subsets of $\Gamma$ and the space of subsets of $\Gamma$ containing identity. They are endowed with the standard product topology and the corresponding Borel $\sigma$-algebra. Define the \emph{re-rooting equivalence relation} $\cO$ on $\cM_o(\Gamma)$ by $E\sim \gamma^{-1}E$ for all $\gamma\in E$. This is a countable Borel equivalence relation which can also be obtained as restriction of the orbit equivalence relation of $\Gamma$ acting on $\cM(\Gamma)$ by left translations.

\begin{defn}
A random subset $E\in \cM_o(\Gamma)$ is unimodular if its distribution is an $\cO$-invariant probability measure. We use the name unimodular random subset both for a $\cO$-invariant probability measure on $\cM_o(\Gamma)$ and the random set itself. 
\end{defn}

Unimodular random subsets arise naturally as return times to a positive measure set in a measure preserving action of $\Gamma$. Indeed, given a measure preserving action $\Gamma\curvearrowright (Y,\nu)$ and a finite positive measure set $U$, the random set $\cF(Y,U):=\{\gamma\in \Gamma\mid \gamma^{-1} y\in U\}$ obtained by choosing $\nu$-random $y\in U$ is unimodular. In section \ref{sec-urs}, we show that all unimodular random subsets arise in this way. There are many other ways to construct them, for example as connected components of a $\Gamma$-invariant percolation of the Cayley graph, via sub-equivalence relations of the orbit equivalence relation of a probability measure preserving $\Gamma$-actions. We can also turn any finite subset $F\subset \Gamma$ to a unimodular random subset by taking a translate $\gamma^{-1}F$ where $\gamma\in F$ is uniform random. The sets obtained in this way will be called \emph{finite unimodular random subsets}.

A unimodular random subset $E\subset\Gamma$ is \emph{finite covolume} if it can be realized as $\cF(Y,U)$ for a p.m.p. action $(Y,\nu)$. We show in Lemma \ref{lem-model} that this fact does not depend on the choice of $(Y,\nu)$ or the set $U$. We say that the set $E$ is \emph{infinite covolume} if it is not finite covolume. 

The space of unimodular random subsets is closed under the weak-* convergence. We can therefore speak about weak-* limits of unimodular random subsets. %We can now define weak containment of a unimodular random subset in a measure preserving action and define a thinning of a measure preserving action. 

\begin{defn}
Let $(Y,\nu)$ be a measure preserving action. A unimodular random subset of $\Gamma$ is weakly contained in $Y$ if it arises as weak-* limit of the sets $\cF(Y,U)$, $U\subset Y$. If we can obtain it as a limit of  $\cF(U,Y)$ with $\nu(U)\to 0$, we call it a thinning of $Y$.
\end{defn}

Thinnings are meant to capture the statistics of very small measure subsets of the measure preserving action. We can think of them as a formal way of taking limits of a measure preserving action as we "zoom-in" on smaller and smaller subsets. If an infinite covolume unimodular random subset is weakly contained in a p.m.p. action $(X,\mu)$, it is automatically a thinning of $(X,\mu).$ %Weak containment of unimodular random subsets is strongly related to the projective containment of action. %As we show in Section \ref{sec-urs}, a measure preserving action $(Y,\nu)$ is projectively contained in $(Z,\tau)$ if and only if $\cF(Y,U)$ is weakly contained in $(Z,\tau)$ for all finite measure $U\subset Y$. 

A \emph{factor of iid thinning} will be a thinning of $([0,1]^\Gamma, {\rm Leb}^\Gamma).$ We prove:
\begin{thm}\label{mthm-fiidfinite}
Any factor of iid thinning on $\Gamma$ is a weak-* limit of finite unimodular random subsets.
\end{thm}

As a consequence of Theorem \ref{mthm-fiidfinite} we learn that the construction of Poisson zoos from \cite{pete2025nonamenable} already exhausts all possible local statistics of sparse factor of iid subsets. 
%If $\Gamma$ is exact, one can deduce Theorem \ref{mthm-approxhyp} from Theorem \ref{mthm-fiidfinite}. 
The proof of Theorem \ref{mthm-fiidfinite} is a relatively quick consequence of Theorem \ref{mthm-infcontBernouilli} and a few properties of projective and infinitesimal containment. In fact, more is true: given any p.m.p. action $(X,\mu)\xhookrightarrow{inf}\Gamma$ we prove that the thinnings of $(X,\mu)$ are weak-* limits of finite unimodular subsets. As mentioned before, all examples of such p.m.p. actions  that we currently know are weakly contained in Bernoulli shift, so as of this moment this generalization is not very exciting. This could change if some of the candidate actions listed in Section \ref{sec-questions} are indeed infinitesimally contained in $\Gamma$. 

\subsection{Conformal $\sigma$-algebra homomorphisms and the entropy support maps}
To prove Theorem \ref{mthm-infcontBernouilli} we will use conformal $\sigma$-algebra homomorphisms. 
\begin{defn}Let $(X,\mu),(Y,\nu)$ be measure preserving actions of $\Gamma.$ We will say that a map $\varphi\colon \cB(X,\mu)\to \cB(Y,\nu)$ is a \emph{conformal homomorphism} if for any $\varepsilon>0$ there is a $\delta>0$ such that for all $U,V\in\cB(X)$ with $\mu(U),\mu(V)\leq \delta$ there is $\lambda>0$ satisfying 
$$|\mu(U)-\lambda\nu(\varphi(U))|,|\mu(V)-\lambda\nu(\varphi(V))|, |(\mu(U\cap V)-\lambda\nu(\varphi(U)\cap\varphi(V))|\leq \varepsilon\max\{\mu(U),\mu(V)\}$$
\end{defn}
For small sets, these maps approximate isometric isomorphisms between $\sigma$-algebras up to measure scaling, hence the name conformal. The relation to infinitesimal containment is not surprising but not entirely trivial to prove. Let $(X,\mu),(Y,\nu)$ be measure preserving actions of $\Gamma$. Any $\Gamma$-equivariant conformal homomorphism $\varphi\colon \cB(X,\mu)\to \cB(Y,\nu)$ gives rise to infinitesimal containment $(X,\mu)\xhookrightarrow{inf}(Y,\nu)$, realized by taking $B_i=\varphi(A_i)$. In fact, having an equivariant conformal $\sigma$-algebra homomorphism implies certain uniformity in the quality of approximation which is not guaranteed by Definition \ref{def-infcont}. Indeed, we can make $\varepsilon$ uniform in $|F|$. We leave exploring the consequences of that fact for future work. 

To shorten notation, from now on we shall write $\mu$ for the measure $(\frac{1}{2}\delta_0+\frac{1}{2}\delta_2)^I$ on $\{0,1\}^I.$ In view towards Theorem \ref{mthm-infcontBernouilli}, we would like to find a $\Gamma$-equivariant conformal homomorphism $\cB([0,1]^\Gamma,{\rm Leb}^\Gamma)\to \cB(\Gamma,|\cdot|).$ This is a bit much to ask, so instead we construct a conformal homomorphism $\cE\colon \cB(\{0,1\}^\Gamma,\mu)\to \cB(\Gamma\times \{0,1\}^\Gamma\times Z, |\cdot|\times\mu\times \tau)$, where $(Z,\tau)$ is an auxiliary p.m.p. action. This yields $(\{0,1\}^\Gamma,\mu)\xhookrightarrow{inf}\Gamma\times (Z,\tau)$. Using general properties of projective and infinitesimal containment combined with Abert-Weiss theorem \cite{AbertWeiss2013BernoulliWeaklyContained} we deduce the desired containment $([0,1]^\Gamma,{\rm Leb^\Gamma})\xhookrightarrow{inf}\Gamma$.

% , namely construct a conformal homomorphism $\varphi\colon \cB([0,1]^\Gamma)\to \cB(\Gamma\times Z)$, where $(Z,\tau)$ is an auxiliary measure preserving action. This yields the infinitesimal containment $([0,1]^\Gamma, {\rm Leb}^\Gamma)\xhookrightarrow{inf}\Gamma\times (Z,\tau).$ Fortunately, we can show that $\Gamma\times (Z,\tau)\xhookrightarrow{proj}\Gamma$ so the transitivity properties of infinitesimal and projective containment give us the desired containment $([0,1]^\Gamma, {\rm Leb}^\Gamma)\xhookrightarrow{inf}\Gamma.$ 

The $\Gamma$-equivariant conformal homomorphism $\cE\colon\cB(\{0,1\}^\Gamma,\mu)\to \cB(\Gamma\times \{0,1\}^\Gamma\times Z, {|\cdot|\times \mu\times \tau})$ does the heavy lifting in this paper. We find it as a special case of conformal homomorphisms for products $\{0,1\}^I,$ where $I$ is a countable set, that are equivariant with respect to the full group of permutations of $I.$  %Our starting point are the \emph{entropy support maps}, constructed in Section \ref{sec-esm}. These are explicit conformal homomorphisms $\cE\colon \cB(\{0,1\}^I)\to \cB(I\times Z)$ where $Z$ is a certain $\sigma$-finite measure space equipped with a measure preserving action of $\rm Sym(I).$  
\begin{thm}\label{mthm-confhom}
Let $I$ be a countable set. Let $\mathfrak O$ be the space of linear orders on $I$ equipped with some ${\rm Sym}(I)$-invariant probability measure $\tau$ and let $m$ be the product of counting measure on $I$, $\mu$ on $\{0,1\}^I$ and $\rm Leb$ on $[0,\log 2].$ There is a ${\rm Sym}(I)$-equivariant conformal homomorphism $$\cE\colon \cB(\{0,1\}^I,\mu)\to \cB(\mathfrak O\times \{0,1\}^I\times I \times [0,\log 2] ,\tau\times m).$$
\end{thm}

The conformal homomorphism $\cE$ in Theorem \ref{mthm-confhom} is explicit and we call it the \emph{entropy support map}. The definition is somewhat technical, so we refer the reader to Section \ref{sec-esmaps}, where we construct $\cE$ and show that it is indeed a conformal homomorphism. Here we only give the formula in the case $|I|<\infty$. Write $\mathfrak O$ for the set of all linear orders on $I$, $\prec$ for an element of $\mathfrak O$, $\mu$ for the uniform probability measure on $\{0,1\}^I$ and $Z_i\colon \{0,1\}^I\to \{0,1\}$ for the $i$-th coordinate function. The map $\cE$ takes subsets of $\{0,1\}^I$ to subsets of $\mathfrak O \times\{0,1\}^I\times I \times [0,\log 2]$. For any $U\subset \{0,1\}^I$ we put 
$$\cE(U)=\left\{(\prec,w,i,t)\mid 0\leq t\leq 1_U(w)\frac{H(1_U|Z_j=w_j,j\prec i)-H(1_U|Z_j=w_j,j\preceq i)}{\mu(U\cap \{Z_j=w_j,j\prec i\})}\right\}.$$

Very roughly, these maps use entropy to quantify how fast we learn about small subsets of $\{0,1\}^I$ by revealing the $i$-th bit in order determined by $\prec$. The formula distributes this entropy as a mass over $U$ according to how ''sensitive'' the parts of $U$ are to flipping the $i$-th bit. The shape of the entropy support map is inspired by CP-processes defined by Furstenberg \cite{Furstenberg2008,Furstenberg2014}, local entropy averages \cite{HochmanShmerkin2012} developed to study fractal sets and crucially the paper \cite{csoka2020entropy}. In \cite{csoka2020entropy}, Cs\'oka, Harangi, and Vir\'ag considered the averaged entropy increments when we reveal bits in a random order and applied them to obtain entropy inequalities in arbitrary graphs. The entropy support maps are probably the most important contribution of the paper, but their highly technical nature prevents us from giving adequate feeling in this introduction. We hope that the reader will be tempted to look at Section \ref{sec-esmaps} where the details are fleshed out. 
\subsection{Structure of the paper}
In Section 2 we build the basic theory of projective and infinitesimal containment. We prove transitivity properties between different kinds of containment in Lemma \ref{lem-transitivity} and show that a measure preserving action of an exact group $\Gamma$ is projectively contained in $\Gamma$ if and only if it is amenable (Theorem \ref{thm-exactprojcont}). Section 3 is devoted to unimodular random subsets and thinnings. Here, we prove that any thinning of a p.m.p. action infinitesimally contained in $\Gamma$ must be a limit of finite unimodular random subsets ({Proposition \ref{prop-approxhypinf}). Approximate hyperfiniteness is an important ingredient in the proof of Theorem \ref{mthm-approxhyp} and we introduce it in Section \ref{sec-approxhyp}. In Theorem \ref{thm-pairmodelfree} we prove that the cut sets involved in the definition of approximate hyperfiniteness can be chosen deterministically if the unimodular random set comes from an essentially free action.  
In Section \ref{sec-conf} we show how conformal $\sigma$-algebra homomorphisms lead to infinitesimal containment. The entropy support maps are constructed in Section \ref{sec-esmaps}, where we also verify that they are conformal homomorphisms. The proof that the Bernoulli shift is infinitesimally contained in the action of the group on itself occupies Section \ref{sec-infcontproof}. We then deduce Theorems \ref{mthm-fiidfinite} and \ref{mthm-approxhyp}. In Section \ref{sec-QCI} we prove that for an exact group $\Gamma$ the p.m.p. action infinitesimally contained in $\Gamma$ satisfies the conclusions of Chifan-Ioana's theorem. Finally in Section \ref{sec-questions} we describe several conjectural instances of infinitesimal containment and discuss potential corollaries. 
\subsubsection{Acknowledgment}
The author thanks Miklos Abert, Tom Hutchcroft, H\'ector Jard\'on-Sanchez, Sam Mellick and Gabor Pete for valuable discussions. The author was supported by the Dioscuri program, initiated by the Max Planck Society, jointly managed with the National Science Center in Poland, and mutually funded by Polish the Ministry of Education and Science and the German Federal Ministry of Education and Research.  For the purpose of open access, the
authors have applied a CC BY public copyright licence to any author accepted manuscript arising from this submission.
\section{Basic properties of the projective and infinitesimal containment}\label{sec-containment}

In this section we will develop the basic theory of projective and infinitesimal containment. This will include transitivity properties and the relation to Zimmer amenability. We also produce some examples. %The following Lemma will not be used in the sequel. 

\begin{lem}
Suppose $(X,\mu)\xhookrightarrow{proj}(Y,\nu)$. Then $L^2(X,\mu)\xhookrightarrow{w}L^2(Y,\nu)$ in the sense of Zimmer \cite[Section 1]{BurtonKechris2020WeakContainment}
\end{lem}
\begin{proof}
The argument is completely standard so we only sketch the proof. Suppose a tuple $u_1,\ldots,u_k$ of vectors in $L^2(X,\mu)$ has the property that the matrix coefficients $\langle u_i,\gamma u_j\rangle$ can be approximated uniformly on finite subsets of $\Gamma$ by the matrix coefficients $\langle v_i,\gamma v_j\rangle$ for some tuple $v_1,\ldots,v_k\in L^2(Y,\nu).$ Then, the same is true for any tuples consisting of linear combinations of $u_1,\ldots,u_k$. The projective containment informs us that all tuples of characteristic functions of disjoint sets have this property. Any finite tuple $f_1,\ldots,f_k\in L^2(X,\mu)$ can be approximated arbitrarily well by a tuple of linear combinations of characteristic functions of some finite collection of disjoint sets, so their matrix coefficients can also be uniformly approximated on finite set of $\Gamma$ by matrix coefficients of some tuple of vectors in $L^2(Y,\nu).$

\end{proof}
%We start by rephrasing the definition of projective and infinitesimal containment using distance on projective spaces. For any vector $v\in \mathbb R^n$ we will write $[v]$ for its class in $\mathbb P(\mathbb R^n)$. Write $d$ for a fixed $\SO(n)$-invariant metric on $\mathbb P(\mathbb R^n)$. The exact choice of $d$ wont matter, but if we want to keep thing precise we can take $d([v],[w])=\frac{\|v\wedge w\|}{\|v\|\|w\|}.$

% Let $(X,\mu),(Y,\nu)$ be measure preserving actions of $\Gamma$. Then 
% \begin{enumerate}
%     \item $(X,\mu)\xhookrightarrow{proj}(Y,\nu)$ if for any $k\geq 1$ and finite $F\subset \Gamma$, the point $[(\mu(A_i\cap \gamma A_j))_{{i,j=1,\dots,k, \gamma\in F}}]$ can be approximated by $[(\nu(B_i\cap \gamma B_j))_{{i,j=1,\dots,k, \gamma\in F}}]$ for some sets $B_i,\ldots, B_k\subset Y.$
%     \item $(X,\mu)\xhookrightarrow{inf}(Y,\nu)$ if for any $\varepsilon>0, k\in\mathbb N$ and finite $F\subset \Gamma$ there is $\delta>0$ such that for any $A_1,\ldots, A_k\subset X$ with $\mu(A_i)\leq \delta$ there exist $B_1,\ldots,B_k\subset Y$ such that 
%     $$d([(\mu(A_i\cap \gamma A_j))_{{i,j=1,\dots,k, \gamma\in F}}],[(\nu(B_i\cap \gamma B_j))_{{i,j=1,\dots,k, \gamma\in F}}])\leq \varepsilon.$$
% \end{enumerate}

% In practice, we verify the projective and infinitesimal containment by finding $B_i$'s such that $\frac{\mu(A_i)}{\mu(A_j)}-\frac{\nu(B_i)}{\nu(B_j)},\frac{\mu(A_i\cap \gamma A_j)}{\mu(A_i)}-\frac{\nu(B_i\cap B_j)}{\nu(B_i)}$ tend to $0$.  
 
Weak containment says that we can approximate measures of intersections of pair of translates, but the approximation actually holds for arbitrary finite intersections. This fact is rather standard, see for example \cite{aaserud2018approximate}, where it is deduced from \cite[Prop. 10.1] {Kechris2010GlobalAspects}. We do not have the analogue of \cite[Prop. 10.1] {Kechris2010GlobalAspects} at our disposal so we give an elementary direct argument for projective or infinitesimal containment. The proof is just a tedious calculation so the reader might wish to skip it, To shorten notation write $U_E:=\bigcap_{\gamma\in E}\gamma U,$ where $U$ is a subset of some space $X$ with right $\Gamma$ action and $E\subset \Gamma$ is a finite set. Write $U_{\emptyset}:=X$.

\begin{lem}\label{lem-ProjContCond}
Let $(X,\mu),(Y,\nu)$ be measure preserving actions of $\Gamma$. The following conditions are equivalent 
\begin{enumerate}
    \item $(X,\mu)\xhookrightarrow{proj}(Y,\nu).$
    \item For any finite positive measure sets $U^1,\ldots,U^k\subset X$, finite $F\subset\Gamma$ and any $\varepsilon>0$ we can find $V^1,\ldots,V^k\subset Y$ and $\lambda>0$ such that for any $E_1,\ldots,E_k\subset F$, not all empty, we have
   $$\left|\mu(U_{E_1}^1\cap\ldots\cap U^k_{E_k})-\lambda \nu(V_{E_1}^1\cap\ldots\cap V^k_{E_k})\right|\leq \varepsilon\max_{i=1,\ldots,k}{\mu(U_i)}.$$
   \item Same as above but allow arbitrary finite measure, finite unions and intersections of translates of $U_i,i=1,\ldots,k$ or their complements $X\setminus U_i$ by $\gamma\in F$.
    \item Let $U\subset X$ be a finite measure subset, such that $\gamma U,\gamma\in\Gamma$ generate a dense algebra in the $\sigma$-algebra of $X$. For any finite $F\subset \Gamma$ and any $\varepsilon>0$ we can find a subset $V\subset Y$ and $\lambda>0$ such that 
    $$\left| \mu(U_E)-\lambda\nu(V_E)\right|\leq \varepsilon.$$
\end{enumerate}
\end{lem}

\begin{proof}
(3)$\Rightarrow$(2)$\Rightarrow$(1) is clear. Let us prove (1)$\Rightarrow$(2). Let $U^1,\ldots,U^k\subset X$ be positive finite measure sets. To shorten notation even further, write ${U^*}_{\bf E}:=U^1_{E_1}\cap \ldots \cap U^k_{E_k},$ where ${\bf E}=(E_1,\ldots, E_k)$ is a collection of finite subsets of $\Gamma$. We will need the following claim.\\
\textbf{Claim.} Let $F_1,\ldots,F_k\subset \Gamma$ be finite subsets containing $1$ and let $\kappa>0$. Suppose a collection of subsets $V_{\bf E}\subset  Y, {\bf E}=(E_i)_{i=1}^k, E_i\subset F_i$ satisfies 
$$|\mu({U^*}_{\bf E}\cap \gamma U^*_{\bf E'})-\nu(V_{\bf E}\cap\gamma V_{\bf E'})|\leq \kappa$$ for all ${\bf E}=(E_i)_{i=1}^k, {\bf E'}=(E'_i)_{i=1}^k,  E_i,E_i'\subset F$. Put $V^i:=V_{(\emptyset,\ldots,\{1\},\ldots,\emptyset)}$ where the singleton $\{1\}$ appears on $i$-th place. Define $V^*_{\bf E}:=V^1_{E_1}\cap \ldots \cap V^k_{E_k}.$ Then, 
$$|\nu(V_{\bf E}\Delta V^*_{\bf E})\leq 14 |{\bf E}|\kappa,$$ where $|{\bf E}|:=\sum_{i=1}^k|E_i|.$

We prove the claim by induction on $|{\bf E}|$. The case $|{\bf E}|=1$ corresponds to $E_i=\{1\}$ for some $i=1,\ldots,k$ and all other $E_j=\emptyset$. Here there is nothing to show as $V_{\bf E}=V^i.$ Suppose the claim holds for all $|{\bf E}|\leq m.$ Let ${\bf E}$ be a tuple of subsets with $|{\bf E}|=m+1.$ Write ${\bf E}=(E_1,\ldots, E_j'\sqcup\{\lambda\},\ldots, E_k)$ for some $1\leq j\leq k$ and put ${\bf E'}:=(E_1,\ldots, E_j',\ldots, E_k).$ Then $|{\bf E'}|=m.$ To shorten notation, write $\{\gamma\}^j:=(\emptyset,\ldots,\{\gamma\},\ldots,\emptyset)$, where the unique non-empty entry is on $j$-th position. We have 
\begin{align*}
|\mu(U^j)-\nu(V^j)|\leq&\kappa\\
|\mu({U^*}_{\{\lambda\}^j})-\nu(V_{\{\lambda\}^j})|=&|\mu(U^j)-\nu(V_{\{\lambda\}^j})|\leq \kappa\\
|\mu({U^*}_{\{\lambda\}^j}\cap \lambda U^j)-\nu({V}_{\{\lambda\}^j}\cap \lambda V^j)|=& |\mu(U^j)-\nu({V}_{\{\lambda\}^j}\cap \lambda V^j)|\leq \kappa.
\end{align*}
We deduce $\nu(V_{\{\lambda\}^j}\Delta \lambda V^j)\leq 4\kappa.$ Similarly,
\begin{align*}
|\mu({U^*}_{\bf E'})-\nu(V_{\bf E'})|\leq& \kappa\\
|\mu({U^*}_{\bf E'}\cap {U^*}_{\{\lambda\}^j})-\nu({V}_{\bf E'}\cap {V}_{\{\lambda\}^j})|=& |\mu({U^*}_{\bf E})-\nu({V}_{\bf E'}\cap {V}_{\{\lambda\}^j})|\leq \kappa\\
|\mu({U^*}_{\bf E})-\nu(V_{\bf E})|\leq&\kappa\\
|\mu({U^*}_{\bf E}\cap {U^*}_{\{\lambda\}^j})-\nu(V_{\bf E}\cap V_{\{\lambda\}^j})|=&|\mu({U^*}_{\bf E})-\nu(V_{\bf E}\cap V_{\{\lambda\}^j})|\leq \kappa\\
|\mu({U^*}_{\bf E}\cap {U^*}_{{\bf E'}})-\nu(V_{\bf E}\cap V_{{\bf E'}})|=&|\mu({U^*}_{\bf E})-\nu(V_{\bf E}\cap V_{{\bf E'}})|\leq \kappa.
\end{align*}
Together these yield $\nu(V_{\bf E}\Delta (V_{\{\lambda\}}\cap V_{\bf E'}))\leq 10\kappa$. By the inductive hypothesis we get 
\begin{align*}\nu(V_{\bf E}\Delta V^*_{\bf E'})\leq& \nu(V_{\bf E}\Delta (V_{\bf E'}\cap V_{\{\lambda\}^j}))+\nu(V_{\{\lambda\}^j}\Delta \lambda V^j)\\ +&\nu(V_{\bf E'}\Delta V^*_{\bf E'})\\
\leq& 10\kappa +4\kappa+14\kappa|{\bf E'}|=14\kappa |{\bf E}|.\end{align*} This proves the Claim. 

We can now prove the implication (1)$\Rightarrow$(2). Using the definition of projective containment we can find a family of sets $V_{\bf E}$, where $\bf E$ runs over all $k$-tuples of subsets of $F$ (not all empty) and $\lambda>0$ such that 
$$ |\mu({U^*}_{\bf E}\cap \gamma U^*_{\bf E'})-\lambda\nu(V_{\bf E}\cap\gamma V_{\bf E'})|\leq \kappa\max_{i=1,\ldots, k}\{\mu(U^i)\},$$ for any $\kappa>0.$ 
Put $V^i:=V_{\{1\}^i}$. By the Claim, $\lambda\nu(V_{\bf E}\Delta V^*_{\bf E})\leq 14\kappa |{\bf E}| \max_{i=1,\ldots, k}\{\mu(U^i)\},$ so 
$$|\mu(U^1_{E_1}\cap\ldots\cap U^k_{E_k})-\lambda\nu((V^1_{E_1}\cap\ldots\cap V^k_{E_k})|\leq (14|{\bf E}|+1)\kappa \max_{i=1,\ldots, k}\{\mu(U^i)\}.$$
Taking $\kappa$ small enough we prove (2).

The implication (2)$\Rightarrow$(3) follows from the inclusion-exclusion formula. For example \begin{align*}\mu(\gamma_1U_1\cup\gamma_2U_2\setminus\gamma_3U_3)=&\mu(\gamma_1U_1)+\mu(\gamma_2U_2) -\mu(\gamma_1U_1\cap\gamma_2U_2) \\-&\mu(\gamma_1U_1\cap \gamma_3U_3)-\mu(\gamma_2U_2\cap\gamma_3U_3)+\mu(\gamma_1U_1\cap \gamma_2U_2\cap \gamma_3U_3),\end{align*} so once we projectively approximate all finite measure, finite intersections of translates by $\gamma\in F$ we automatically projectively approximate all expressions using unions, intersections and complements. It remains to show that (4) implies (2) or (3). By (4) and the inclusion-exclusion principle, the approximations from (2) or (3) will hold as long as $U_i$ are in the algebra generated by the translates of $U$. By our assumption, this algebra is dense in the full $\sigma$-algebra of $X$, so the approximation holds for all tuples $U_1,\ldots, U_k$.
 \end{proof}

There is an analogous lemma for infinitesimal containment, with exactly the same proof. 

\begin{lem}\label{lem-InfContCond}
Let $(X,\mu),(Y,\nu)$ be measure preserving actions of $\Gamma$. The following conditions are equivalent 
\begin{enumerate}
    \item $(X,\mu)\xhookrightarrow{inf}(Y,\nu).$
    \item For any $k\in\mathbb N, \varepsilon>0$ and finite $F\subset \Gamma$ there exists $\delta>0$ with the following property. For any finite positive measure sets $U^1,\ldots,U^k\subset X$ satisfying $\mu(U_i)\leq \delta$ we can find $V^1,\ldots,V^k\subset Y$ and $\lambda>0$ such that for any $E_1,\ldots,E_k\subset F$, not all empty, we have
   $$\left|\mu(U_{E_1}^1\cap\ldots\cap U^k_{E_k})-\lambda \nu(V_{E_1}^1\cap\ldots\cap V^k_{E_k})\right|\leq \varepsilon\max_{i=1,\ldots,k}{\mu(U_i)}.$$
    \item Same as above but allow arbitrary finite measure, finite unions and intersection of translates of $U_i, X\setminus U_i, i=1,\ldots,k$ and their complements by $\gamma\in F$.
   % \item for any finite $F\subset \Gamma$ and any $\varepsilon>0$ there exists $\delta>0$ with the following property. For some positive measure set $U\subset X$ such that $\mu(U)\leq \delta$,  and we can find a subset $V\subset Y$ such that 
    %$$\left| \frac{\mu(U_E)}{\mu(U)}-\frac{\nu(V_E)}{\nu(V)}\right|\leq \varepsilon.$$
\end{enumerate}
\end{lem}
\begin{proof}
Same as for Lemma \ref{lem-ProjContCond}
\end{proof}
Weak containment implies projective containment implies infinitesimal containment. In general none of these implications can be reversed. There are several transitivity properties between different kinds of containment. 
\begin{lem}\label{lem-transitivity}
 \begin{enumerate}
     \item $(X,\mu)\xhookrightarrow{proj}(Y,\nu)\xhookrightarrow{proj}(Z,\tau)$ implies $(X,\mu)\xhookrightarrow{proj}(Z,\tau).$
     \item $(X,\mu)\xhookrightarrow{inf}(Y,\nu)\xhookrightarrow{proj}(Z,\tau)$ implies $(X,\mu)\xhookrightarrow{inf}(Z,\tau).$
     \item $(X,\mu)\xhookrightarrow{w}(Y,\nu)\xhookrightarrow{inf}(Z,\tau)$ implies $(X,\mu)\xhookrightarrow{inf}(Z,\tau).$
     \item If $(X,\mu)$ is infinite ergodic, $(Y,\nu)$ is p.m.p. and $(X,\mu)\xhookrightarrow{proj}(Y,\nu)\xhookrightarrow{inf}(Z,\tau)$ then $(X,\mu)\xhookrightarrow{proj}(Z,\tau).$
     \item $(X_1,\mu_1)\xhookrightarrow{proj}(Y_1,\nu_1)$ and $(X_2,\mu_2)\xhookrightarrow{proj}(Y_2,\nu_2)$ implies $(X_1\times X_2,\mu_1\times\mu_2)\xhookrightarrow{proj}(Y_1\times Y_2,\nu_1\times\nu_2).$
 \end{enumerate}   
\end{lem}
\begin{proof}
(1),(2),(3) are clear. For (4) we need to argue that for any $A_1,\ldots, A_k\subset X$ and $\delta>0$ the approximation $B_1,\ldots, B_k\subset Y$ can be chosen so that $\nu(B_i)\leq \delta$. Then, infinitesimal containment would guarantee existence of an approximation $C_1,\ldots,C_k\subset Z$, which would be the desired projective approximation for $A_1,\ldots, A_k.$ Let $A_1,\ldots, A_k\subset X$, let $F\subset \Gamma$ be finite and let $\delta>0.$
Choose an $F'\supset F$ such that $\mu(F'A_i)\geq \frac{2}{\delta}\mu(A_i)$ for all $i=1,\ldots, k$. This can be done because the ergodicity of the action guarantees that the union of all translates of $A_i$ covers almost all $X$, which has infinite measure. By Lemma \ref{lem-ProjContCond} (3), for any $\varepsilon>0$ we can find $B_1,\ldots,B_k\subset Y, \lambda>0$ such that 
\begin{align*}|\mu(A_i)-\lambda\nu(B_i)|\leq& \varepsilon \max_{i=1,\ldots,k}\mu(A_i)\\
|\mu(F'A_i)-\lambda\nu(F'B_i)|\leq& \varepsilon\max_{i=1,\ldots,k}\mu(A_i),
\end{align*} for all $i=1,\ldots,k.$ For $\varepsilon$ small enough we will have $\nu(B_i)\leq \delta\nu(F'B_i)\leq \delta\mu(X)=\delta$.

For (5), we observe that we can find an approximation for any finite tuple of product sets $A_i=A_i^1\times A_i^2, A_i^1\subset X_1, A_i^2\subset X_2$. These give rise to approximations for all tuples of finite unions of product sets, which in turn are dense in the $\sigma$-algebra of $X_1\times X_2$, so the approximation exists for all tuples.  
\end{proof}

\begin{cor}
Let $(X,\mu)$ be an essentially free p.m.p. action of $\Gamma$ and suppose that $(X,\mu)\xhookrightarrow{inf}(Y,\nu).$ Then $([0,1]^\Gamma,{\rm Leb}^\Gamma)\xhookrightarrow{inf}(Y,\nu).$
\end{cor}
\begin{proof}
By Abert-Weiss theorem \cite{AbertWeiss2013BernoulliWeaklyContained}, $([0,1]^\Gamma,{\rm Leb}^\Gamma)\xhookrightarrow{w}(X,\mu).$ Then, $([0,1]^\Gamma,{\rm Leb}^\Gamma)\xhookrightarrow{inf}(Y,\nu),$ by Lemma \ref{lem-transitivity} (3).
\end{proof}

% \begin{lem}\label{lem-ergodicdecompproj}
% Let $(Y,\nu)$ be a measure preserving action
% \end{lem}

\begin{lem}\label{lem-gammaprod}
Let $(Z,\tau)$ be a measure preserving action of $\Gamma$. Then, $\Gamma\xhookrightarrow{proj}\Gamma\times Z$ and $\Gamma\times Z\xhookrightarrow{proj}\Gamma.$
\end{lem}
\begin{proof}
The action on $\Gamma\times Z$ can be always untwisted, so we can assume without loss of generality that $\Gamma$ acts on $(Z,\tau)$ trivially. 
The containment $\Gamma\xhookrightarrow{proj}\Gamma\times Z$ is clear, we can model the action on $A_1,\ldots, A_k$ by taking $B_i:=A_i\times U$ where $U$ is an arbitrary fixed positive finite measure set. Now we argue the containment $\Gamma \times Z\xhookrightarrow{proj}\Gamma$. 
Let $F\subset \Gamma$ be finite and let $\varepsilon>0$. Let $A_1,\ldots,A_k$ be non-empty finite subsets of $\Gamma$ and let $U_1,\ldots,U_m$ be disjoint positive finite measure subsets of $Z$. Choose natural numbers $q_j,j=1,\ldots, m$ and $\lambda>0$ such that $|\tau(U_j)-\lambda q_j|\leq \varepsilon/2.$
For any finite set $F\subset \Gamma$ let us choose $\gamma_j^l, j=1,\ldots,m, l=1,\,\ldots,q_j$ such that $F\bigcup_{i=1}^kA_i\gamma_j^l$ are pairwise disjoint for $j=1,\ldots,m, l=1,\ldots,q_j$. Put $B_{i,j}:=\bigcup_{l=1}^{q_j} A_i\gamma_j^l.$ Then, for all $\gamma\in F$ 
$$\lambda|B_{i,j}\cap \gamma B_{i',j'}|=\begin{cases}\lambda q_j |A_i\cap \gamma A_{i'}| &\text{ if }j=j'\\ 0 & \text{ otherwise.}\end{cases}$$
Since $|\lambda q_J-\tau(U_j)|\leq \varepsilon,$ $B_{i,j}$ approximate the dynamics of $A_{i,j}$ on $F$ as $\varepsilon\to 0$.  

As we can approximate the dynamics of tuples $A_i,\times U_j$ for any finite set $F$ up to any positive error, we can also do it for all their finite unions. To finish the proof, note that the finite unions of products set of the form $A_i,\times U_j$, with $U_j$ disjoint, can approximate any finite tuple of sets in $\Gamma \times Z$, so we can indeed model the dynamics of any finite tuple in $\Gamma \times Z$ inside $\Gamma.$
\end{proof}
Using similar ideas we can show:
\begin{lem}\label{lem-projergodic}
Let $(Y,\nu)$ be a measure preserving action. Let $\rho\colon (Y,\nu)\to(Z,\tau)$ be the ergodic decomposition and let $\nu=\int_Z\nu_zd\tau(z).$ be the disintegration of measure $\nu$, so that $\nu_z$ are the ergodic components. The following conditions are equivalent.
\begin{itemize}
    \item $(Y,\nu)\xhookrightarrow{proj}\Gamma$,
    \item $(Y,\nu_z)\xhookrightarrow{proj}\Gamma$ for $\tau$-almost all $z\in Z$.
\end{itemize}
\end{lem}
\begin{proof}
$(1)\Rightarrow(2)$. Using martingale convergence theorem one can show that for $\tau$-almost all $z\in Z$ we have $(Y,\nu_z)\xhookrightarrow{proj}(Y,\nu),$ so $(Y,\nu_z)\xhookrightarrow{proj}\Gamma$, by Lemma \ref{lem-transitivity} (1). $(2)\Rightarrow(1)$ In the setting of infinite measure preserving action, only the class of measure $\tau$ is uniquely determined, so we can choose $\tau$ to be a probability measure. Let $A_1,\ldots,A_k\subset Y$ with $\nu(A_i)<\infty.$ Let $\varepsilon>0$ and $F\subset \Gamma$ finite. For $\tau$-almost all $z\in Z$, let $\Theta_z$ be set of tuples $B_1,\ldots, B_k\subset \Gamma, |B_i|<\infty$ and $\lambda>0$ satisfying $$|\nu_z(A_i\cap \gamma A_j)-\lambda|B_i\cap \gamma B_j||\leq \varepsilon/2 \text{ for all }i,j=1,\ldots,k,\gamma\in F.$$ Since $(Y,\nu_z)\xhookrightarrow{proj}\Gamma,$ for $\tau$-almost all $z\in Z$, the set $\Theta_z$ is non-empty $\tau$-almost surely. The assignment $z\mapsto \Theta_z$ is measurable, so by the measurable selection theorem we can find a measurable section $z\mapsto (B_1^z,\ldots, B_k^z,\lambda^z).$ We can find an $\mathbb N$-valued function $q\in L^1(Z,\tau)$ and $\lambda>0$ such that $\int_Z|\lambda^z-\lambda q(z)|\max_{1=1,\ldots,k}|B_i^z|d\tau(z)\leq \varepsilon/2.$ 
For each $z\in Z$ choose $\gamma^z_1.\ldots,\gamma^z_{q(z)}\in\Gamma$ such that $F\bigcup_{i=1,\ldots,k}B^z_i\gamma_j$ are pairwise disjoint for $j=1,\ldots, q(z).$ Again, this can be done measurably in $z.$ Define subsets $B_i\subset Z\times\Gamma$ by
$B_i:=\{(z,\gamma)\mid z\in Z,\gamma\in \bigcup_{j=1}^{q(z)}B_i^z\gamma_j^z\}.$
Let $\kappa$ be the product of the counting measure on $\Gamma$ and $\tau$ on $Z$. Then 
\begin{align*}|\nu(A_i\cap \gamma A_j)-\lambda\kappa(B_i\cap \gamma B_j)|=&\left|\int_Z \left(\nu_z(A_j\cap \gamma A_j)-\lambda q(z)|B_i^z\cap \gamma B_j^z|\right)d\tau(z)\right| \\ \leq& \int_Z \left|\varepsilon/2+|\lambda^z-\lambda q(z)||B_i^z\cap\gamma B_j^z|\right|d\tau(z)\leq \varepsilon.\end{align*}
We have now shown that $(Y,\nu)\xhookrightarrow{proj}\Gamma\times (Z,\tau)$. By Lemma \ref{lem-gammaprod}, $\Gamma\times (Z,\tau)\xhookrightarrow{proj}\Gamma,$ so $(Y,\nu)\xhookrightarrow{proj}\Gamma$, by Lemma \ref{lem-transitivity} (1). 
\end{proof}
\begin{lem}\label{lem-projcontextension}
Let $(X,\mu),(Y,\nu)$ be  measure preserving actions with $\mu(X)=1$. The following conditions are equivalent.
\begin{enumerate}
    \item $(Y,\nu)\xhookrightarrow{proj}\Gamma,$
    \item $(Y\times X,\nu\times \mu)\xhookrightarrow{proj}\Gamma.$
\end{enumerate}
\end{lem}
\begin{proof}
$(2)\Rightarrow(1)$ is clear. For the converse, suppose that $(Y,\mu)\xhookrightarrow{proj}\Gamma.$ By Lemma \ref{lem-transitivity} (5) and Lemma \ref{lem-gammaprod}, $(Y\times X,\nu\times\mu)\xhookrightarrow{proj}\Gamma\times (X,\mu)\xhookrightarrow{proj}\Gamma.$ By Lemma \ref{lem-transitivity} (1), $(Y\times X,\nu\times\mu)\xhookrightarrow{proj}\Gamma.$
\end{proof}
\begin{remark}
It would be tempting to state more general version of this lemma, where instead of $Y\times X$ we consider a general measure preserving extension of $Y$. This is true for exact groups by Theorem \ref{thm-exactprojcont} and Lemma \ref{lem-amenableextensions}. For non-exact groups the question appears to be more subtle.
\end{remark}
Recall that a measure class preserving action $\Gamma\curvearrowright (X,\mu)$ is \emph{amenable} if there exists a $\Gamma$-equivariant conditional expectation $\Phi\colon L^\infty(\Gamma\times Y)\to L^\infty(\Gamma)$ (see \cite{Zimmer3}). 

\begin{lem}\label{lem-amenableextensions}
Let $\rho\colon(Z,\tau)\to(Y,\nu)$ be a measure preserving extension of measure preserving $\Gamma$-actions. Then $(Z,\tau)$ is amenable if and only if $(Y,\nu)$ is.   
\end{lem}
\begin{proof}
If $(Y,\nu)$ is amenable the so is $(Z,\tau)$, by \cite[Thm 2.4]{Zimmer3}. Suppose now that $(Z,\tau)$ is amenable. Then, there exists a $\Gamma$-equivariant conditional expectation: $\Phi_Z\colon L^\infty(\Gamma\times Z)\to L^\infty(Z).$ Define $\Phi_Y\colon L^\infty(\Gamma\times Y)\to L^\infty(Y)$ by $\Phi_Y(f):=\mathbb  E(\Phi_{Z}(\tilde f)|Y)$, where $\tilde f(z):=f(\rho(z)), f\in L^\infty(Y,\nu)$ and $\mathbb E(\cdot|Y)\colon L^\infty(Y)\to L^\infty(Z)$ is the conditional expectation. Then, $\Phi_Y$ is a $\Gamma$ equivariant conditional expectation, so $(Y,\nu)$ is amenable. 
\end{proof}

\begin{prop}\label{prop-amenableproj}
If $(Y,\nu)$ is amenable then $(Y,\nu)\xhookrightarrow{proj}\Gamma.$    
\end{prop}
For the proof we will need the following Lemma, which is a direct consequence of \cite[Theorem 7.3]{miller2017edge}. We provide a short self contained proof for reader's convenience. 
\begin{lem}[{\cite[Theorem 7.3]{miller2017edge}}]\label{lem-erghyperfinite}
Let $(Y,\nu)$ be a $\sigma$-finite measure space with an ergodic measure preserving countable equivalence relation $\cR$ and let $\cE_n$ be an ascending sequence of finite equivalence relations such that $\cR=\bigcup_{n=1}^\infty \cE_n$. Let $f,g\in L^1(Y,\nu)$ be functions with non-vanishing integrals, supported on a finite measure set. Then for $\nu$-almost every $y\in Y$ 
$$\lim_{n\to\infty} \frac{\sum_{z\in [y]_{\cE_n}}f(z)}{\sum_{z\in [y]_{\cE_n}}g(z)}=\frac{\int f(y)d\nu(y)}{\int g(y)d\nu(y)}.$$
\end{lem}
\begin{proof}
Let $U$ be any finite measure set containing the supports of $f,g$. Let $\cR',\cE_n'$  be the equivalence relations $\cR,\cE_n$ restricted to $U$ and let $\cF',\cF_n'$ be the $\sigma$-algebras of respectively $\cR',\cE_n'$-invariant subsets of $U$. Note that $\cF_n\subset \cF_{n+1}$ and  $\cF'=\bigcap_{n=1}^\infty \cF'_n.$ The ergodicity of $\cR$ implies that $\cR'$ is ergodic on $U,$ hence $\cF'$ consist of null and co-null subsets of $U$. We have 
$$\mathbb E(f\mid\cF_n)(y)=\frac{\sum_{z\in [y]_{\cE_n'}}f(y)}{|[y]_{\cE_n'}|}=\frac{\sum_{z\in [y]_{\cE_n}}f(y)}{|[y]_{\cE_n'}|}.$$  
By the martingale convergence theorem \cite{klenke2020martingale} $$\lim_{n\to\infty}\frac{\sum_{z\in [y]_{\cE_n}}f(y)}{|[y]_{\cE_n'}|}=\lim_{n\to\infty}\mathbb E(f\mid \cE_n')(y)=\mathbb E(f\mid \cF')(y)=\frac{\int f(y)d\nu(y)}{\nu(U)},$$ similarly for $g$.
We get $$\lim_{n\to\infty} \frac{\sum_{z\in [y]_{\cE_n}}f(z)}{\sum_{z\in [y]_{\cE_n}}g(z)}=\frac{\int f(z)d\nu(z)}{\int g(z)d\nu(z)},$$ for $\nu$-almost all $y\in U$. We finish the proof by taking an exhausting sequence of $U$'s.
\end{proof}
\begin{proof}[Proof of Proposition \ref{prop-amenableproj}]
Let $(X,\mu)$ be any essentially free p.m.p. action. Then $(Y\times X,\nu\times \mu)$ is essentially free and amenable, by Lemma \ref{lem-amenableextensions}. By Lemma \ref{lem-projcontextension}, $(Y,\nu)\xhookrightarrow{proj}\Gamma$ if and only if $(Y\times X,\nu\times \mu)\xhookrightarrow{proj}\Gamma.$ This allows us to reduce the problem to the case when the action is essentially free. Using Lemma \ref{lem-projergodic} we can further reduce to the case where $(Y,\nu)$ is ergodic.

Let $(Y,\nu)$ be an essentially free amenable ergodic measure preserving action. Let $\cR$ be the orbit equivalence relation. By \cite{connes1981amenable}, $\cR$ is \emph{hyperfinite}, meaning that there exists a sequence of ascending sub-equivalence relations $\cE_n$ on $Y$ such that $\cR=\bigcup_{n=1}^\infty \cE_n$ and the equivalence classes of $\cE_n$ are finite almost surely. Let $y\in Y$ and put $B_i^n:=\{\gamma\in \Gamma\mid \gamma y\in A_i\cap [y]_{\cE_n}\}.$ We argue that $B_i^n$ approximate the statistics of $A_i$ as $n\to\infty$, for almost all $y$. Let $\gamma_0\in \Gamma$ We have 
\begin{align*}B_i^n\cap \gamma B_j^n=& \{\gamma\in\Gamma\mid \gamma y\in A_i\cap \gamma_0 A_j\cap [y]_{\cE_n}\cap \gamma_0[y]_{\cE_n}\}.\end{align*}
\textbf{Claim.} For any positive finite measure set $U\subset Y,\gamma_0\in\Gamma$ and almost every $y\in Y$ we have 
$$\lim_{n\to\infty}\frac{|\{\gamma\in \Gamma \mid \gamma y\in U\cap [y]_{\cE_n}\cap \gamma_0[y]_{\cE_n}\}|}{|\{\gamma\in \Gamma \mid \gamma y\in U\cap [y]_{\cE_n}\}|}=1.$$
\begin{proof}
$\cR=\bigcup_{n=1}^\infty \cE_n$, so for almost every $y\in Y$ there is a number $n(y)$ such that $\gamma_0^{-1}y\in [y]_{\cE_n}$ for all $n\geq n(y).$
Let $U_m:=\{y\in U\mid \gamma_0^{-1}y\in [y]_{\cE_n} \text{ for all }n\geq m\}.$ For all $n\geq m$, we have $|\{U_m\cap [y]_{\cE_n}\cap \gamma_0^{-1}[y]_{\cE_n}\}|=|\{U_m\cap [y]_{\cE_n}\}|.$ By Lemma \ref{lem-erghyperfinite} $$\liminf_{n\to\infty}\frac{|\{U\cap [y]_{\cE_n}\cap \gamma_0[y]_{\cE_n}\}|}{|\{U\cap [y]_{\cE_n}\}|}\geq \liminf_{n\to\infty}\frac{|\{U_m\cap [y]_{\cE_n}\cap \gamma_0[y]_{\cE_n}\}|}{|\{U\cap [y]_{\cE_n}\}|}=\lim_{n\to\infty}\frac{|\{U_m\cap [y]_{\cE_n}\}|}{|\{U\cap [y]_{\cE_n}\}|}=\frac{\nu(U_m)}{\nu(U)}.$$
Since $\cR=\bigcup_{n=0}^\infty \cE_n$, we have $\lim_{m\to\infty}\frac{\nu(U_m)}{\nu(U)}=1$ so the limit superior above is in fact equal to one. The sequence is bounded by $1$, so the limit exists and is equal to $1$. This proves the claim. 
\end{proof}
Using the claim we get  
$$\lim_{n\to \infty} \frac{|B_i^n\cap \gamma_0 B_j^n|}{|B_l^n\cap \gamma_1 B_m^n|}=\lim_{n\to\infty}\frac{|\{\gamma\in \Gamma \mid \gamma y\in A_i\cap \gamma_0A_j\cap [y]_{\cE_n}\}|}{|\{\gamma\in\Gamma\mid \gamma y\in A_l\cap \gamma_1 A_m\cap [y]_{\cE_n}\}|}=\frac{\nu(A_i\cap \gamma_0 A_j)}{\nu(A_l\cap\gamma_1 A_m)},$$ for all $i,j,l,m=1,\ldots,k$ and $\gamma_0,\gamma_1\in \Gamma.$ Of course we discard the pairs $l,m$ where intersection $A_l\cap\gamma_1 A_m$ is measure zero because then $B_l^n\cap\gamma_1 B_m^n=\emptyset$ for almost all $y$. This proves that $B_i^n$ projectively approximate the dynamics of $A_i$ as $n\to\infty$.
\end{proof}

A discrete group $\Gamma$ is \emph{exact} if it admits a topologically amenable action on a compact space \cite{Oza00} or equivalently that its reduced $C^*$-algebra $C^*(\Gamma)$ is exact. The class of exact groups is stable under many natural operations like taking subgroups, extensions, or under measure equivalence. It contains all amenable groups, all linear groups, Gromov hyperbolic groups, and groups acting properly on finite-dimensional $CAT(0)$ cube complexes and mapping class groups. It is highly non-trivial to find examples of non-exact groups, but they exist nonetheless \cite{Osajda2018}. 

We refer to \cite{BNSWW13} for the definition of a topologically amenable action. We will be using only the following property.

\begin{lem}[{\cite{BNSWW13}}]\label{lem-topamam}
Let $\Gamma\curvearrowright B$ be a topologically amenable action on a compact space $B$. Then, for any quasi-invariant probability measure $\kappa$ on $B$, the action $\Gamma\curvearrowright(B,\kappa)$ is amenable. 
\end{lem}

The following lemmas are probably well known to experts.

\begin{lem}\label{lem-amenabilitychar}
Let $\Gamma$ be an exact group with a compact topologically amenable action $B$. Let $(Y,\nu)$ be a measure preserving action. The following conditions are equivalent. 
\begin{enumerate}
    \item $(Y,\nu)$ is amenable.
    \item There exists a $\Gamma$-invariant measure $\tilde\nu$ on $Y\times B$ which projects down to $\nu$ on $Y$.
\end{enumerate}
\end{lem}
\begin{proof}
$(1)\Rightarrow(2)$  By \cite[Prop. 4.3.9]{Zimmer1984}, there exists a $\Gamma$-equivariant map $\kappa\colon Y\to {\rm Prob}(B)$. Let $\tilde\nu:=\int_Y\delta_y\times\kappa(y)d\nu(y).$ Then $\tilde\nu$ is invariant and projects down to $\nu.$

$(2)\Rightarrow(1)$ Let $f\in L^1(Y,\nu)$ be a positive function. Let $\kappa$ be the projection of $f\tilde \nu$ to $B$. Then $(B,\kappa)$ is a $\Gamma$-equivariant factor of the measure class preserving action $(Y,f\tilde{\nu}).$ By Lemma \ref{lem-topamam}, $(B,\kappa)$ is amenable, so by \cite[Thm 2.4]{Zimmer3} $(Y,f\tilde\nu)$ is also amenable. Amenability depends only on the class of the measure so $(Y,\nu)$ is amenable as well. 
\end{proof} 

\begin{lem}\label{lem-exactclosure}
Let $\Gamma$ be an exact group and let $Y$ be a locally compact space with a continuous action of $\Gamma$. The set of locally finite $\Gamma$-invariant measures $\nu$ on $Y$ such that $(Y,\nu)$ is amenable, is closed under weak-* convergence.  
\end{lem}
\begin{proof}
Let $\nu_n$ be a sequence of $\Gamma$-invariant measures on $Y$ converging weakly-* to $\nu$. Let $B$ be a compact topologically amenable action of $\Gamma$. Suppose that $(Y,\nu_n)$ is amenable for all $n$. By Lemma \ref{lem-amenabilitychar}, we find $\Gamma$-invariant measures $\tilde\nu_n$ on $Y\times B$ which project down to $\nu_n.$ The fact that $\tilde\nu_n$ project to convergent sequence means that the set $\{\tilde\nu_n\}_{n\in\mathbb N}$ is relatively compact, so we can extract a convergent subsequence. Let $\tilde\nu$ be a sub-sequential weak-* limit of $\tilde\nu_n$. Then $\tilde\nu$ in an invariant measure on $Y\times B$ which projects down to $\nu$, so $(Y,\nu)$ is amenable, by Lemma \ref{lem-amenabilitychar}.
\end{proof}

Now we are ready to prove the converse of Proposition \ref{prop-amenableproj} in the exact case.

\begin{thm}\label{thm-exactprojcont}
Let $\Gamma$ be an exact group and let $(Y,\nu)$ be a measure preserving action. Then, $(Y,\nu)$ is amenable if and only if $(Y,\nu)\xhookrightarrow{proj}\Gamma.$
\end{thm}
\begin{proof}
By Lemma \ref{lem-projergodic}, we can reduce to the case where $(Y,\nu)$ is ergodic, so let us assume ergodicity going forward. 
If the action is amenable, then it is projectively contained in $\Gamma$ by Proposition \ref{prop-amenableproj}. Suppose now that $(Y,\nu)\xhookrightarrow{proj}\Gamma.$ Let $U$ be any positive finite measure set. Consider the map $\iota_U\colon Y\to \cM(\Gamma)$ given by 
$$\iota_U(y)=\{\gamma\in \Gamma\,\mid \gamma^{-1}y\in U\}.$$ Then $\iota_U$ is $\Gamma$-equivariant ($\Gamma$ acts on its subsets on the left). Since $(Y,\nu)$ is ergodic, it is supported on non-empty subsets of $\Gamma.$

Let $\nu_U:=(\iota_U)_*(\nu).$ Let $E\subset \Gamma$ be a finite set. Put $C_{E}=\{S\in \cM(\Gamma)\mid E\subset S\}.$  Then
$\nu_U(C_E)=\nu\left(\bigcap_{\gamma\in E}\gamma^{-1}U\right)$. 
Let $B_n\subset \Gamma, \lambda_n>0$ be sequences of finite subsets of $\Gamma$ and positive reals such that 
$$\lim_{n\to\infty}\left|\nu\left(\bigcap_{\gamma\in E}\gamma^{-1}U\right)-\lambda_n\left|\bigcap_{\gamma\in E}\gamma^{-1}B_n\right|\right|=0,$$ for all finite subsets $E\subset \Gamma.$ The existence of such $B_n, \lambda_n$ is guaranteed by the Lemma \ref{lem-ProjContCond} (3). 
Define measures $\nu_n:=\lambda\sum_{\gamma\in\Gamma}\delta_{\gamma B_n}$ on $\{0,1\}^\Gamma.$ 
Then, for any finite subset $E\subset \Gamma$ 
$\nu_n(C_E)=\lambda_n \left|\bigcap_{\gamma\in E}\gamma^{-1}B_n\right|.$
The measures $\nu_n$ converge to $\nu_U$ because any measure on $\cM(\Gamma)$ is uniquely determined \footnote{Here we use the fact that $\cM(\Gamma)$ consists only of non-empty subsets.} by the measure of the sets $C_E$. Each action $(\cM(\Gamma),\nu_n)$ is (up to measure scaling) isomorphic to $\Gamma$ acting on $\Gamma/H_n$, where $H_n$ is the finite group stabilizing $B_n$. All these actions are amenable. By Lemma \ref{lem-exactclosure}, $(\cM(\Gamma_U),\nu)$ is amenable. It is a factor of $(Y,\nu),$ so $(Y,\nu)$ is amenable by \cite[Thm 2.4]{Zimmer3}. 
\end{proof}
\begin{cor}
Let $\Gamma$ be an exact group. Suppose that $(X,\mu)$ is a p.m.p. action and that $(X,\mu)\xhookrightarrow{inf}\Gamma.$ Then, any infinite measure preserving action $(Y,\nu)$ with $(Y,\nu)\xhookrightarrow{proj}(X,\mu)$ is amenable.
\end{cor}
\begin{proof}
By Lemma \ref{lem-transitivity}, $(Y,\nu)\xhookrightarrow{proj}\Gamma$, so it is amenable by Theorem \ref{thm-exactprojcont}.
\end{proof}

We end this section with several examples of projective containment.
\begin{example}
Let $\Gamma$ act on a countable set $\Xi.$ Then $\Xi\xhookrightarrow{proj}([0,1]^\Xi,{\rm Leb}^\Xi).$
\end{example}
\begin{proof}
Let $A_1,\ldots,A_k\subset \Xi.$ For any $\varepsilon>0,\xi\in\Xi$ let $V_{\xi}^{\varepsilon}:=\{w\in [0,1]^\Xi\mid w_{\xi}\in[0,\varepsilon]\}$. The sets $V_\xi^\varepsilon$ have the following properties
\begin{enumerate}
    \item ${\rm Leb}^\Xi(V_\xi^\varepsilon)=\varepsilon,$ 
    \item ${\rm Leb}^\Xi(V_{\xi_1}^{\varepsilon}\cap V_{\xi_2}^\varepsilon)=\varepsilon^2$ for $\xi_1\neq\xi_2,$ 
    \item $\gamma V_{\xi}^\varepsilon=V_{\gamma\xi}^\varepsilon.$
\end{enumerate}
Put $B_i:=\bigcup_{\xi\in A_i}V_{\xi}^\varepsilon.$ Then $|{\rm Leb}^\Xi(B_i\cap \gamma B_j)-\varepsilon|A_i\cap\gamma A_j|\leq C\varepsilon^2,$ where the constant $C$ depends only on the cardinalities of $A_1,\ldots,A_k.$ Taking $\varepsilon$ to $0$, we find that $B_i$ projectively approximate the dynamics of $A_i$.
\end{proof}
\begin{example}
Let $(X,\mu)$ be an essentially free measure preserving action. Then $\Gamma\xhookrightarrow{proj}{(X,\mu)}$
\end{example}
\begin{proof}
By \cite[II.\S 2, Lemma 1]{OrnsteinWeiss1987}, for any finite subset $F\subset \Gamma$ there exists a positive measure set $V\subset X$ such that $\gamma V$ are pairwise disjoint for $\gamma\in F.$ Put $U:=\{1\}\subset\Gamma.$ For all $E\subset F$, $|\bigcap_{\gamma\in E}\gamma U|=\mu(V)^{-1}\mu(\bigcap_{\gamma\in E}\gamma V)$. By Lemma \ref{lem-ProjContCond} (4), we deduce that $\Gamma\xhookrightarrow{proj}(X,\mu).$
\end{proof}
\begin{example}\label{ex-compactactions}
Let $G$ be a compact group with Haar probability measure $\mu$ and let $\rho\colon \Gamma\to G$ be a homomorphism with dense image. For any closed subgroup $H\subset G$ we have 
$\Gamma/\rho^{-1}(H)\xhookrightarrow{proj}(G,\mu)$.
\end{example}
\begin{proof}
Let $W_n$ be a basis of open neighborhoods of $1$ in $G$. Put $V_n:=HW_n$ and let $U=\rho^{-1}(H)\in \Gamma/\rho^{-1}(H).$ Then, $$\lim_{n\to\infty}\frac{\mu(\bigcap_{\gamma\in F}\gamma V_n)}{\mu(V_n)}=\begin{cases}1 &\text{ if } |F|=1\\ 0 & \text{otherwise}\end{cases}=\frac{|\bigcap_{\gamma\in F}\gamma U|}{|U|}.$$
By Lemma \ref{lem-ProjContCond}, $\Gamma/\rho^{-1}(H)\xhookrightarrow{proj}(G,\mu).$
\end{proof}

\section{Unimodular random subsets}

\subsection{Definition and basic properties.}
Recall that $\cM(\Gamma)$ is the space of non-empty subsets of $\Gamma$ and $\cM_o(\Gamma)$ is the space of subsets containing identity.
Let $\cO$ be the orbit equivalence relation of $\Gamma$ acting on $\cM(\Gamma)$ by left translations. We use the same letter for the restriction of the orbit equivalence relation to $\cM_o(\Gamma)$, this is exactly the re-rooting equivalence relation defined in the introduction. Unimodular random subsets are defined as $\cO$-invariant probability measures on $\cM_o(\Gamma)$. As it is often the case, we use the name unimodular random subset both for the probability distribution and the random variable. We stick to the convention that lowercase Greek letters will stand for the distribution while capitalized Latin letters are reserved for the random variable.

We will say that a unimodular subset $\kappa$ is ergodic if $(\cM_o(\Gamma),\cO,\kappa)$ is an ergodic measure preserving equivalence relation. Every unimodular random subset decomposes as a convex combination of ergodic ones. This follows from the ergodic decomposition for p.m.p. countable equivalence relations \cite[Prop. 3.2]{FeldmanMoore1977})

\begin{lem}
Let $(Y,\nu)$ be a measure preserving system. Let $U\subset Y$ be a positive finite measure subset. Let $\iota_U\colon Y\to \cM(\Gamma)\cup\{\emptyset\}$ be the map $\iota_U(y):=\{\gamma\in \Gamma\mid \gamma^{-1}y\in U\}.$ Define the measure $$\cF(Y,U):=\frac{1}{\nu(U)}(\iota_U)_*(\nu|_U).$$ Then $\cF(Y,U)$ is a unimodular random subgroup. Moreover, every unimodular random subgroup arises in this way.
\end{lem}

\begin{proof}
The map $\iota_U$ is $\Gamma$-equivariant so $(\iota_{U})_*\nu$ is a $\Gamma$-invariant, hence $\cO$-invariant measure. The preimage $\iota_U(\cM_o(\Gamma))$ is exactly $U,$ so $(\iota_U)_*(\nu|_U)$ is the restriction of $(\iota_{U})_*\nu$ to $\cM_o(\Gamma).$ Restriction of $\cO$ invariant measure is invariant under the restricted equivalence relation, which proves that $\cF(Y,U)$ is unimodular. 

Conversely, suppose that $\kappa$ is an $\cO$-invariant probability measure on $\cM_o(\Gamma).$ The set $\cM_o(\Gamma)$ is a complete section of the action of $\Gamma$ on $\cM(\Gamma)$, so by \cite[Theorem 1.14]{bjorklund2025intersection} it admits a unique extension to a $\Gamma$-invariant measure $\tilde\kappa$ on $\cM(\Gamma).$ Then $\kappa=\cF(\cM(\Gamma),\tilde\kappa).$
\end{proof}

Let $\kappa$ be a unimodular random subset. A space $(Y,\nu)$ and a set $U\subset Y$ such that $\kappa=\cF(Y,U)$ and the translates $U$ generate the full $\sigma$-algebra\footnote{If this condition fails, then $(Y,\nu)$ admits a proper factor where we can model $\kappa$.} is called a \emph{model for $\kappa.$} Let $\tilde\kappa$ be the unique extension of $\kappa$ to a $\Gamma$-invariant measure on $\cM(\Gamma).$ Then, the pair $(\cM(\Gamma),\tilde\kappa),\cM_o(\Gamma)$ is a model for $\kappa$. It turns out that up to measure scaling all models are isomorphic.

\begin{lem}\label{lem-model}
Let $(Y,\nu),U$ be a model for a unimodular random subset $\kappa.$ Then, the map $\iota_U\colon Y\to \cM(\Gamma)$ is an isomorphism, and $\tilde\kappa=\frac{1}{\nu(U)}(\iota_U)_*(\nu)$.
\end{lem}
\begin{proof}
By definition of a model, $(\iota_U)_*(\nu|_U)=(\iota_U)_*(\nu)|_{\cM_o(\Gamma)}=\nu(U)\kappa.$ Since $(\iota_U)_*(\nu)$ is a $\Gamma$-invariant extension of $\nu(U)\kappa$, it must be equal to $\nu(U)\kappa.$ To see that $\iota_U$ is an isomorphism, note that the preimage of the Borel $\sigma$-algebra of $\cM(\Gamma)$ is a $\Gamma$-invariant sub-$\sigma$-algebra of $Y$ containing $U=\iota_U^{-1}(\cM_o(\Gamma)),$ so it must be the full $\sigma$-algebra. 
\end{proof}

Now that we know models are unique, we can define several classes of unimodular random subsets. 

\begin{defn}
A unimodular random subset $\kappa$ with a model $(Y,\nu), U$ is:
\begin{enumerate}
    \item infinite covolume if $\nu(Y)=\infty.$ Otherwise we say that it is finite covolume and define the covolume of $\kappa$ as $\nu(Y)/\nu(U),$ 
    \item hyperfinite\footnote{We opt for the name hyperfinite instead of amenable to avoid confusion with amenable graphs. A unimodular random subset with the induced graph structure from the Cayley graph might be amenable as a graph but it is not hyperfinite in our sense, see \cite{KaimanovichAmenability} for the discussion of analogous debacle for graphings and equivalence relations.} if the action of $\Gamma$ on $(Y,\nu)$ is amenable,
    \item finite if $\kappa$ is supported on finite subsets of $\Gamma,$
    \item contained in an action $(Z,\tau)$ if $(Y,\nu)$ is a factor of $(Z,\tau)$ up to measure scaling,
    \item weakly contained in $(Z,\tau)$ if it is a weak-* limit of unimodular random subsets contained in $(Z,\tau),$
    \item a thinning of a p.m.p. action $(Z,\tau)$ if it is a weak-* limit of unimodular random subsets contained in $(Z,\tau)$ with covolumes tending to infinity.
\end{enumerate}
\end{defn}
We note that being contained in the action of $\Gamma$ on itself is tantamount to being finite. In particular, a unimodular random subset is weakly contained in $\Gamma$ if and only if it is a weak-* limit of finite unimodular random subsets. 
\begin{lem}\label{lem-ursweakcont}
\begin{enumerate}
\item  A unimodular random subset $\kappa$ with model $(Y,\nu),U$ is weakly contained in $(Z,\tau)$ if and only if $(Y,\nu)\xhookrightarrow{proj}(Z,\tau).$
\item A unimodular random subset $\kappa$ with model $(Y,\nu),U$ is a thinning of a p.m.p. action $(Z,\tau)$ if and only if $(Y,\nu)\xhookrightarrow{proj}(Z,\tau).$
 \end{enumerate}
\end{lem}
\begin{proof}
(1). Let $U_n$ be a sequence of subsets of $(Z,\tau),$ such that $\cF(Z,U_n)$ converges weakly-* to $\kappa$. For any finite subset $F\subset \Gamma$ let $C_F=\{E\in \cM(\Gamma)\mid F\subset E\}.$ Then 
\begin{equation}\label{eq-ursweak1}\kappa(C_F)=\frac{\nu(\bigcap_{\gamma\in F}\gamma^{-1}U)}{\nu(U)}=\lim_{n\to\infty}\frac{\tau(\bigcap_{\gamma\in F}\gamma^{-1}U_n)}{\tau(U_n)}.\end{equation} Since the translates of $U$ generate the full $\sigma$-algebra of $Y$, we can use Lemma \ref{lem-ProjContCond} (4) to deduce $(Y,\nu)\xhookrightarrow{proj}(Z,\tau).$ For the converse we just reverse the steps. Lemma \ref{lem-ProjContCond}, implies existence of a sequence $U_n\subset Z$ such that (\ref{eq-ursweak1}) holds for all finite $F\subset \Gamma$. Therefore, any subsequential weak-* limit of $\cF(Z,U_n)$ agrees with $\kappa$ on all of the sets $C_F.$ This of course means that the limit is unique and equal $\kappa.$

(2) it follows from (1) combined with Lemma \ref{lem-transitivity} (4).

\end{proof}
\begin{lem}\label{lem-urscontingamma}
Suppose $\Gamma$ is an exact group. Let $\kappa$ be a unimodular random subset of $\Gamma$. The following conditions are equivalent:
\begin{enumerate}
    \item $\kappa$ is weakly contained in $\Gamma.$
    \item $\kappa$ is hyperfinite.
\end{enumerate}
\end{lem}
\begin{proof}
Let $(Y,\nu),U$ be a model for $\kappa$.  (1)$\Rightarrow$(2). By Lemma \ref{lem-ursweakcont}, $(Y,\nu)\xhookrightarrow{proj}\Gamma,$ hence $(Y,\nu)$ is amenable by Theorem \ref{thm-exactprojcont}. (2)$\Rightarrow$(1).  $(Y,\nu)$ is amenable so $(Y,\nu)\xhookrightarrow{proj}\Gamma$, by Proposition \ref{prop-amenableproj}. By Lemma \ref{lem-ursweakcont}, this means that $\kappa$ is weakly contained in $\Gamma.$
\end{proof}
\subsection{Approximate hyperfiniteness.}\label{sec-approxhyp}
\begin{defn}
Let $S\subset\Gamma$ be a finite subset, $k,\varepsilon\geq 0$. A unimodular random subset $E$ is $(S,k,\varepsilon)$-hyperfinite if there exists a coupling $(E,C)\subset \cM_o(\Gamma)\times \cM(\Gamma)$ with the following properties:
\begin{enumerate}
    \item the distribution is preserved by the re-rooting equivalence relation $$(E,C)\sim (\gamma^{-1}E,\gamma^{-1}C),\gamma\in E,$$
    \item $\mathbb P(1\in C)\leq \varepsilon$,
    \item the connected components of the set $E\setminus C$ in the Cayley graph ${\rm Cay}(\Gamma,S)$ are finite of size at most $k$.
\end{enumerate}
\end{defn}
Intuitively, a unimodular random subset is $(S,k,\varepsilon)$-hyperfinite if it can be cut into connected pieces of size at most $k$ by removing $\varepsilon$-proportion of the set (this is the "cut-set" $C$ in the definition). 
The following theorem is due to Schramm, translated into the language of unimodular random subsets. 
\begin{thm}[\cite{schramm2011hyperfinite}]\label{thm-Schramm}
Let $E_n$ be a sequence of unimodular random subsets of $\Gamma$ converging weakly-* to a hyperfinite unimodular random subset $E$. Then, for every $S\subset \Gamma,\varepsilon>0$ there exists $k>0$ such that $E_n$ are $(S,k,\varepsilon)$-hyperfinite for $n$ large enough.
\end{thm}
The original formulation was in terms of unimodular random graphs. In order to pass to our version it is enough to observe that weak-* convergence of unimodular random subsets implies Benjamini-Schramm convergence of induced subgraphs of ${\rm Cay}(\Gamma,S)$ rooted at $1$. 

\begin{prop}\label{prop-approxhypinf}
Suppose $\Gamma$ is exact and let $(X,\mu)$ be a p.m.p. action of $\Gamma$ with $(X,\mu)\xhookrightarrow{inf}\Gamma$. Then,
\begin{enumerate}
    \item any thinning of $(X,\mu)$ is hyperfinite,
    \item for any finite subset $S\subset \Gamma, \varepsilon>0$ there exist $k\geq 0$ and $\delta>0$ such that the unimodular random subsets $\cF(X,U)$ are $(S,\varepsilon\,k)$-hyperfinite for all $U\subset X$, with $\mu(U)\leq \delta$
\end{enumerate} 
\end{prop}
\begin{proof}
Let $(Y,\nu),U$ be the model for a thinning $\kappa$ of $(X,\mu).$ By Lemma \ref{lem-ursweakcont}, $(Y,\nu)$ is an infinite measure preserving system projectively contained in $(X,\mu).$ By Lemma \ref{lem-transitivity}, $(Y,\nu)\xhookrightarrow{proj}\Gamma.$ By Proposition \ref{prop-amenableproj}, this means that $(Y,\nu)$ is amenable, hence $\kappa$ is hyperfinite. To prove the second part we note that by (1), any weak-* limit of the sets of form $\cF(X,U)$ with $\mu(U)\to 0$ is hyperfinite. By Theorem \ref{thm-Schramm}, the sets are eventually $(S,\varepsilon,k)$-hyperfinite.
\end{proof}

 The definition of $(S,k,\varepsilon)$-hyperfiniteness allows for additional randomness involved in the choice of the "cut-set" $C$ coupled with the unimodular random subset $E$ -- we don't have to select $C$ as a measurable function of $E$. This is in general unavoidable, for example when $E$ is a fixed amenable subgroup of $\Gamma$. Being able to select $C$ as a function of $E$ would make it $E$-invariant, hence either empty or equal to $E$. We will have to control this additional randomness. 
\begin{lem}\label{lem-pairmodel}
Let $(Y,\nu)$ be a measure preserving action and let $U\subset Y$ be a positive finite measure set. Suppose that $\cF(Y,U)$ is $(S,k,\varepsilon)$-hyperfinite. Then, there exists a measure preserving extension $\rho\colon(Y_1,\nu_1)\to (Y,\nu)$ and a subset $V_1\subset \rho^{-1}(U):=U_1$ such that
\begin{enumerate}
    \item $\nu_1(V_1)\leq \varepsilon \nu(U)=\varepsilon\nu_1(U_1)$,
    \item The equivalence relation $\cS$ on $U_1\setminus V_1$ generated by $s\in S$ restricted to $U_1\setminus V_1$ is finite with classes of size at most $k$. 
\end{enumerate}
\end{lem}
\begin{proof}
We only sketch the argument since it is a rather standard way of encoding additional randomness by taking a measure preserving extension (see e.g.  \cite[Prop. 13]{abert2014kesten} or \cite[Example 9.9]{AldousLyons}). 

Let $\tau_0$ be the distribution of the pair $(E,C)\in \cM_o(\Gamma)\times \cM(\Gamma)$ witnessing the $(S,k,\varepsilon)$-hyperfiniteness. It is invariant under the orbit equivalence relation of $\Gamma$ restricted to $\cM_o(\Gamma)\times \cM(\Gamma).$ By \cite[Theorem 1.14]{bjorklund2025intersection}, it extends to a unique $\Gamma$ invariant measure $\tau$ on $\cM(\Gamma)\times \cM(\Gamma)$.  Let $\tau_E, E\in \cM_o(\Gamma)$  be the conditional probability measure on $\cM(\Gamma)$ obtained by disintegrating $\tau$ along the fibers of the projection the first coordinate. For $E\in \cM_o(\Gamma)$ it is the conditional distribution of $C$ given $E$. 
Let $Y_1:=Y\times \cM(\Gamma)$ and put 
$$\nu_1:=\int_Y \delta_y\times\tau_{\iota_U(y)}d\nu(y).$$
Let $\rho\colon Y\times \cM(\Gamma)\to\cM(\Gamma)$ be the projection to the first coordinate.
Choose $V_1:=U\times \cM_o(\Gamma).$ One can easily check that $(\iota_{U_1}(y,z),\iota_{V_1}(y,z))$ for a $\nu_1$-random $(y,z)\in Y\times \cM(\Gamma)$ has the same distribution as $(E,C)$. In particular $\nu_1(V_1)=\nu(U)\mathbb P(1\in C)\leq \varepsilon\nu(U).$ Since $E\setminus C$ has finite connected components of size at most $k$, the relation $\cS$ on $U_1\setminus V_1$ has classes of size at most $k$. 
\end{proof}

It is already implicit in Schramm's proof \cite{schramm2011hyperfinite} that the cut set $C$ can be chosen as a factor of i.i.d. which means that we can take $Y_1=Y\times [0,1]^\Gamma$ in Lemma \ref{lem-pairmodel}. In fact, no additional randomness is necessary when $Y$ is essentially free.

 \begin{thm}\label{thm-pairmodelfree}
    Let $S$ be a finite symmetric subset of $\Gamma$. Let $(Y,\nu)$ be an essentially  free measure preserving action. Let $U\subset Y$ be a positive finite measure set and assume that $\cF(Y,U)$ is $(S,\varepsilon,k)$-hyperfinite, with $\varepsilon<1/2$. Put $\varepsilon':=-11|S|\varepsilon\log\varepsilon$. There exists a subset $V\subset U$ with $\nu(V)\leq \varepsilon' \nu(U)$ such that the relation $\cS$ spanned by $S$ on $U\setminus V$ has only finite classes of size at most $k$.
 \end{thm}

\begin{proof}
The proof builds on ideas of Schramm \cite{schramm2011hyperfinite} and Abert-Weiss \cite{AbertWeiss2013BernoulliWeaklyContained}.
Let $\rho\colon (Y_1,\nu_1)\to (Y,\nu)$, $U_1,V_1\in Y_1$ and $\cS$ be as in Lemma \ref{lem-pairmodel}. Normalize $\nu,\nu_1$ so that $\nu(U)=\nu_1(U_1)=1$. We recall that $\cS$ is the equivalence relation on $U_1\setminus V_1$ obtained by restricting $s\in S$  to $U_1\setminus V_1$ and its classes are finite of size at most $k.$
Let $\nu_1=\int_Y (\nu_1)_y d\nu(y)$ be the disintegration of $\nu_1$ with respect to the map $\rho.$ We need a sub-lemma.

\begin{lem}\label{lem-nicepartition}There exists a countable partition $U_1\setminus V_1=\bigsqcup_{P\in \cP} P$ such that 
\begin{enumerate}
    \item Sets $P\in \cP$ are $\cS$-invariant.
    \item For $\nu_1$-almost every $y_1\in SP\cap V_1$, $(\nu_1)_{\rho(y_1)}(P)=0.$ This is equivalent to saying that $\rho$ maps $P$ and its boundary $SP\setminus P$ into essentially disjoint sets.
    \item For all $P\in \cP$ and $\nu$-almost all $y,y'\in \rho(P)$ we have $$\frac{1}{2}\leq \frac{(\nu_1)_y(P)}{(\nu_1)_{y'}(P)}\leq 2.$$ This means that $P$ has roughly the same thickness over $\rho(P)$.
\end{enumerate}
\end{lem}
\begin{proof}Let $y_1\in U_1\setminus V_1$. Define the \emph{mask} of $y$ as $m(y):=(A,B)\in \cM_o(\Gamma)\times \cM(\Gamma)$ where 
$$A:=\{\gamma\in\Gamma\mid \gamma y_1\in [y_1]_\cS\},\quad B:=\{\gamma\in SA\setminus A\mid \gamma y\in V_1\}.$$

By the construction of $\cS$, the set $A$ contains identity, is connected and of size at most $k$. The set $B$ is always a subset boundary $SA\setminus A$ and records which elements $\gamma\in SA\setminus A$ bring us to $V_1$ (any $\gamma\in SA\setminus A$ either takes $y$ to $V_1$ or outside $U_1$). We will call such pairs \emph{admissible}. Since we have only finitely many options for $A$, there can be only finitely many options for $B$. It follows that $m(y)$ takes only finitely many values. 

For any $A\in \cM_o(\Gamma), B\subset SA\setminus A$ define the class $[A,B]=\{(\gamma^{-1}A,\gamma^{-1}B)\mid \gamma\in A\}.$ Write 
$$E_{A,B}:=\{y\in U_1\setminus V_1\mid m(y)=(A,B)\}, \quad E_{[A,B]}=\bigcup_{\gamma\in A}E_{\gamma^{-1}A,\gamma^{-1}B}.$$

Note that $E_{\gamma^{-1}A,\gamma^{-1}B}=\gamma E_{A,B}$. It follows that $[y]_{\cS}=Ay\subset E_{[A,B]}$ for any $y\in E_{[A,B]},$ so $E_{[A,B]}$ are $\cS$-invariant. The sets $E_{[A,B]}$, as $(A,B)$ varies over admissible pairs, form a partition of $U_1\setminus V_1$ which satisfies condition (1).

Let $f(y_1):=(\nu_1)_{\rho(y_1)}(E_{[m(y_1)]}).$ Define $E_{A,B}^i:=E_{A,B}\cap f^{-1}([2^{-i},2^{-i+1}))$ and $E^i_{[A,B]}:=E_{[A,B]}\cap f^{-1}([2^{-i},2^{-i+1})).$
The sets $E_{[A,B]}^i, (A,B)$ admissible and $i\in\mathbb N$, form a partition of $U_1\setminus V_1$ satisfying (1) and (3). 

We need to refine our partition one more time to arrange for (2). This step will use the fact that $\Gamma$ acts essentially freely on $(Y,\nu).$ We will use the following observation quite often so we record it here. If a subset $E\subset U_1\setminus V_1$ satisfies (3) and $W\in Y$, then $E\cap \rho^{-1}(W)$ also satisfies (3). Indeed, for all $y\in W$ $(\nu_1)_y(E)=(\nu_1)_y(E\cap \rho^{-1}(W)).$

% For any finite sets $F\subset F'\in \cM_o(\Gamma)$ let  $[F,F']:=\{(\gamma^{-1}F,\gamma^{-1}F'\}$. Define 
% \begin{align*}
% E_{F,F'}:=\{y_1\in U_1\setminus V_1\mid [y_1]_{\cS}=F y \text{ and }F'y\in U_1\}, \quad E_{[F]}:=\bigcup_{\gamma\in F}E_{\gamma^{-1}F,\gamma^{-1}F'}=F\cdot E_{F,F'}.
% \end{align*}
% The last identity follows from $\gamma E_{[F]}=E_{[\gamma^{-1}F]}$ for any $\gamma\in F.$  For any $i\in \mathbb Z$ let $E_{F]}^i=\{y_1\in E_{[F]}\mid (\nu_1)_{\rho(y_1)}(E_{[F]})\in [2^i,2^{i+1})\}$ and $E_F^i:=E_{[F]}^i\cap E_F.$
%  The properties of $\cS$ mean that $E_F$ is empty unless $F$ is a connected subset of ${\Cay}(\Gamma, S)$ of size at most $k$.

% The sets $E_{[F]}^i$ form a countable partition of $U_1\setminus V_1$. We will later show they satisfy conditions (1) and (3), but we  have to further refine this partition so that it satisfies (2). 

The action $\Gamma \curvearrowright (Y,\nu)$ is essentially free, so we can partition $Y$ into sets $R_j,j\in J$ so that $\gamma R_j, \gamma\in (\{1\}\cup S)^{2k}$ are pairwise disjoint modulo null sets. For any admissible pair $(A,B)$ let $\Gamma_{A,B}=\{\gamma\in \Gamma\mid \gamma A=A \text{ and } \gamma B=B\}.$ If $\Gamma$ has torsion, these groups can be non-trivial but they are always contained in $A$, because $1\in A$. 
For any admissible $(A,B)$ let $Y_{A,B}$ be a fundamental domain for the action of $\Gamma_{A,B}$ on $Y$. Such a fundamental domain exists because $\Gamma_{A,B}$ is finite. 
Let  
$$E_{A,B}^{i,j}:=(E_{A,B}^i\cap \rho^{-1}(R_j\cap Y_{A,B})).$$
For each class $[A,B]$ we choose a representative $(A_0,B_0)\in [A,B]$ and define
$$E_{[A,B]}^{i,j}:=A_0\cdot E_{A_0,B_0}^{i,j}.$$
For distinct triples $[A,B],i,j$ and $[A',B'],i',j'$ the sets $E_{[A,B]}^{i,j}, E_{[A',B']}^{i',j'}$ are essentially disjoint. This is clear if $[A,B]\neq [A',B']$ or $i\neq i'$. Suppose we had a pair $j\neq j'$ and $y_1 \in \gamma E_{A_0,B_0}^{i,j}\cap \gamma' E^{i,j'}_{A_0,B_0}.$ Then $\gamma^{-1}\gamma'\in \Gamma_{A_0,B_0}$ and both $\gamma^{-1}\rho(y_1), (\gamma')^{-1} \rho(y_1)\in Y_{A_0,B_0}$. This contradicts the fact that $Y_{A_0,B_0}$ is a fundamental domain for the action of $\Gamma_{A_0,B_0}.$
The translates $\gamma E_{A_0,B_0}, \gamma\in A_0$ have pairwise disjoint projections to $Y$ because $\gamma R_j,\gamma\in A$ are pairwise disjoint.
We have $$\bigsqcup_{j\in J} E_{[A,B]}^{i,j}=A\cdot (E_{A_0,B_0}^i\cap \rho^{-1}(Y_{A_0,B_0}))=A_0\cdot \Gamma_{A_0,B_0}\cdot (E_{A_0,B_0}^i\cap \rho^{-1}(Y_{A_0,B_0}))=A_0\cdot E_{A_0,B_0}^i=E_{[A,B]}^i,
$$ so the sets $E_{[A,B]}^{i,j}$ are $\cS$-invariant, satisfy (3) and form a partition of $U_1\setminus V_1.$

We are now ready to show they also satisfy (2). Observe that $(S E_{[A,B]}^{i,j}\setminus E_{[A,B]}^{i,j})\cap U_1 = (SA_0\setminus A_0)\cdot E_{A_0,B_0}^{i,j}\cap U_1$. By definition of the mask function $m(y_1)$, for any $y_1\in E_{A_0,B_0}$ 
the set of translates $\gamma y_1\in U_1$ such that $\gamma\in SA_0\setminus A_0$ is exactly $By_1$ and all of them are in $V_1$. Therefore $$(S E_{[A,B]}^{i,j}\setminus E_{[A,B]}^{i,j})\cap U_1= S E_{[A,B]}^{i,j}\cap V_1= B_0 E^{i,j}_{A_0,B_0}.$$ The properties of $R_j$ tell us that $\rho(B_0 E^{i,j}_{A_0,B_0})\subset B_0 R_j$ is disjoint from $A_0 R_j$, hence disjoint from $\rho(E^{i,j}_{[A,B]}).$ This proves (2) and concludes the proof of the lemma. 
\end{proof}

We can now return to the proof of Theorem \ref{thm-pairmodelfree}, assuming the existence of a partition $\cP$ satisfying the conditions of Lemma \ref{lem-nicepartition}.

For any $P\in \cP$ choose a $\nu_1$-random element $y_P\in P$. Let $w(P):=(\nu_1)_{\rho(y_P)}(P).$ Because of the condition (2), the weight $w(P)$ doesn't depend on the choice of $y_P$ up to multiplicative factor of order $2$.

Let $\lambda>0$. We select a random subset $\cP'_\lambda \subset \cP$ by choosing each $P\in\cP$ independently with probability $1-e^{-\lambda w(P)}.$ Define the random subsets $V_\lambda, W_\lambda\subset U$ by 
$$V_\lambda:=\bigcup_{P\in \cP'_\lambda}\rho(SP\cap V_1),\quad W_\lambda:= \bigcup_{P\in \cP'_\lambda}\rho(P).$$
Let $\cS_\lambda$ be the equivalence relation spanned by $S$ restricted to $U\setminus V_\lambda.$ 
Let $y\in W_\lambda\setminus V_\lambda.$ Then, there exists some $P\in\cP_\lambda\subset \cP$ with $y\in \rho(P).$ The set $\rho(SP\cap V)$ is in $V_\lambda$. Choose a lift $y_1\in P$ such that $\rho(y_1)=y$ and let $m(y_1)=(A_1,B_1).$ For any $\gamma\in SA_1\setminus A_1$ either $\gamma y_1\in V_1$ if $\gamma\in B_1$, in which case $\gamma y\in V_\lambda$ or $\gamma y_1\not \in U_1,$ in which case $\gamma\not \in U.$ We conclude that $[y]_{\cS_\lambda} \subset A_1$ because the projection of $SP\cap V_1$ already cuts out the connected component $A_1 y$. 

We now estimate the measures of $V_\lambda$ and $W_\lambda.$ 
For any $y\in U$, the probability that it is not in $W_\lambda$ satisfies 
$$\mathbb P(y\not\in W_\lambda)=\prod_{y\in \rho(P)}e^{-\lambda w(P)}\leq \prod_{P\in \cP} e^{-\lambda (\nu_1)_y(P)/2}=e^{-\lambda(\nu_1)_y(U_1\setminus V_1)/2}.$$
By Markov's inequality $\nu(\{y\in U\mid (\nu_1)_y(V_1)\geq 1/2\})\leq 2\varepsilon$ so $\mathbb E(\nu(U\setminus W_\lambda))\leq 2\varepsilon+e^{-\lambda/4}.$ Hence, $\mathbb E(\nu(W_\lambda))\geq 1-2\varepsilon- e^{-\lambda/4}.$
Similarly 
\begin{align*}\mathbb P(y\in V_\lambda)\leq& \sum_{y\in S\rho(P)\cap V_1}(1-e^{-\lambda w(P)})\leq 2\sum_{s\in S}\sum_{s^{-1}y\in \rho(P\cap s^{-1}V_1)}\lambda w(P)\\ \leq& 4\lambda\sum_{s\in S}\sum_{P\in \cP}(\nu_1)_{s^{-1}y}(P\cap s^{-1}V) \leq 4\lambda\sum_{s\in S}(\nu_1)_{s^{-1}y}(s^{-1}V).\end{align*}
Integrating over $U$ we get 
$$\mathbb E(\nu(V_\lambda))\leq 4\lambda \sum_{s\in S} \int_U (\nu_1)_{s^{-1}y}(s^{-1}V)d\nu\leq 4\lambda |S| \nu(V).$$
Choose $\lambda=- 4\log \varepsilon$ and finally put $V:=V_\lambda \cup (U\setminus W_\lambda).$ 
We have $\mathbb E(\nu(V))\leq -8|S| \nu(V)\log \varepsilon +3\varepsilon \leq -11|S| \varepsilon\log \varepsilon.$ By construction, the equivalence relation generated by $S$ restricted to $U\setminus V$ has finite classes of size at most $k.$ 
\end{proof}

\section{Conformal homomorphisms}\label{sec-conf}
The condition defining conformal homomorphisms directly informs us only about the pairs of subsets $U,V\in \cB(X)$. For applications, we will be interested in tuples of sets of arbitrary finite cardinality. In this section we show how to go from pairs to any finite tuples. Finally we will show how $\Gamma$-equivariant conformal homomorphisms lead to infinitesimal containment between actions.
\begin{lem}\label{lem-projmetric}
Let $d$ be a metric on $\mathbb P(\mathbb R^n)$ given by 
$$d(u,v):=\frac{\|u\wedge v\|}{\|u\|\|v\|},$$ where $\|\cdot\|$ stands for standard Euclidean norm on $\mathbb R^n$ and $\bigwedge^2\mathbb R^n.$ For any $i\neq j=1,\ldots,k$ let $u_{i,j},v_{i,j}$ be the vector obtained from $u,v$ by restricting to $i,j$-th coordinates. We have 
$$d([u],[v])\leq (n-1)\max_{i\neq j}d([u_{i,j}],[v_{i,j}]).$$
\end{lem}
\begin{proof}
\begin{align*}
\|u\wedge v\|^2=&\sum_{i<j}(u_iv_j-u_jv_i)^2=\sum_{i<j}\|u_{i,j}\wedge v_{i,j}\|^2\leq \max_{i<j}d([u_{i,j}],[v_{i,j}])^2\sum_{i<j}\|u_{i,j}\|^2\|v_{i,j}\|^2\\
\leq& \max_{i<j}d([u_{i,j}],[v_{i,j}])^2\left(\sum_{i<j}\|u_{i,j}\|^2\right) \left(\sum_{i<j}\|u_{i,j}\|^2\right)\\
=& \max_{i<j}d([u_{i,j}],[v_{i,j}])^2(n-1)^2\|u\|^2\|v\|^2.
\end{align*}
Taking square roots and dividing both sides by $\|u\|\|v\|$ we get the desired inequality. 
\end{proof}

\begin{lem}\label{lem-projnorm}
For any $n\in\mathbb N$ there is a constant $C=C(n)>0$ with the following property. For any non-zero vectors $u,v\in \mathbb R^n$ we have
$$C^{-1}d([u],[v])\leq \min_{\lambda\in\mathbb R} \frac{\|u-\lambda v\|_\infty}{\|u\|_\infty}\leq C d([u],[v]).$$
\end{lem}
\begin{proof}
Let $u=u_0+\lambda_0v,$ with $\langle u_0,v\rangle=0.$ Then $\|u-\lambda v\|\geq \|u_0\|$ with equality if and only of $\lambda=\lambda_0.$
We have $\|u\wedge v\|=\|u_0\wedge v\|=\|u_0\|\|v\|,$ so $d([u],[v])=\|u_0\|/\|u\|. $
Let $C_0$ be such that $C^{-1}\|w\|\leq \|w\|_\infty\leq C\|w\|$ for all $w\in \mathbb R^n$. Then 

$$C_0^{-2}d([u],[v])=C_0^{-2}\frac{\|u_0\|}{\|u\|}\leq \min_{\lambda\in\mathbb R}\frac{\|u-\lambda v\|_\infty}{\|u\|_\infty}\leq \frac{\|u-\lambda_0v\|_\infty}{\|u\|_\infty}\leq C_0^2\frac{\|u_0\|}{\|u\|}=C_0^2d([u],[v]).$$
\end{proof}
\begin{lem}\label{lem-confmult}
Let $(X,\mu),(Y,\nu)$ be $\sigma$-finite measured spaces. Suppose $\varphi\colon\cB(X,\mu)\to \cB(Y,\nu)$ is a conformal homomorphism. Then for any $k,\varepsilon>0$ there exists $\delta>0$ with the following property. Let $A_1,\ldots,A_k$ be a family of finite measure subsets of $X$ with $\mu(A_i)\leq \delta, i=1,\ldots,k$. Then, there is $\lambda>0$ such that
$$|\mu(A_i\cap A_j)-\lambda\nu(\varphi(A_i)\cap \varphi(A_j))|\leq \varepsilon\max_{l=1,\ldots,k}\mu(A_l),$$ for all $i,j=1,\ldots,k.$  
\end{lem}
The proof should be a straightforward exercise in metrics on $\mathbb{P}(\mathbb{R}^N)$, but since it took us some effort, we include the details. 
\begin{proof}
%Let $d$ be a metric on $\mathbb P(\mathbb R^n)$ given by 
%$$d(u,v):=\frac{\|u\wedge v\|}{\|u\|\|v\|},$$ where $\|\cdot\|$ stands for standard Euclidean norm on $\mathbb R^n$ and $\bigwedge^2\mathbb R^n.$

% First we claim that for $\delta$ small enough there is $\lambda>0$ such that 
% $$|\mu(A_i)-\lambda\nu(\varphi(A_i))|\leq \varepsilon\max_{i=1,\ldots,k}\mu(A_i).$$
% By Lemma \ref{lem-projnorms} this 
% $$d([\mu(A_i)]_{i=1,\ldots,k},[\nu(\varphi(A_i))]_{i=1,\ldots,k})\leq \varepsilon',$$ where $\varepsilon'$
To abbreviate notation we set $B_i:=\varphi(A_i).$
For any $\eta>0$ there exists $\delta>0$ such that any tuple $A_1,\ldots,A_k$ with $\mu(A_i)\leq \delta$ satisfies
$$|\mu(A_i)-\lambda_{i,j}\nu(B_i)|, |\mu(A_j)-\lambda_{i,j}\nu(B_j)|,|\mu(A_i\cap A_j)-\lambda_{i,j}\nu(B_i\cap B_j)|\leq \eta\max\{\mu(A_i),\mu(A_j)\},$$
for all $i,j=1,\ldots,k.$ Let $u=(\mu(A_i))_{i=1,\ldots,k}, v=(\nu(B_i))_{i=1,\ldots,k}$. Write $u_{i,j},v_{i,j}$ for the restrictions to $i$-th and $j$-th coordinates. 

By Lemma \ref{lem-projnorm}, $d([u_{i,j}],[v_{i,j}])\leq C\eta$ for all $i<j=1,\ldots,k$. By Lemma \ref{lem-projmetric}, $d([u],[v])\leq (k-1)C\eta.$ Using Lemma \ref{lem-projnorm} once again we get $\min_{\lambda}\|u-\lambda v\|_\infty \leq (k-1)C^2\eta \|u\|_\infty,$ which translates to 
$$|\mu(A_i)-\lambda\nu(B_i)|, |\mu(A_i)-\lambda\nu(B_i)|\leq C^2(k-1)\eta\max_{i=1,\ldots,k}\mu(A_i),$$ For some $\lambda>0.$
for the same $\lambda$ we have 
\begin{align*}|\mu(A_i\cap A_j)-\lambda\nu(B_i\cap B_j)|\leq & |\mu(A_i\cap A_j)-\lambda_{i,j}\nu(B_i\cap B_j)|+|(\lambda-\lambda_{i,j})\nu(B_i\cap B_j)|\\
\leq& |\mu(A_i\cap A_j)-\lambda_{i,j}\nu(B_i\cap B_j)|+|(\lambda-\lambda_{i,j})\nu(B_i)|\\
\leq& |\mu(A_i\cap A_j)-\lambda_{i,j}\nu(B_i\cap B_j)|+|\mu(A_i)-\lambda_{i,j}\nu(B_i)|+|\mu(A_i)-\lambda\nu(B_i)|\\
\leq& (2+C^2(k-1))\eta\max_{l=1,\ldots,k}\mu(A_l).
\end{align*}
To prove the Lemma, choose $\eta=\varepsilon/(C^2(k-1)+2).$

%Case $k=2$ is the definition of conformal homomorphism. Let $k\geq 3$ and write $\lambda_{i,j}>0$ for a real number such that
% $$|\mu(A_i)-\lambda_{i,j}\nu(\varphi(A_i))|,|\mu(A_j)-\lambda_{i,j}\nu(\varphi(A_j))|,|\mu(A_i\cap A_j)-\lambda_{i,j}\nu(\varphi(A_j)\cap\varphi(A_j))|\leq \varepsilon\max\{\mu(A_i),\mu(A_j)\}.$$
% Let $\alpha:=\max_{i=1,\ldots,k} \{\mu(A_i)\}.$ Using triangle inequality $|(\lambda_{i,j}-\lambda_{i,l})\nu(\varphi(A_i))|\leq 2\varepsilon\alpha$.  
% \begin{align*}|(\lambda_{i,j}-\lambda_{l,m})\nu(\varphi(A_l)\cap\varphi(A_m))|\leq& |(\lambda_{i,j}-\lambda_{i,l})\nu(\varphi(A_l)\cap\varphi(A_m))|+|(\lambda_{i,j}-\lambda_{i,l})\nu(\varphi(A_l)\cap\varphi(A_m)|\\
% \leq&|(\lambda_{i,j}-\lambda_{i,l})\nu(\varphi(A_l) \cap\varphi(A_m))|+|(\lambda_{i,j}-\lambda_{i,l})\nu(\varphi(A_l)\cap\varphi(A_m)|
% \end{align*}
\end{proof}

\begin{lem}
Let $(X,\mu),(Y,\nu)$ be measure preserving action of a countable group $\Gamma$. Suppose $\varphi\colon\cB(X,\mu)\to \cB(Y,\nu)$ is a $\Gamma$-equivariant conformal homomorphism. Then $(X,\mu)\xhookrightarrow{inf}(Y,\nu)$
\end{lem}
\begin{proof}
Let $k\in\mathbb N$ and let $F\subset\Gamma$ be a finite subset containing identity. Using Lemma \ref{lem-confmult} for $k|F|$ tuples of sets, we get $\delta>0$ with the following property. Let $A_1,\ldots,A_k\in \cB(X)$ and $F\subset\Gamma$ finite. Suppose $\mu(A_i)\leq \delta$ for all $i=1,\ldots,k$. Then, there is $\lambda>0$ such that 
$$|\mu(A_i\cap \gamma A_j)-\lambda\nu(\varphi(A_i)\cap \varphi(\gamma A_j))|=|\mu(A_i\cap \gamma A_j)-\lambda\nu(\varphi(A_i)\cap \gamma\varphi( A_j))|\leq \max_{l=1,\ldots,k}\mu(A_l).$$
This means that the tuple $\varphi(A_1),\ldots,\varphi(A_k)$ approximates the dynamics of $A_1,\ldots,A_k.$ Thus, $(X,\mu)\xhookrightarrow{inf}(Y,\nu).$
\end{proof}

\section{Entropy support maps}\label{sec-esmaps}
The entropy support map is a map between Borel subsets of $\{0,1\}^I$, where $I$ is a countable index set, to Borel subsets of $\mathfrak O\times \{0,1\}^I\times I \times [0,\log 2],$ where $\mathfrak O$ is the space of linear orders on $I$ 
$$\cE\colon \cB(\{0,1\}^I)\to \cB(\mathfrak O\times \{0,1\}^I\times I \times[0,\log 2]).$$
 %We do not need to choose any measure on $\mathfrak O$ to state the desired properties of $\cE$. 
It will take us some time to state the definition of this map, so we start by listing its properties. 
For any fixed order $\prec \in \mathfrak O$ write $$\cE_\prec(U):=\{(w,i,t)\in \{0,1\}^I\times I\times [0,\log 2]\mid (\prec,w,i,t)\in \cE(U)\}.$$
We endow $\{0,1\}^I$ with the product measure $\mu=(\frac{1}{2}\delta_0+\frac{1}{2}\delta_1)^{\otimes I}.$ Let $m$ be the product of $\mu$, the counting measure on $I$ and the Lebesgue measure on $[0,\log 2].$
We will write $H(A)$ for the Shannon entropy of a random variable $A$ and $h(t):=-t\log t-(1-t)\log(1-t), t\in [0,1].$ If $U\subset \{0,1\}^I$, we put $H(U):=H(1_U)=h(\mu(U)).$
%Our main result specifies the properties of the entropy support map. 
\begin{thm}\label{thm-ESM}
There exists a measurable map $\cE \colon \cB(\{0,1\}^I)\to \cB(\mathfrak O\times \{0,1\}^I\times I \times[0,\log 2])$ satisfying the following properties 
\begin{enumerate}
    \item For any measurable $U\subset \{0,1\}^I$ and any order $\prec \in \mathfrak O$ we have $m(\cE_\prec(U))=H(U).$
    \item If $U,V\subset \{0,1\}^I$ are measurable disjoint sets,  then $\cE_\prec(U),\cE_\prec(V)$ are disjoint for all $\prec\in\mathfrak O$ (which means that $\cE(U),\cE(V)$ are disjoint).
    \item %Let the symmetric group $\Sym(I)$ act on $\{0,1\}^I$ by permuting coordinates, on $\cO$ by permuting the order relation and trivially on $[0,\log 2]$. 
    The map $\cE$ is $\Sym(I)$ equivariant. More explicitly, the following diagram commutes.
    % https://q.uiver.app/#q=WzAsNCxbMCwwLCJcXG1hdGhjYWwgQihcXHswLDFcXH1eSSkiXSxbMiwwLCJcXG1hdGhjYWwgQihcXHswLDFcXH1eSVxcdGltZXMgSVxcdGltZXMgWzAsK1xcaW5mdHkpKSJdLFswLDIsIlxcbWF0aGNhbCBCKFxcezAsMVxcfV5JKSJdLFsyLDIsIlxcbWF0aGNhbCBCKFxcezAsMVxcfV5JXFx0aW1lcyBJXFx0aW1lcyBbMCwrXFxpbmZ0eSkpIl0sWzAsMSwiXFxtYXRoY2FsIEVfe1xccHJlY30iXSxbMCwyLCJcXHNpZ21hIiwyXSxbMiwzLCJcXG1hdGhjYWwgRV97XFxzaWdtYSBcXHByZWN9IiwyXSxbMSwzLCJcXHNpZ21hIl1d
\[\begin{tikzcd}[ampersand replacement=\&]
	{\mathcal B(\{0,1\}^I)} \&\& {\mathcal B(\{0,1\}^I\times I\times [0,\log 2])} \\
	\\
	{\mathcal B(\{0,1\}^I)} \&\& {\mathcal B(\{0,1\}^I\times I\times [0,\log 2])}
	\arrow["{\mathcal E_{\prec}}", from=1-1, to=1-3]
	\arrow["\sigma"', from=1-1, to=3-1]
	\arrow["\sigma\times \sigma\times \id", from=1-3, to=3-3]
	\arrow["{\mathcal E_{\sigma \prec}}"', from=3-1, to=3-3]
\end{tikzcd}\]
where $\sigma \prec$ is the order $x (\sigma\prec)y$ iff $\sigma^{-1}x\prec \sigma^{-1}y.$
\item  %$\cE_\prec\colon \cB(\{0,1\}^I)\to \cB(\{0,1\}^I\times I\times [0,\infty))$ is a conformal homomorphism.
For any $\varepsilon>0$ there exists a $\delta>0$, dependent only on $\varepsilon$, with the following property. Let $U,V\subset \{0,1\}^I$ be measurable subsets with $\mu(U),\mu(V)\leq \delta$ and let $\prec \in \mathfrak O$. Then 
$$m((\cE_\prec(U)\cup \cE_\prec(V))\Delta \cE_\prec(U\cup V))\leq \varepsilon(m(\cE_\prec(U))+m(\cE_\prec(V)).$$
\end{enumerate}
\end{thm}
The last point is by far the most difficult one to establish. Before proving Theorem \ref{thm-ESM} we explain how it implies Theorem \ref{mthm-confhom}. We restate it below for convenience. 
\begin{reptheorem}{mthm-confhom}
Let $\tau$ be any $\Sym(I)$-invariant probability measure on $\mathfrak O$. The map $\cE\colon \cB(\{0,1\}^I,\mu)\to \cB(\mathfrak O\times \{0,1\}\times I\times [0,\infty),\tau\times m)$ is a $\Sym(I)$-equivariant conformal homomorphism. 
\end{reptheorem}
\begin{proof}
The equivariance is clear. Let $\varepsilon>0$ and let $\delta>0$ be such that Theorem \ref{thm-ESM} (4) holds. 
Let $U,V\subset \{0,1\}^I$ and assume $\mu(V)\leq \mu(U)\leq\delta$. Put $\lambda=-\log\mu(U)^{-1}$ Then 
\begin{align*}
|\mu(U)-\lambda(\tau\times m)(\cE(U))|=&|\mu(U)-\log\mu(U)^{-1}h(\mu(U))|\ll \frac{\mu(U)}{|\log\mu(U)|},\\
|\mu(V)-\lambda(\tau\times m)(\cE(V))|=&|\mu(V)-\log\mu(U)^{-1}h(\mu(V))|\ll \frac{\mu(U)\log|\log\mu(U)|}{|\log\mu(U)|},\\
|\mu(U\cap V)-\lambda(\tau\times m)(\cE(U)\cap\cE(V))|\leq & |\mu(U)-\lambda(\tau\times m)(\cE(U))|+|\mu(V)-\lambda(\tau\times m)(\cE(V))|\\
+&|\mu(U\cup V)-\lambda(\tau\times m)(\cE(U)\cup\cE(V))|\\
\ll& \frac{\mu(U)\log|\log\mu(U)|}{|\log\mu(U)|}+\varepsilon\mu(U).
\end{align*}
For any $\eta>0$, choosing $\varepsilon,\delta$ small enough makes all right-hand sides of above inequalities smaller than $\eta\mu(U)$. This proves that $\cE$ is indeed a conformal homomorphism.
\end{proof}

To define the map $\cE$ we need to introduce some notations. Let $Z_i\colon \{0,1\}^I\to \{0,1\}$ be the $i$-th coordinate function. To shorten notation, for any measurable subset $U\subset \{0,1\}^I$ we will write $U=1_U$ for the characteristic function of $U$, so that way $H(U)=h(\mu(U))$ and $H(U|E)$ is the entropy of $1_U$ conditioned on the event $E$, etc. We write $\{\prec i\},\{\preceq i\}$ for respectively the set of elements of $I$ less and less or equal to $i$, $w_{\prec i}$ is the projection of $w\in \{0,1\}^I$ to $\{0,1\}^{\prec i}$ and $Z_{\prec i}$ is the projection function $Z_{\prec i}(w)=w_{\prec i}.$ We make the same definitions for $\succ i, \succeq i$.

Let $i\in I$. Let $\mu_{\prec i}, \mu_{\succeq i}$ be the product measures on $\{0,1\}^{\prec i}, \{0,1\}^{\succeq i}$, so that $\mu=\mu_{\prec i}\times \mu_{\succeq i}.$ Let $U,V\subset \{0,1\}^I$ be Borel subsets. By Fubini's theorem, for $\mu_{\prec i}$-almost all $w_0\in \{0,1\}^{\prec i}$ the measure $\mu_{\succeq i}(\{w_1\in \{0,1\}^{\succeq i}\mid w_0\times w_1\in U\})$ is well defined. 
Put $$\varepsilon_U(w,i):= \mu_{\succeq i}(\{w'\in \{0,1\}^{\succeq i}\mid w_{\prec i}\times w'\in U\}),\quad w\in \{0,1\}^I, i\in I $$ We think of $\varepsilon_U(w,i)$ as the relative density of $U$ in the cylinder specified by the coordinates of $w$ less than $i.$ We have 
$$\int_{\{0,1\}^I} \varepsilon_U(w,i)d\mu(w)=\mu(U).$$
In addition we define 
\begin{align*}\varepsilon_U^0(w,i):=&\frac{1}{2}\mu_{\succ i}(\{w'\in \{0,1\}^{\succ i}\mid w_{\prec i}\times 0_i \times  w'\in U\}),\quad w\in \{0,1\}^I, i\in I ,\\
\varepsilon_U^1(w,i):=& \frac{1}{2}\mu_{\succ i}(\{w'\in \{0,1\}^{\succ i}\mid w_{\prec i}\times 1_i \times  w'\in U\}),\quad w\in \{0,1\}^I, i\in I.
\end{align*}
We have $\varepsilon_U(w,i)=\varepsilon_U^0(w,i)+\varepsilon_U^1(w,i).$ The functions $\varepsilon_U(w,i),\varepsilon_U^0(w,i),\varepsilon_U^1(w,i)$ can be also understood in terms of conditional expectations.
$$\varepsilon_U(w,i)=\mathbb E(1_U|Z_{\prec i})(w),\quad \varepsilon_U^0(w,i)=\mathbb E(1_U|Z_{\preceq i})(w_{\prec i}\times 0_i),\quad \varepsilon_U^1(w,i)=\mathbb E(1_U|Z_{\preceq i})(w_{\prec i}\times 1_i).$$ We adopt the convention that 
$$H_{w,i}(\diamondsuit|\heartsuit):=H(\diamondsuit|\heartsuit, Z_{\prec i}=w_{\prec i})$$
We are conditioning on measure zero events here, so formally we have to define these entropies using conditional expectations. For example
$$H_{w,i}(U):=h(\varepsilon_U(w,i)),\quad H_{w,i}(U|Z_i):=\frac{1}{2}h(2\varepsilon^0_U(w,i))+\frac{1}{2} h(2\varepsilon_U^1(w,i)).$$
By Fubini's theorem these conditional entropies are well defined for almost every $w\in \{0,1\}^I$ and coincide with the usual definition when $I$ is finite. 
Put \begin{align*}\Delta_U(w,i)=& H_{w,i}(U)-H_{w,i}(U|Z_i), \quad w\in \{0,1\}^I, i\in I,\\
\Delta_{U,V}(w,i)=& H_{w,i}(U,V)-H_{w,i}(U,V|Z_i), \quad w\in \{0,1\}^I, i\in I.
\end{align*} We interpret $\Delta_U(w,i)$ as the amount of information we gain on $U$ by revealing the $i$-th coordinate, conditional on the previous coordinates being $w_{\prec i}.$ The average of $\Delta_U(w,i)$ over all $w\in \{0,1\}^I$ already appeared in \cite{csoka2020entropy} where it was used to establish entropy inequalities for factors of iid on Cayley graphs. In our construction it will be quite important to work with the non-averaged version. We are now ready to define the entropy support map:

\begin{defn}
Let $\prec$ be a linear order on $I$ and let $U\in\cB(\{0,1\}^I)$. Define 
$$\cE_{\prec}(U):=\left\{(w,i,t)\in \{0,1\}^I\times I\times [0,\log 2]\mid 0\leq t< 1_U(w)\frac{\Delta_U(w,i)}{\varepsilon_U(w,i)}\right\}. 
$$  We then put $\cE(U):=\{(\prec,w,i,t)\mid (w,i,t)\in \cE_\prec(U)\}$.
\end{defn}
Note that \begin{align*}\Delta_U(w,i)=&H_{w,i}(U)-H_{w,i}(U|Z_i)=h(\varepsilon_U(w,i))-\frac{1}{2}h(2\varepsilon_U^0(w,i))-\frac{1}{2}h(2\varepsilon_U^1(w,i))\\
=& H_{w,i}(Z_i)-H_{w,i}(Z_i|U)\leq H_{w,i}(Z_i)=\log 2.\end{align*}
This explains why we put $[0,\log 2]$ as the last factor. 

The first few lemmas will be stated in the generality of countable index set $I$ with arbitrary linear order. After we establish certain continuity properties of $\cE$ we will be able to reduce the proof of Theorem \ref{thm-ESM} for the case of finite $I$. We can quickly verify some of the properties required by Theorem \ref{thm-ESM}.
\begin{lem}\label{lem-points2and3}
The map $\cE$ satisfies conditions (2),(3) of Theorem \ref{thm-ESM}
\end{lem}
\begin{proof}
For (2) note that $\cE(U)\subset \mathfrak O\times U \times I \times [0,\log 2]$ so $\cE(U)\cap \cE(V)\subset \mathfrak O\times (U\cap V) \times I \times [0,\log 2].$ For $(3)$ it is enough to observe that the map $(U,\prec, w, i)\mapsto 1_U(w)\frac{\Delta_U(w,i)}{\varepsilon_U(w,i)}$ is constant under $\Sym(I)$ action. 
\end{proof}
The verification of points (1) and (4) is done pointwise for each order $\prec$. 
%To shorten notation we will suppress the dependence on the order when it is not relevant and simply write $\cE(U)=\cE_\prec(U)$. 
\begin{lem}\label{lem-totalmass}
$m(\cE_{\prec}(U))=H(U).$
\end{lem}
\begin{proof}
\begin{align*}
    m(\cE_{\prec}(U))=& \int_{\{0,1\}^I}\sum_{i\in I}1_U(w)\frac{\Delta_U(w,i)}{\varepsilon_U(w,i)} d\mu(w)\\
    =&\sum_{i\in I}\int_{\{0,1\}^I}\Delta_U(w,i)d\mu(w)=\sum_{i\in I} (H_{w,i}(U)-H_{w,i}(U|Z_i)). 
\end{align*}
If $I$ were finite, the last expression would be a telescoping sum, summing up to $H(U).$ In the general case, the argument is slightly more involved. Let $\delta>0$. We choose a finite subset $F\subset I$, such that $H(U|Z_i,i\in F)\leq \delta$. Then, by Lemma \ref{lem-indepcond} 
\begin{align*}
H(U)-H(U|Z_i,i\in F)&=\sum_{i\in F} (H(U|Z_j, j\in F,j\prec i)-H(U|Z_j, j\in F,j\preceq i))\\
&\leq\sum_{i\in F} (H(U|Z_j,j\prec i)-H(U|Z_j,j\preceq i))\\
&\leq\sum_{i\in I} (H(U|Z_j,j\prec i)-H(U|Z_j,j\preceq i))
\end{align*}
Taking $\delta$ to $0$ we get $H(U)\leq \sum_{i\in I} (H(U|Z_j,j\prec i)-H(U|Z_j,j\preceq i))=m(\cE_{\prec}(U)).$
On the other hand 
\begin{align*}
m(\cE_{\prec}(U))=&\sup_{F\subset I, |F|<\infty} \sum_{i\in F}(H(U|Z_j,j\prec i)-H(U|Z_j,j\preceq i))
\end{align*}
For a finite $F\subset I$ and $i\in F$ let $i^{--}$ be the largest element of $F$ less than $i$ or $-\infty$ if $i$ is the smallest element of $F$. We have 
\begin{align*} \sum_{i\in F}(H(U|Z_j,j\prec i)-H(U|Z_j,j\preceq i)) \leq& \sum_{i\in F}(H(U|Z_j, j\preceq i^{--})-H(U|Z_j,j\preceq i))\\
=&H(U)-H(U|Z_{\preceq\sup F})\leq H(U).
\end{align*} Therefore $m(\cE_{\prec}(U))\leq H(U).$
\end{proof}
\begin{lem}\label{lem-indepcond}
Let $(\Omega,\mu)$ be a standard probability space. Let $\cX,\cY,\cZ$ be independent sub-$\sigma$-algebras of the Borel $\sigma$-algebra on $\Omega$. Then, for any random variable $U$, we have 
$$H(U|\cZ)-H(U|\cX,\cZ)\leq H(U|\cY,\cZ)-H(U|\cX,\cY,\cZ).$$
\end{lem}
\begin{proof}
First, by disintegrating over $\cZ$ we can assume it is the trivial sub-algebra.
By standard martingale convergence theorems (e.g. \cite{klenke2020martingale}), we can approximate $\cX_n, \cY_n$ by finite subalgebras in such a way that 
$$\lim_{n\to\infty} H(U|\cX_n)=H(U|\cX), \lim_{n\to\infty} H(U|\cY_n)=H(U|\cY) \text{ and } \lim_{n\to\infty} H(U|\cX_n,\cY_n)=H(U|\cX,\cY).$$ Therefore, the proof is reduced to the case when $\cX,\cY$ are finite. In that case, we have 
\begin{align*}H(U)&-H(U|\cX)-H(U|\cY)+H(U|\cX,\cY)\\&=H(U)-H(U,\cX)+H(\cX)-H(U,\cY)+H(\cY)+H(U,\cX,\cY)-H(\cX,\cY)\\
&=H(U,\cX,\cY)-H(U,\cY)-H(U,\cX)+H(U)\\
&=H(\cX|U,\cY)-H(\cX|U)\leq 0.\end{align*}
\end{proof}
\begin{lem}\label{lem-continuity}
The map $\cE_{\prec}$ is continuous with respect to total variation distance on both sides. That is, if $U_n$ is a sequence of subsets of $\{0,1\}^I$ such that $\mu(U\Delta U_n)\to 0$ then $m(\cE_{\prec}(U)\Delta \cE_{\prec}(U_n))\to 0.$ 
\end{lem}
\begin{proof}
%The total variation distance between two sets is just the measure of their symmetric distance. The map $V\mapsto H(V)=m(\cE_{\prec}(V))$ is continuous. 
For any $i\in I$ consider the function 
$J_V(w,i):= 1_V(w)\frac{\Delta_V(w,i)}{\varepsilon_V(w,i)}$. Simple computation shows that $\cB(\{0,1\}^I)\ni V\mapsto J_V(\cdot,i)\in L^1(\{0,1\}^I)$ is continuous. Let $\delta>0$ and let $U,V\subset \{0,1\}^I$. Let $F\subset I$ be a finite set such that 
$$\sum_{i\in F}H(U|Z_{\prec i})-H(U|Z_{\preceq i})\geq m(\cE_{\prec}(U))-\frac{\delta}{4}$$.
Then 
\begin{align*}
m(\cE_{\prec}(U)\setminus\cE_{\prec}(V))\leq &\int_{\{0,1\}^I}\sum_{i\in F}\left|J_U(w,i)-J_V(w,i)\right|d\mu(w) +\frac{2\delta}{3}.
\end{align*} As the map $V\mapsto J_V(\cdot,i),$ is continuous, we deduce that for $V$ close enough to $U$ we will have 
$$\int_{\{0,1\}^I}\sum_{i\in F}\left|J_U(w,i)-J_V(w,i)\right|d\mu(w)\leq \delta/4,$$ so $m(\cE_{\prec}(U)\setminus  \cE_{\prec}(V))\leq \frac{\delta}{2}.$ The map $V\mapsto m(\cE_{\prec}(V))=H(V)$ is continuous, so for $V$ close enough to $U$ we will also have $|m(\cE_{\prec}(U))-m(\cE_{\prec}(V))|\leq \frac{\delta}{4}$ and $m(\cE_{\prec}(U)\Delta \cE_{\prec}(V))\leq \delta.$
\end{proof}

\begin{lem}\label{lem-DeltaU}
$$\Delta_U(w,i)=\varepsilon_U(w,i)\left(h\left(\frac{1}{2}\right)-h\left(\frac{\varepsilon_U^0(w,i)}{\varepsilon_U(w,i)}\right)\right)+(1-\varepsilon_U(w,i))\left(h\left(\frac{1}{2}\right)-h\left(\frac{\varepsilon_U^1(w,i)-\varepsilon_U^0(w,i)}{2-2\varepsilon_U(w,i)}\right)\right)$$
\end{lem}
\begin{proof}
This is proved by a straightforward but slightly unmotivated algebraic manipulation of the defining formula $\Delta_U(w,i)=h(\varepsilon_U(w,i))-\frac{1}{2}h(2\varepsilon_U^0(w,i))-\frac{1}{2}h(2\varepsilon_U^1(w,u))$. When $I$ is finite, there is a more conceptual argument.
$$\Delta_U(w,i)=H_{w,i}(U)-H_{w,i}(U|Z_i)=H_{w,i}(Z_i)-H_{w,i}(Z_i|U).$$
Now, $H_{w,i}(Z)=h(\frac{1}{2}), H_{w,i}(Z|U)=\varepsilon_U(w,i)h(\frac{\varepsilon_U^0(w,i)}{\varepsilon_U(w,i)})+(1-\varepsilon_U(w,i))h(\frac{\frac{1}{2}-\varepsilon_U^0(w,i)}{1-\varepsilon_U(w,i)})$, so we get the expression from the Lemma. 
\end{proof}
\begin{thm}\label{thm-almostcontained}
Let $\delta>0$. There exists an $\eta>0$, dependent only on $\delta$, such that for every $U\subset \{0,1\}^I$ with $\mu(U)\leq \eta$ and every $V\subset U$ we have 
$$m(\cE_{\prec}(U)\cup \cE_{\prec}(V))-m(\cE_{\prec}(U))\leq \delta m(\cE_{\prec}(U))$$.
\end{thm}
Before moving to the proof of Theorem \ref{thm-almostcontained}, we explain how it implies Theorem \ref{thm-ESM}. 

\begin{proof}[Proof of Theorem \ref{thm-ESM}]
Point (1) follows from Lemma \ref{lem-totalmass} and points (2),(3) were shown in Lemma \ref{lem-points2and3}. It remains to prove that Theorem \ref{thm-almostcontained} implies (4).

Let $U,V\subset\{0,1\}^I$ be measurable subsets satisfying $\mu(U),\mu(V)\leq \varepsilon.$  Below, the notation $f=o(g)$ will mean that $\lim_{\varepsilon\to 0}\frac{f}{g}=0.$
By (2), the sets $\cE_\prec(U\setminus V),\cE_\prec(V\setminus U),\cE_\prec(U\cap V)$ are disjoint.  
By Theorem \ref{thm-almostcontained} and Lemma \ref{lem-totalmass}, we have 
\begin{align*}
    m(\cE_\prec(U)\Delta (\cE_\prec(U\setminus V)\cup \cE_\prec(U\cap V)))=o&(m(\cE_\prec(U))) ,\\
    m(\cE_\prec(V)\Delta (\cE_\prec(V\setminus U)\cup \cE_\prec(U\cap V)))=o&(m(\cE_\prec(V))),\\
    m(\cE_\prec(U\cup V)\Delta (\cE_\prec(U\setminus V)\cup \cE(V\setminus U)\cup \cE_\prec(U\cap V)))=o&(m(\cE_\prec(U\cup V)).
\end{align*}
Therefore, by the triangle inequality, \begin{align*}m((\cE_\prec(U)\cup \cE_\prec(V))\Delta \cE_\prec(U\cup V))=&o(m(\cE_\prec(U))+m(\cE_\prec(V))+m(\cE_\prec(U\cup V))\\=&o(m(\cE_\prec(U))+m(\cE_\prec(V))).\end{align*}
\end{proof}

It remains to prove Theorem \ref{thm-almostcontained}. We need to collect several preliminary lemmas. From now on we assume that $V\subset U\subset \{0,1\}^I$ and that $I$ is finite. %The order $\prec$ will be fixed and suppressed from the notation.

% \begin{lem}\label{lem-unionformula}
% $$m(\cE(U)\cup \cE(V))=\int_{\{0,1\}^I}\sum_{i\in I}\max\left\{\Delta_U(w,i), \Delta_V(w,i)+\frac{\varepsilon_U(w,i)-\varepsilon_V(w,i)}{\varepsilon_U(w,i)}\right\}d\mu(w).$$
% \end{lem}
% \begin{proof}
We have
\begin{align*}
m(\cE_{\prec}(U)\cup\cE_{\prec}(V))=& \int_{\{0,1\}^I}\sum_{i\in I}\max\left\{ \frac{1_U(w)\Delta_U(w,i)}{\varepsilon_U(w,i)},\frac{1_V(w)\Delta_V(w,i)}{\varepsilon_V(w,i)}
\right\}d\mu(w)\\
=& \int_{\{0,1\}^I}\sum_{i\in I}\max\left\{\Delta_U(w,i),\Delta_V(w,i)+\frac{\varepsilon_U(w,i)-\varepsilon_V(w,i)}{\varepsilon_U(w,i)}\Delta_U(w,i)\right\}d\mu(w)\end{align*}
Similarly 
$$m(\cE_{\prec}(U))=\int_{\{0,1\}^I}\sum_{i\in I}\Delta_U(w,i)d\mu(w,i).$$
By combining these identities, we can write 
\begin{equation}\label{eq-differnceintegral}
m(\cE_{\prec}(U)\cup\cE_{\prec}(V))-m(\cE_{\prec}(U))=\int_{\{0,1\}^I}\sum_{i\in I}\max\left\{0,\Delta_V(w,i)-\frac{\varepsilon_V(w,i)}{\varepsilon_U(w,i)}\Delta_U(w,i)\right\}d\mu(w).
\end{equation}
Set 
\begin{align*}r(w,i):=&\max\left\{0,\Delta_V(w,i)-\frac{\varepsilon_V(w,i)}{\varepsilon_U(w,i)}\Delta_U(w,i)\right\}\\ s_1(w,i):=&\Delta_{U,V}(w,i)-\Delta_U(w,i)\end{align*}

\begin{lem}\label{lem-rsmallbigV}
\begin{enumerate}
    \item $$r(w,i)\leq \min\left\{\Delta_V(w,i), \frac{\varepsilon_U(w,i)-\varepsilon_V(w,i)}{\varepsilon_U(w,i)}\Delta_U(w,i)+s_1(w,i)\right\}$$
\end{enumerate}
\end{lem}
\begin{proof}
The inequality $r(w,i)\leq \Delta_V(w,i)$ is clear. To prove the remaining inequality we observe that 
\begin{align*}\Delta_V(w,i)-\frac{\varepsilon_V(w,i)}{\varepsilon_U(w,i)}\Delta_U(w,i)\leq& \Delta_{U,V}(w,i)-\frac{\varepsilon_V(w,i)}{\varepsilon_U(w,i)}\Delta_U(w,i)\\
=& s_1(w,i)+\frac{\varepsilon_U(w,i)-\varepsilon_V(w,i)}{\varepsilon_U(w,i)}\Delta_U(w,i).\end{align*}

\end{proof}
\begin{lem}\label{lem-s1bound}
\begin{enumerate}
\item 
$$\int_{\{0,1\}^I}\sum_{i\in I}|s_1(w,i)|d\mu(w)=H(V|U)\leq \mu(U).$$
\item \begin{align*}s_1(w,i)=&\varepsilon_U(w,i)h\left(\frac{\varepsilon_U^0(w,i)}{\varepsilon_U(w,i)}\right)-\varepsilon_V(w,i)h\left(\frac{\varepsilon_V^0(w,i)}{\varepsilon_V(w,i)}{}\right)\\ &-(\varepsilon_U(w,i)-\varepsilon_V(w,i))h\left(\frac{\varepsilon_U^0(w,i)-\varepsilon_V^0(w,i)}{\varepsilon_U(w,i)-\varepsilon_V(w,i)}\right).\end{align*}
\end{enumerate}
\end{lem}
\begin{proof}
\begin{enumerate}
    \item 
$s_1(w,i)=\Delta_{U,V}(w,i)-\Delta_U(w,i)=H_{w,i}(V|U)-H_{w,i}(V|U,Z_i)\geq 0.$ Hence 
\begin{align*}
\int_{\{0,1\}^I}\sum_{i\in I}|s_1(w,i)|d\mu(w)=&\int_{\{0,1\}^I}\sum_{i\in I}\left(H_{w,i}(V|U)-H_{w,i}(V|U,Z_i)\right)d\mu(w)\\
=& \sum_{i\in I}(H(V|U, Z_{\prec i})-H(V|U,Z_{\preceq i}))\\
=& H(V|U)=\mu(U) h\left(\frac{\mu(V)}{\mu(U)}\right)\leq \mu(U)\log 2\leq \mu(U).
\end{align*}
\item 
\begin{align*}
s_1(w,i)=& H_{w,i}(U,V)-H_{w,i}(U,V|Z_i)-H_{w,i}(U)+H_{w,i}(U|Z_i)\\ =&H_{w,i}(U,V)-H_{w,i}(U,V,Z_i)+H_{w,i}(Z_i)-H_{i,w}(U)+H_{w,i}(U,Z_i)-H_{w,i}(Z_i)\\
=&H_{w,i}(Z_i|U)-H_{w,i}(Z_i|U,V).
% s_1(w,i)=& H_{w,i}(U,V|Z_i)-H_{w,i}(U|Z_i)\\ =&H_{w,i}(U,V,Z_i)-H_{w,i}(Z_i)-H_{w,i}(U,Z_i)+H_{w,i}(Z_i)\\
% =&H_{w,i}(Z_i|U)-H_{w,i}(Z_i|U,V).
\end{align*}
The last expression evaluates to the desired formula.
\end{enumerate}
\end{proof}

\begin{lem}\label{lem-hquadbounds}
Let $\kappa\leq x,y\leq 1-\kappa.$ Then 
$$2\left(x-\frac{1}{2}\right)^2\leq \left|h(x)-h\left(\frac{1}{2}\right)\right|\leq 8\left(x-\frac{1}{2}\right)^2.$$
\end{lem}
\begin{proof}
The first order Taylor expansion yields $h(x)=h(\frac{1}{2})+h'(\frac{1}{2})(x-\frac{1}{2})+\frac{1}{2}h''(t)(x-\frac{1}{2})^2,$ for some $t\in (\frac{1}{2},x)$ so $h(x)-h(\frac{1}{2})=\frac{1}{2}h''(t)(x-\frac{1}{2})^2.$
We have $h''(t)=-\frac{1}{t(1-t)}\leq -4.$ This explains the lower bound. For the upper bound the same formula works as long as $\frac{1}{t(1-t)}\leq 8$ which corresponds to $|t-\frac{1}{2}|\leq \frac{\sqrt{2}}{4}$, which is the case for $|x-\frac{1}{2}|\leq \frac{\sqrt{2}}{4}$. For $|x-\frac{1}{2}|\geq \frac{\sqrt{2}}{4}$ we can use the crude estimate $|h(x)-h(\frac{1}{2})|\leq \log 2<1= 8(\frac{\sqrt 2}{4})^2\leq 4(x-\frac{1}{2})^2.$
% we use the second order Taylor expansion, which gives 
% $$h(x)-h(\frac{1}{2})=-2\left(x-\frac{1}{2}\right)^2+\frac{1}{6}h'''(t)\left(x-\frac{1}{2}\right)^3,$$ for some $t\in (\frac{1}{2},x).$ As $h'''(t)=\frac{1}{t^2}-\frac{1}{(1-t)^2}$, this is enough to verify the inequality for $x\in[\frac{1}{4},\frac{3}{4}]$. For $t\in [0,\frac{1}{4})\cup (\frac{3}{4},1]$ we have 
% $$|h(x)-h(\frac{1}{2})|\leq \log2 $$.
\end{proof}

\begin{lem}\label{lem-s1lowerbound}
Suppose $\kappa\leq \frac{\varepsilon_V(w,i)}{\varepsilon_U(w,i)}\leq 1-\kappa, \kappa>0.$ Then,
$$|s_1(w,i)|\geq 2\kappa\varepsilon_U(w,i)\left(\frac{\varepsilon_U^0(w,i)}{\varepsilon_U(w,i)}-\frac{\varepsilon_V^0(w,i)}{\varepsilon_V(w,i)}\right)^2.$$
\end{lem}
\begin{proof}
By Lemma \ref{lem-s1bound},
$$s_1(w,i)=\varepsilon_U(w,i)h\left(\frac{\varepsilon_U^0(w,i)}{\varepsilon_U(w,i)}\right)-\varepsilon_V(w,i)h\left(\frac{\varepsilon_V^0(w,i)}{\varepsilon_V(w,i)}{}\right) -(\varepsilon_U(w,i)-\varepsilon_V(w,i))h\left(\frac{\varepsilon_U^0(w,i)-\varepsilon_V^0(w,i)}{\varepsilon_U(w,i)-\varepsilon_V(w,i)}\right).$$
Put $p=\frac{\varepsilon_U^0(w,i)}{\varepsilon_U(w,i)}, p_0=\frac{\varepsilon_V^0(w,i)}{\varepsilon_V(w,i)}, p_1=\frac{\varepsilon_U^0(w,i)-\varepsilon_V^0(w,i)}{\varepsilon_U(w,i)-\varepsilon_V(w,i)}$ and $\alpha=\frac{\varepsilon_V(w,i)}{\varepsilon_U(w,i)}, 1-\alpha=\frac{\varepsilon_U(w,i)-\varepsilon_V(w,i)}{\varepsilon_U(w,i)}$. Then $p=\alpha p_0+(1-\alpha)p_1$ and the above formula can be rewritten as 
$$s_1(w,i)=\varepsilon_U(w,i)(h(p)-\alpha h(p_0)-(1-\alpha)h(p_1)).$$
Note the identities $\alpha(p_0-p)+(1-\alpha)(p_1-p)=0, \alpha(p_0-p_1)=p-p_1, (1-\alpha)(p_1-p_0)=p-p_1.$ By the first order Taylor expansion, there exist $t_0\in (p_0,p),t_1\in (p_1,p)$ such that 
\begin{align*}
s_1(w,i)=&\varepsilon_U(w,i)\left(h(p)-\alpha(h(p)+h(p)'(p_0-p)+\frac{h''(t_0)}{2}(p_0-p)^2)\right.\\ &\left.-(1-\alpha)(h(p)+h'(p)(p_1-p)+\frac{h''(t_1)}{2}(p_1-p)^2)\right)\\
=&\varepsilon_U(w,i)\frac{\alpha(1-\alpha)}{2}\left(-(1-\alpha)h''(t_0)-\alpha h''(t_1)\right)(p_1-p_0)^2\\
\geq & 2\varepsilon_U(w,i)\alpha(1-\alpha)(p_1-p_0)^2.
\end{align*}
For the last inequality we used $h''(t)=-\frac{1}{t(1-t)}\leq -4$ for all $t\in (0,1).$ We can now use the assumption $\kappa\leq \frac{\varepsilon_V(w,i)}{\varepsilon_U(w,i)}\leq 1-\kappa$ to get 
\begin{align*}
s_1(w,i)\geq& 2\varepsilon_U(w,i)\frac{\varepsilon_V(w,i)(\varepsilon_U(w,i)-\varepsilon_V(w,i))}{\varepsilon_U(w,i)^2}\left(\frac{\varepsilon_V^0(w,i)}{\varepsilon_V(w,i)}-\frac{\varepsilon_U^0(w,i)-\varepsilon_V^0(w,i)}{\varepsilon_U(w,i)-\varepsilon_V(w,i)}\right)^2\\
=&2 \varepsilon_U(w,i)\frac{\varepsilon_V(w,i)}{\varepsilon_U(w,i)-\varepsilon_V(w,i)}\left(\frac{\varepsilon_V^0(w,i)}{\varepsilon_V(w,i)}-\frac{\varepsilon_U^0(w,i)}{\varepsilon_U(w,i)}\right)^2\\
\geq& 2\kappa\varepsilon_U(w,i)\left(\frac{\varepsilon_V^0(w,i)}{\varepsilon_V(w,i)}-\frac{\varepsilon_U^0(w,i)}{\varepsilon_U(w,i)}\right)^2.
\end{align*}
\end{proof}
\begin{lem}\label{lem-DeltaUapprox}
Suppose $\varepsilon_U(w,i)\leq \frac{1}{4}.$ Then 
$$0\leq \Delta_U(w,i)-\varepsilon_U(w,i)\left(h\left(\frac{1}{2}\right)-h\left(\frac{\varepsilon_U^0(w,i)}{\varepsilon_U(w,i)}\right)\right)\leq 16\varepsilon_U(w,i)\Delta
_U(w,i).$$
\end{lem}
\begin{proof}
By Lemma \ref{lem-DeltaU} and Lemma \ref{lem-hquadbounds} 
\begin{align*}
\Delta_U(w,i)-&\varepsilon_U(w,i)\left(h\left(\frac{1}{2}\right)-h\left(\frac{\varepsilon_U^0(w,i)}{\varepsilon_U(w,i)}\right)\right)\\=&(1-\varepsilon_U(w,i))\left(h\left(\frac{1}{2}\right)-h\left(\frac{1}{2}+\frac{\varepsilon_U^1(w,i)-\varepsilon_U^0(w,i)}{2-2\varepsilon_U(w,i)}\right)\right)\\
\leq& 8\left(\frac{\varepsilon_U^1(w,i)-\varepsilon_U^0(w,i)}{2-2\varepsilon_U(w,i)}\right)^2\leq 8\varepsilon_U(w,i)^2\left(\frac{\varepsilon_U(w,i)-2\varepsilon_U^0(w,i)}{\varepsilon_U(w,i)}\right)^2\\
= & 32\varepsilon_U(w,i)^2\left(\frac{1}{2}-\frac{\varepsilon_U^0(w,i)}{\varepsilon_U(w,i)}\right)^2
\leq 16 \varepsilon_U(w,i)^2\left( h\left(\frac{1}{2}\right)-h\left(\frac{\varepsilon_U^0(w,i)}{\varepsilon_U(w,i)}\right)\right)\\
\leq& 16 \varepsilon_U(w,i)\Delta_U(w,i).
\end{align*}
\end{proof}

Write 
\begin{align*}
s_2(w,i):=&\varepsilon_V(w,i)\left(h\left(\frac{\varepsilon_U^0(w,i)}{\varepsilon_U(w,i)}\right)-h\left(\frac{\varepsilon_V^0(w,i)}{\varepsilon_V(w,i)}\right)\right)\\
s_3(w,i):=& \Delta_V(w,i)-\frac{\varepsilon_V(w,i)}{\varepsilon_U(w,i)}\Delta_U(w,i)-s_2(w,i).
\end{align*}
\begin{lem}\label{lem-s3bound}
$$|s_3(w,i)|\leq 16 (\varepsilon_U(w,i)\Delta_U(w,i)+\varepsilon_V(w,i)\Delta_V(w,i)).$$
\end{lem}
\begin{proof}
\begin{align*}
s_3(w,i)=& \Delta_V(w,i)-\varepsilon_V(w,i)\left(h\left(\frac{1}{2}\right)-h\left(\frac{\varepsilon_V^0(w,i)}{\varepsilon_V(w,i)}\right)\right)\\ +&\frac{\varepsilon_V(w,i)}{\varepsilon_U(w,i)}\left(\Delta_U(w,i)-\varepsilon_U(w,i)\left(h\left(\frac{1}{2}\right)-h\left(\frac{\varepsilon_U^0(w,i)}{\varepsilon_U(w,i)}\right)\right)\right)
\end{align*}
By Lemma \ref{lem-DeltaUapprox},
\begin{align*}|s_3(w,i)|\leq& 16(\varepsilon_V(w,i)\Delta_V(w,i)+\frac{\varepsilon_V(w,i)}{\varepsilon_U(w,i)}\varepsilon_U(w,i)\Delta_U(w,i))\\ \leq& 16(\varepsilon_V(w,i)\Delta_V(w,i)+\varepsilon_U(w,i)\Delta_U(w,i))\end{align*}
\end{proof}
\begin{lem}\label{lem-s2bound}
Suppose $\kappa_1\leq \frac{\varepsilon_V(w,i)}{\varepsilon_U(w,i)}\leq 1-\kappa_1, \kappa_1<\frac{1}{4}$ and let $\kappa_2\leq \frac{\kappa_1}{8}.$ Then,
$$|s_2(w,i)|\leq -8\kappa_1^{-1/2}\log \kappa_2 \left(\sqrt{\Delta_U(w,i)s_1(w,i)}+\sqrt{\Delta_V(w,i)s_1(w,i)}\right)+4h(2\kappa_1^{-1}\kappa_2)\Delta_U(w,i).$$
\end{lem}
\begin{proof}
We split the proof into three cases.

\textbf{Case } $\kappa_2\leq \frac{\varepsilon_V^0(w,i)}{\varepsilon_V(w,i)},\frac{\varepsilon_U^0(w,i)}{\varepsilon_U(w,i)}\leq 1-\kappa_2.$

By the mean value theorem 
\begin{equation}\label{eq-mvt}h\left(\frac{\varepsilon_U^0(w,i)}{\varepsilon_U(w,i)}\right)-h\left(\frac{\varepsilon_V^0(w,i)}{\varepsilon_V(w,i)}\right)=\left(\frac{\varepsilon_U^0(w,i)}{\varepsilon_U(w,i)}-\frac{\varepsilon_V^0(w,i)}{\varepsilon_V(w,i)}\right)h'(t),\end{equation} for some $t\in \left(\frac{\varepsilon_U^0(w,i)}{\varepsilon_U(w,i)},\frac{\varepsilon_V^0(w,i)}{\varepsilon_V(w,i)}\right).$ The function $h'(t)(t-\frac{1}{2})^{-1}$ is negative on $(0,1)$, increasing on $(0,\frac{1}{2})$ and decreasing on $(\frac{1}{2},1).$ Therefore, by Lemma \ref{lem-hquadbounds} 
$$|h'(t)|\leq \left|t-\frac{1}{2}\right|\frac{|h'(\kappa_2)|}{|\kappa_2-\frac{1}{2}|}\leq -2\log \kappa_2\left|t-\frac{1}{2}\right|\leq -2\log \kappa_2 \sqrt{h\left(\frac{1}{2}\right)-h(t)}.$$
Using the inequality $\varepsilon_V(w,i)\leq \varepsilon_U(w,i)$ and the first inequality of Lemma \ref{lem-DeltaUapprox}, we can write 
\begin{align*}\label{eq-s2case1}
\varepsilon_V(w,i)^{1/2}|h'(t)|\leq& -2\log\kappa_2\max\left\{\sqrt{\varepsilon_V(w,i)\left(h\left(\frac{1}{2}\right)-h\left(\frac{\varepsilon_V^0(w,i)}{\varepsilon_V(w,i)}\right)\right)}\right.\\ &,\left.\sqrt{\varepsilon_U(w,i)\left(h\left(\frac{1}{2}\right)-h\left(\frac{\varepsilon_U^0(w,i)}{\varepsilon_U(w,i)}\right)\right)}\right\}\\
\leq& -2\log \kappa_2 \max\{\Delta_V(w,i)^{1/2},\Delta_U(w,i)^{1/2}\}.
\end{align*}
By Lemma \ref{lem-s1lowerbound}, 
$$\varepsilon_V(w,i)^{1/2}\left| \frac{\varepsilon_V^0(w,i)}{\varepsilon_V(w,i)}-\frac{\varepsilon_U^0(w,i)}{\varepsilon_U(w,i)}\right|\leq 2\kappa_1^{-1/2}|s_1(w,i)|^{1/2}.$$
Multiplying the last two inequalities and using (\ref{eq-mvt}), we obtain 
\begin{align*}|s_2(w,i)|=&\left|\varepsilon_V(w,i)\left(h\left(\frac{\varepsilon_U^0(w,i)}{\varepsilon_U(w,i)}\right)-h\left(\frac{\varepsilon_V^0(w,i)}{\varepsilon_V(w,i)}\right)\right)\right|\\ \leq& -4\kappa_1^{-1/2}\log\kappa_2 \left(\sqrt{\Delta_U(w,i)s_1(w,i)}+\sqrt{\Delta_V(w,i)s_1(w,i)}\right)\end{align*}

\textbf{Case} $\frac{\varepsilon_U^0(w,i)}{\varepsilon_U(w,i)}\leq 2\kappa_2$ or $\frac{\varepsilon_U^0(w,i)}{\varepsilon_U(w,i)}\geq 1-2\kappa_2$.

We have 
$$\frac{\varepsilon^0_V(w,i)}{\varepsilon_V(w,i)}\leq \frac{\varepsilon^0_U(w,i)}{\varepsilon_V(w,i)}\leq \kappa_1^{-1}\frac{\varepsilon_U^0(w,i)}{\varepsilon_U(w,i)}\leq 2\kappa_1^{-1}\kappa_2\leq \frac{1}{2}.$$
$$\frac{\varepsilon^1_V(w,i)}{\varepsilon_V(w,i)}\leq \frac{\varepsilon^1_U(w,i)}{\varepsilon_V(w,i)}\leq \kappa_1^{-1}\frac{\varepsilon_U^1(w,i)}{\varepsilon_U(w,i)}\leq 2\kappa_1^{-1}\kappa_2\leq \frac{1}{2},\text{ so } \frac{\varepsilon^0_V(w,i)}{\varepsilon_V(w,i)}\geq 1-2\kappa_1^{-1}\kappa_2\geq \frac{1}{2}.$$
Therefore we have a trivial estimate 
$$\left|h\left(\frac{\varepsilon_U^0(w,i)}{\varepsilon_U(w,i)}\right)-h\left(\frac{\varepsilon_V^0(w,i)}{\varepsilon_V(w,i)}\right)\right|\leq h(2\kappa_1^{-1}\kappa_2).$$
At the same time $$\Delta_U(w,i)\geq \varepsilon_U(w,i)(h(\frac{1}{2})-h(2\kappa_2))\geq \varepsilon_U(w,i)(h(1/2)-h(1/8))\geq \frac{1}{4}\varepsilon_U(w,i)\geq \frac{1}{4}\varepsilon_V(w,i).$$ Putting the inequalities together, we get 
$$|s_2(w,i)|=\varepsilon_V(w,i)\left|h\left(\frac{\varepsilon_U^0(w,i)}{\varepsilon_U(w,i)}\right)-h\left(\frac{\varepsilon_V^0(w,i)}{\varepsilon_V(w,i)}\right)\right|\leq 4 h(2\kappa_1^{-1}\kappa_2)\Delta_U(w,i)$$

\textbf{Case} $\frac{\varepsilon_V^0(w,i)}{\varepsilon_V(w,i)}\leq \kappa_2$ or $\frac{\varepsilon_V^0(w,i)}{\varepsilon_V(w,i)}\geq 1-\kappa_2$ and $2\kappa_2\leq \frac{\varepsilon_U^0(w,i)}{\varepsilon_U(w,i)}\leq 1-2\kappa_2.$

In this case 
$$\Delta_V(w,i)\geq \varepsilon_V(w,i)(h(\frac{1}{2})-h(\kappa_2))\geq \frac{1}{4}\varepsilon_V(w,i)$$ and 
$$\left|h\left(\frac{\varepsilon_U^0(w,i)}{\varepsilon_U(w,i)}\right)-h\left(\frac{\varepsilon_V^0(w,i)}{\varepsilon_V(w,i)}\right)\right|\leq h(2\kappa_2)\leq \frac{h(2\kappa_2)}{\kappa_2}\left|\frac{\varepsilon_U^0(w,i)}{\varepsilon_U(w,i)}-\frac{\varepsilon_V^0(w,i)}{\varepsilon_V(w,i)}\right|.$$
By Lemma \ref{lem-s1lowerbound}
$$\varepsilon_V(w,i)^{1/2}\left|h\left(\frac{\varepsilon_U^0(w,i)}{\varepsilon_U(w,i)}\right)-h\left(\frac{\varepsilon_V^0(w,i)}{\varepsilon_V(w,i)}\right)\right|\leq -4\kappa_1^{-1/2}\log\kappa_2\sqrt{s_1(w,i)}.$$ Combined with the lower bound on $\Delta_V(w,i)$ we get 
$$|s_2(w,i)|=\varepsilon_V(w,i)\left|h\left(\frac{\varepsilon_U^0(w,i)}{\varepsilon_U(w,i)}\right)-h\left(\frac{\varepsilon_V^0(w,i)}{\varepsilon_V(w,i)}\right)\right|\leq -8\kappa_1^{-1/2}\log\kappa_2\sqrt{s_1(w,i)\Delta_V(w,i)}.$$

The upper bound in the lemma is at least as large as the upper bound in all three cases. 
\end{proof}

We are now ready to prove Theorem \ref{thm-almostcontained}.
\begin{proof}[Proof of Theorem \ref{thm-almostcontained}]
First, using Lemma \ref{lem-continuity}, we reduce the problem to a finite index set $I$. We can therefore use all the Lemmas proved above which used that assumption. We can assume that $I=\{1,\ldots, N\}$ for some $N$ with the natural order. Then, $i+1, i-1$ are the successor and predecessor of $i$ respectively. We have a convenient formula  $\Delta_U(w,i)=H_{w,i}(U)-H_{w,i}(U|Z_i)=H_{w,i}(U)-H_{w,i+1}(U),$ similarly for $\Delta_V(w,i), \Delta_{U,V}(w,i).$

Recall that 
$$m(\cE_{\prec}(U)\cup\cE_{\prec}(V))-m(\cE_{\prec}(U))=\int_{\{0,1\}^I}\sum_{i\in I} r(w,i)d\mu(w),$$ where $r(w,i)=\max\{0,\Delta_V(w,i)-\frac{\varepsilon_V(w,i)}{\varepsilon_U(w,i)}\Delta_U(w,i)\}.$

Let $\kappa_0,\kappa_1,\kappa_2>0$ be auxiliary constants satisfying $\kappa_0<1/4, \kappa_1\leq 1/8,\kappa_2\leq \kappa_1/8.$
We will divide $\{0,1\}^I\times I$ into disjoint regions $\Omega_1,\ldots, \Omega_5$:
\begin{align*}
\Omega_1:=&\{(w,i)| \Delta_V(w,i)-\frac{\varepsilon_V(w,i)}{\varepsilon_U(w,i)}\Delta_U(w,i)<0\}\\
\Omega_2:=&\{(w,i)| \varepsilon_U(w,i)>\kappa_0\}\setminus \Omega_1\\
\Omega_3:=&\{(w,i)|\frac{\varepsilon_V(w,i)}{\varepsilon_U(w,i)}<\kappa_1\}\setminus (\Omega_1\cup\Omega_2)\\
\Omega_4:=&\{(w,i)|\frac{\varepsilon_V(w,i)}{\varepsilon_U(w,i)}>1-\kappa_1\}\setminus (\Omega_1\cup\Omega_2\cup\Omega_3)\\
\Omega_5:=&\{0,1\}^I\times I\setminus(\Omega_1\cup\Omega_2\cup\Omega_3\cup\Omega_4).
\end{align*}
Put $I_d:=\int_{\{0,1\}^I}\sum_{(w,i)\in \Omega_d}r(w,i)d\mu(w).$ Then $m(\cE_{\prec}(U)\cup\cE_{\prec}(V))-m(\cE_{\prec}(U))=I_1+\ldots +I_5.$ We shall estimate $I_d$ one by one. 
\begin{enumerate}
\item  $I_1=0$. \begin{proof} This is clear from the definition of $r(w,i)$.\end{proof}
\item $I_2 \leq -4\mu(V)\log\kappa_0.$
\begin{proof}
For any  $w\in\{0,1\}^I$ let $\tau_2(w)=\inf\{i\mid \varepsilon_U(w,i)> \kappa_0\}$, with the convention that $\tau_2(w)=+\infty$ if $\varepsilon_V(w,i)<\kappa_0$ for all $i\in I.$ %We have $\varepsilon_V(w,i)\leq \varepsilon_U(w,i),$ so 
We note that $\varepsilon_U(w,\tau_2(w))\leq 2 \varepsilon_U(w,\tau_2(w)-1)\leq 2\kappa_0\leq 1/2$. In particular $H_{w,\tau_2(w)}(V)=h(\varepsilon_V(w,\tau_2(w)))\leq h(\varepsilon_U(w,\tau_2(w)))\leq h(2\kappa_0).$
Any $(w,i)\in\Omega_2$ satisfies $i\geq \tau(w),$ so 
$$I_2\leq \int_{\{0,1\}^I}\sum_{i\geq \tau_2(w)}r(w,i)d\mu(w).$$
By Lemma \ref{lem-rsmallbigV}, $r(w,i)\leq \Delta_V(w,i)=H_{w,i}(V)-H_{w,i+1}(V).$ Hence
\begin{align*}I_2\leq& \int_{\{0,1\}^I}\sum_{i\geq \tau_2(w)}(H_{w,i}(V)-H_{w,i+1}(V))d\mu(w)=\int_{\tau_2(w)<+\infty}H_{w,\tau_2(w)}(V)d\mu(w)\\ \leq& h(2\kappa_0)\mu(\{\tau_2(w)<+\infty\}).\end{align*}
For every $w\in U$ we eventually have $\varepsilon_U(w,i)=1$ so $U\subset \{\tau_2(w)<+\infty\}$ and
$$\mu(U)=\int_{\tau_2(w)<+\infty} \varepsilon_U(w,\tau_2(w))d\mu(w)\geq \kappa_0\mu(\{\tau_2(w)<+\infty\}),$$ so $\mu(\{\tau_2(w)<+\infty\})\leq \kappa_0^{-1}\mu(U).$
Combining the inequalities we get
$$I_2\leq \kappa_0^{-1}h(2\kappa_0)\mu(U)\leq -4\mu(U)\log \kappa_0.$$
\end{proof}

\item $I_3\leq\frac{k(\kappa_0\kappa_1)}{h(\kappa_0)}H(U).$
\begin{proof}
Let $\tau_3(w):=\inf\{i\in I| \frac{\varepsilon_V(w,i)}{\varepsilon_U(w,i)}<\kappa_1 \text{ and } \varepsilon_U(w,i)\leq \kappa_0\}$. Then $\tau_3(w)\geq i$ for all $(w,i)\in \Omega_3.$ By Lemma \ref{lem-rsmallbigV},
\begin{align*}
I_3\leq& \int_{\{0,1\}^I}\sum_{i\geq \tau_3(w)}r(w,i)d\mu(w)\leq \int_{\{0,1\}^I}\sum_{i\geq \tau_3(w)}\Delta_V(w,i)d\mu(w)\\
=& \int_{\tau_3(w)<+\infty} H(V| Z_{\prec \tau_3(w)}=w_{\prec\tau_3(w)})d\mu(w)\\
=&  \int_{\tau_3(w)<+\infty} h(\varepsilon_V(w,\tau_3(w)))d\mu(w).\\
\end{align*}
But, $\varepsilon_V(w,i)\leq \kappa_1 \varepsilon_U(w,\tau_3(w))$ and $\varepsilon_U(w,\tau_3(w))<\kappa_0$ for all $w$ with $\tau_3(w)<+\infty.$ In particular 
\begin{align*}h(\varepsilon_V(w,\tau_3(w)))\leq & h(\kappa_1 \varepsilon_U(w,\tau_3(w)))\leq \max_{t\leq \kappa_0}\frac{h(\kappa_1 t)}{h(t)}h(\varepsilon_U(w,\tau_3(w)))\\=& \frac{h(\kappa_0\kappa_1)}{h(\kappa_0)} h(\varepsilon_U(w,\tau_3(w))).\end{align*}
Plugging this into our estimate on $I_3$, we obtain 
\begin{align*}I_3\leq& \frac{h(\kappa_0\kappa_1)}{h(\kappa_0)}\int_{\tau_3(w)<+\infty}h(\varepsilon_U(w,\tau_3(w)))d\mu(w)\\ =& \frac{h(\kappa_0\kappa_1)}{h(\kappa_0)}\int_{\tau_3(w)<+\infty}H_{w,\tau_3(w)}(U)d\mu(w) \leq \frac{h(\kappa_0\kappa_1)}{h(\kappa_0)} H(U).\end{align*}
\end{proof}
\item $I_4\leq \kappa_1 H(U)+\mu(U).$
\begin{proof}
%Let $\tau_4(w):=\inf\{i\in I| (w,i)\in \Omega_4\}.$ 
By Lemmas \ref{lem-rsmallbigV} and \ref{lem-s1bound}
\begin{align*}
I_4\leq & \int_{\{0,1\}^I} \sum_{(w,i)\in \Omega_4} r(w,i)d\mu(w)\\
\leq& \int_{\{0,1\}^I} \sum_{(w,i)\in \Omega_4} \left(\left(1-\frac{\varepsilon_V(w,i)}{\varepsilon_U(w,i)}\right)\Delta_U(w,i)+s_1(w,i)\right)d\mu(w)\\
\leq& \kappa_1\int_{\{0,1\}^I}\int_{i\in I}\Delta_U(w,i)d\mu(w)+\int_{\{0,1\}^I}\sum_{i\in I}s_1(w,i)d\mu(w)\\
\leq& \kappa_1 H(U) +H(V|U)\leq \kappa_1 H(U)+\mu(U).
\end{align*}
\end{proof}
\item $I_5\leq (32\kappa_0+4h(2\kappa_1^{-1}\kappa_2))H(U)-16\kappa_1^{-1/2}\log\kappa_2 H(U)^{1/2}\mu(U)^{1/2}.$
\begin{proof}
This is the most challenging case. Let $(w,i)\in \Omega_5$. Then $\varepsilon_U(w,i)\leq \kappa_0$ and $\kappa_1\leq \frac{\varepsilon_V(w,i)}{\varepsilon_U(w,i)}\leq 1-\kappa_1.$ We have $r(w,i)=s_2(w,i)+s_3(w,i),$ so 
$$I_5\leq \int_{\{0,1\}^I}\sum_{(w,i)\in \Omega_5}|s_2(w,i)|d\mu(w)+ \int_{\{0,1\}^I}\sum_{(w,i)\in \Omega_5}|s_3(w,i)|d\mu(w).$$
By Lemma \ref{lem-s3bound} and the fact that $\varepsilon_V(w,i)\leq \varepsilon_U(w,i)\leq \kappa_0$, we have
$$|s_3(w,i)|\leq 16(\varepsilon_U(w,i)\Delta_U(w,i)+\varepsilon_V(w,i)\Delta_V(w,i))\leq 16\kappa_0(\Delta_U(w,i)+\Delta_V(w,i))$$
Therefore 
\begin{align*}\int_{\{0,1\}^I}\sum_{(w,i)\in \Omega_5}|s_3(w,i)|d\mu(w)\leq& 16\kappa_0\int_{\{0,1\}^I}\sum_{i\in I}(\Delta_V(w,i)+\Delta_U(w,i))d\mu(w)\\
=&16\kappa_0(H(U)+H(V))\leq 32\kappa_0 H(U).\end{align*}
By Lemma \ref{lem-s2bound}
$$|s_2(w,i)|\leq -8\kappa_1^{-1/2}\log\kappa_2 \left(\sqrt{\Delta_U(w,i)s_1(w,i)}+\sqrt{\Delta_V(w,i)s_1(w,i)}\right)+4h(2\kappa_1^{-1}\kappa_2)\Delta_U(w,i).$$
By the Cauchy-Schwarz inequality and Lemma \ref{lem-s1bound}
\begin{align*}
\int_{\{0,1\}^I}&\sum_{(w,i)\in \Omega_5}|s_2(w,i)|d\mu(w)\\\leq & -8\kappa_1^{-1/2}\log\kappa_2\left[\left(\int_{\{0,1\}^I}\sum_{i\in I}\Delta_U(w,i)d\mu(w)\right)^{1/2}\left(\int_{\{0,1\}^I}\sum_{i\in I}s_1(w,i)d\mu(w)\right)^{1/2}\right.\\ +&\left. \left(\int_{\{0,1\}^I}\sum_{i\in I}\Delta_V(w,i)d\mu(w)\right)^{1/2}\left(\int_{\{0,1\}^I}\sum_{i\in I}s_1(w,i)d\mu(w)\right)^{1/2}\right]\\
&+ 4h(2\kappa_1^{-1}\kappa_2)\int_{\{0,1\}^I}\sum_{i\in I}\Delta_U(w,i)d\mu(w)\\
\leq& -8\kappa_1^{-1/2}\log\kappa_2 \left[H(U)^{1/2}H(V|U)^{1/2}+H(V)^{1/2}H(V|U)^{1/2}\right]\\ +&4h(2\kappa_1^{-1}\kappa_2) H(U)\\
\leq& -16\kappa_1^{-1/2}\log\kappa_2 H(U)^{1/2}\mu(U)^{1/2}+4h(2\kappa_1^{-1}\kappa_2) H(U).
\end{align*}
We put together the bounds on the integrals of $|s_2(w,i)|, |s_3(w,i)|$ to obtain 
$$I_5\leq (32\kappa_0 + 4h(2\kappa_1^{-1}\kappa_2))H(U)-16\kappa_1^{-1/2}\log\kappa_2 H(U)^{1/2}\mu(U)^{1/2}.$$
\end{proof}

\end{enumerate}

In total 
\begin{align*}
m(\cE_{\prec}(U)\cup\cE_{\prec}(V))-m(\cE_{\prec}(U))=&I_1+I_2+I_3+I_4+I_5\\
\leq& -4\mu(U)\log\kappa_0+\frac{h(\kappa_0\kappa_1)}{h(\kappa_0)}H(U)+\kappa_1 H(U) +\mu(U)\\ &+(32\kappa_0 + 4h(2\kappa_1^{-1}\kappa_2))H(U)-16\kappa_1^{-1/2}\log\kappa_2 H(U)^{1/2}\mu(U)^{1/2}\\
\leq& (-4\log \kappa_0+1)\mu(U)+(32\kappa_0+4 h(2\kappa_1^{-1}\kappa_2)\\ &+\frac{h(\kappa_0\kappa_1)}{h(\kappa_0)})H(U) -16\kappa_1^{-1/2}\log\kappa_2 H(U)^{1/2}\mu(U)^{1/2}.
\end{align*}
Set $\kappa_0=(-\log \mu(U))^{-1/3}, \kappa_1=(-\log\mu(U))^{-1/3}, \kappa_2=(-\log\mu(U))^{-1}.$ Our initial conditions on $\kappa_0,\kappa_1,\kappa_2$ will be satisfied for $\mu(U)$ small enough. The last inequality evaluates to 
$$m(\cE_{\prec}(U)\cup\cE_{\prec}(V))-m(\cE_{\prec}(U))\leq C\frac{\log|\log(\mu(U)|}{|\log(\mu(U)|^{1/3}}m(\cE_{\prec}(U))=o(m(\cE_{\prec}(U)),$$ where $C$ is an absolute positive constant.
\end{proof}
% \end{proof}

% \begin{cor}
% For any $\delta>0$ and $k\in \mathbb N$ there exists $\eta>0$ with the following property. Let $U_1,\ldots, U_k\subset \{0,1\}^I$ be a collection of subsets with $\mu(U_l)\leq \eta$ for $l=1,\ldots, k.$ Then 
% \end{cor}

%Under the assumption that \eqref{eq.EE} does not occur, \cite{EKL} shows 

\section{Proofs of Theorems \ref{mthm-infcontBernouilli}, \ref{mthm-fiidfinite} and \ref{mthm-approxhyp}}\label{sec-infcontproof}

\begin{proof}[Proof of Theorem \ref{mthm-infcontBernouilli}]
Let $\Gamma$ be a countable group. We construct a random linear order on $\Gamma$ by ordering its elements according to iid labels in $[0,1]$. Let $\tau$ be the corresponding probability distribution on $\mathfrak O$. Let $\mu=(\frac{1}{2}\delta_{0}+\frac{1}{2}\delta_{1})^\Gamma$ and let $m$ be the product of $\mu$, counting measure on $\Gamma$ and the Lebesgue measure on $[0,\log 2].$ 
Let $\cE\colon \cB(\{0,1\}^\Gamma,\mu)\to \cB(\mathfrak O\times \{0,1\}^\Gamma\times \Gamma\times [0,\log 2], \tau\times m)$ be the entropy support map. By Theorem \ref{mthm-confhom}, it is a $\Gamma$-equivariant conformal homomorphism, hence it induces an infinitesimal containment $(\{0,1\}^\Gamma,\mu)\xhookrightarrow{inf}\Gamma \times (\mathfrak O\times \{0,1\}^\Gamma\times [0,\log 2], \tau\times \mu\times {\rm Leb})$. By Lemma \ref{lem-gammaprod}, $\Gamma \times (\mathfrak O\times \{0,1\}^\Gamma\times [0,\log 2], \tau\times \mu\times {\rm Leb})\xhookrightarrow{proj}\Gamma$ so by Lemma \ref{lem-transitivity} $(\{0,1\}^\Gamma,\mu)\xhookrightarrow{inf}\Gamma.$ Finally, by Abert-Weiss theorem \cite{AbertWeiss2013BernoulliWeaklyContained}, 
$([0,1]^\Gamma, {\rm Leb}^\Gamma)\xhookrightarrow{w}(\{0,1\}^\Gamma, \mu)$ so using Lemma \ref{lem-transitivity}
once again we get $([0,1]^\Gamma, {\rm Leb}^\Gamma)\xhookrightarrow{inf}\Gamma.$\end{proof}

\begin{proof}[Proof of Theorem \ref{mthm-fiidfinite}]
By Theorem \ref{mthm-infcontBernouilli} 
$([0,1]^\Gamma, {\rm Leb}^\Gamma)\xhookrightarrow{inf}\Gamma.$ 
Let $\kappa$ be any factor of iid thinning with model $(Y,\nu), U$ so that  $\kappa=\cF(Y,U)$. 
By Lemma \ref{lem-ursweakcont} $(Y,\nu)\xhookrightarrow{proj} ([0,1]^\Gamma, {\rm Leb}^\Gamma)$. By Lemma \ref{lem-transitivity} $(Y,\nu)\xhookrightarrow{proj}\Gamma$, so $\kappa$ is weakly contained in $\Gamma$. By Lemma \ref{lem-urscontingamma}, $\kappa$ is a weak-* limit of finite unimodular random subsets.
\end{proof}

\begin{proof}[Proof of Theorem \ref{mthm-approxhyp}]
Follows from Theorem \ref{mthm-infcontBernouilli} and Proposition \ref{prop-approxhypinf}.
\end{proof}
Theorem \ref{mthm-infcontBernouilli} extends to the coset actions, with a slightly weaker conclusion.

\begin{thm}\label{thm-infcontcosetB}
Let $H\subset \Gamma$ be a subgroup. Then $([0,1]^{\Gamma/H},{\rm Leb}^{\Gamma/H})\xhookrightarrow{inf}\Gamma/H\times (W,\kappa)$ where $(W,\kappa)$ is an auxiliary probability measure preserving action of $\Gamma.$
\end{thm}
\begin{proof}
Let $I:=\mathbb N\times \Gamma/H.$ Let $$\cE\colon \cB(\{0,1\}^{\mathbb N\times \Gamma/H})\to \cB(\mathfrak O\times \{0,1\}^{\mathbb N\times \Gamma/H}\times (\mathbb N\times \Gamma/H)\times [0,\log 2])$$ be the entropy support map. We identify $(\{0,1\}^\mathbb N, (\frac{1}{2}\delta_0+\frac{1}{2}\delta_1)^\mathbb N)\simeq ([0,1],{\rm Leb})$ and restrict the action of $\Sym(\mathbb N\times \Gamma/H)$ to $\Gamma$ acting on the left on the second coordinate. Let $(W,\kappa):=(\mathfrak O\times [0,1]^{\Gamma/H}\times \mathbb N\times [0,\log 2], \tau \times {\rm Leb}^{\Gamma/H} \times |\cdot|\times {\rm Leb}).$ Then 
$$\cE\colon \cB([0,1]^{\Gamma/H},{\rm Leb}^{\Gamma/H})\to \cB(\Gamma/H\times W, |\cdot|\times \kappa)$$ is a $\Gamma$-equivariant conformal $\sigma$-algebra homomorphism. It follows that $([0,1]^{\Gamma/H},{\rm Leb}^{\Gamma/H})\xhookrightarrow{inf}(\Gamma/H\times W,|\cdot|\times \kappa).$ 
%The only differences in the proof of this theorem and Theorem \ref{mthm-infcontBernouilli} is that we do not have Lemma \ref{lem-gammaprod}
% for coset action, and we don't have Abert-Weiss theorem for $([0,1]^{\Gamma/H},{\rm Leb}^{\Gamma/H}).$ TODO
\end{proof}

\section{Proof of Chifan-Ioana's theorem for $(X,\mu)\xhookrightarrow{inf}\Gamma$.}\label{sec-QCI}
\begin{thm}\label{thm-QCIinf}
Let $\Gamma$ be an exact group and let $(X,\mu)$ be an ergodic p.m.p. action of $\Gamma$ with $(X,\mu)\xhookrightarrow{inf}\Gamma$. Let $\cR$ be the orbit equivalence relation and let $\cS$ be an equivalence subrelation of $\cR$. Then $X=\sqcup_{i=0} X_i$, where $\cS|_{X_0}$ is hyperfinite, and for each $i\geq 1$ either $X_i$ is empty or $\mu(X_i)\geq 0$ and the restriction $\cS|_{X_i}$ is strongly ergodic.
\end{thm}
\begin{proof}
Let $X_0'$ be the smooth part of the ergodic decomposition of $\cS$, let $X_0''$ be the union of all positive mass hyperfinite ergodic components and let $X_i,i\geq 1$ be the non-hyperfinite positive mass ergodic components. 
We are first going to show that the smooth part $X_0'$ is hyperfinite, i.e. almost every ergodic component of $\cS|_{X_0'}$ is hyperfinite.

Let $F\subset \Gamma$ be a finite subset. We construct a graphing $\Phi_F$ on $X$ by restricting the domain of $\gamma\in F$ to the set $\{x\in X\mid \gamma x\in [x]_{\cS}\}$. Let $\cS_F$ be the equivalence relation on $X$ generated by $\Phi_F$. It is a sub-relation of $\cS$.

Let $\varepsilon>0.$ By Theorem \ref{mthm-approxhyp} we can choose $\delta>0$ such that every set $U\subset [0,1]^\Gamma, {\rm Leb}^\Gamma(U)\leq \delta$ is $(F,k_n,\varepsilon)$-hyperfinite. This is where we used exactness of $\Gamma$. 

Divide $X_0$ into $\cS$-invariant subsets $X_0=X_0^1\sqcup \ldots\sqcup X_0^N$, with $\mu(X_0'^i)\leq \delta$. By Theorem \ref{mthm-approxhyp}, the sets $\cF(X,X_0'^i)$ are $(F,k,\varepsilon)$-hyperfinite. Put $\varepsilon':=-33|F|\varepsilon\log\varepsilon$. By Theorem \ref{thm-pairmodelfree}, there exist $V^i\subset X_0'^i$ such that $\mu(V^i)\leq\varepsilon'\nu(X_0'^i)$ and the equivalence relation generated by $F$ restricted to $X_0'^i\setminus V^i$ is finite with classes of size at most $k$. Put $V:=\bigcup_{i=1}^N V^i$. Since $X_0'^i$ are $\cS$-invariant, they are also $\cS_F$-invariant. We deduce that the relation generated by $x\sim \gamma x, \gamma\in F, \gamma x\in [x]_{\cS}, x,\gamma x\in X_0'\setminus V$ is finite. Since the measure $\mu(V)\leq \varepsilon' \mu(X_0')$ can be made arbitrarily small, the relation $\cS_F$ restricted to $X_0'$ is hyperfinite. The union of subrelations $\cS_F$ for finite symmetric $F$ is $\cS$, so the latter is also hyperfinite on $X_0'$.

Put $X_0=X_0'\cup X_0''.$ The restriction of $\cS$ to $X_0$ is obviously hyperfinite.

It remains to show that $\cS$ is strongly ergodic on the components $X_i, i\geq 1.$ Let $\cS_F, \delta_n,\varepsilon_n$ be as before. Suppose $X_i$ is not strongly ergodic. By \cite{connes1981amenable}, there is class preserving map $\rho\colon (X_1,\cS,\mu)\to (Z,\cQ,\tau)$, where $(Z,\cQ,\tau)$ is a hyperfinite p.m.p. equivalence relation. 
Let $\cQ_n$ be an ascending sequence of finite subrelations of $\cQ$ which sum up to $\cQ.$ Let $\cS_n:=\{(x,x')\in \cS\cap X_i\times X_i\mid (\rho(x),\rho(x')\in \cQ_n\}$. Clearly $\cS_n$ are subrelations of $\cS|_{X_i}$ which sum up to $\cS|_{X_i}$. The ergodic decomposition of $\cQ_n$ is smooth, as the relation is finite, so the same is true for $\cS_n$. By the first part of the proof, $\cS_n$ is hyperfinite. Since these equivalence relations sum up to $\cS|_{Xi},$ the latter is also hyperfinite. This contradicts our assumption that $X_i$ is a non-hyperfinite ergodic component.
\end{proof}

\section{Questions}\label{sec-questions}
\subsection{Other example of infinitesimal containment in $\Gamma$} Let $\Gamma$ be a countable group. We have proved that any p.m.p. action $(X,\mu)$ weakly contained in $([0,1]^\Gamma,{\rm Leb}^\Gamma)$ is infinitesimally contained in $\Gamma$. Several interesting classes of such actions were described in \cite{hayes2019weak,hayes2021max,hayes2021harmonic}. 

On the other hand, non-trivial actions that factor through action of a compact group are never factors of Bernoulli and rarely are weak factors. We don't know if there are any examples of compact actions that are infinitesimally contained in  $\Gamma$.  Here is a candidate for a profinite action of $\Gamma$ that could be infinitesimally embedded in $\Gamma$.

Select a fast growing sequence of natural numbers $t_n$, e.g. $t_n=2^{2^n}$. Construct  a sequence of finite index subgroups $\Lambda_n$ of the free group $F_d$ as follows. Let $\Lambda_0$ be a uniform random subgroup of index $t_0$ and for any $n\geq1$ let $\Lambda_n$ be the uniform random subgroup of $\Lambda_{n-1}$ of index $t_n$ in $\Lambda_{n-1}$. In this way we obtain a nested sequence of finite subgroups $\Lambda_n$. Let $(X,\mu)$ be the inverse limit of the actions on $F_d/\Lambda_n$ with normalized counting measures (see \cite{abert2012rank}). 
\begin{conj}
The action $F_d\curvearrowright (X,\mu)$ is infinitesimally embedded in $F_d.$
\end{conj}

Taking ultra-products of finite actions provides another source of candidates. We refer to \cite{carderi2021non} for the definition of ultralimits of p.m.p. actions. Let $\cU$ be a non principal ultrafilter on $\mathbb N$. The ultralimit $(X_{\cU},\mu_{\cU})$ of p.m.p. actions $(X_n,\mu_n)$  is again a p.m.p. action, with the property that the dynamics of any finite tuple of sets $A_1,\ldots,A_k$ in $(X_{\cU},\mu_{\cU})$ can be approximated in $(X_n,\mu_n)$ for some $n\in\mathbb N.$ 
\begin{conj}
Let $\cU$ be a non-principal ultrafilter. Let $(X_n,\mu_n)$ be the uniform random action of $F_d$ on an $n$-element set, with normalized counting measure. 
Then  $(X_{\cU},\mu_{\cU})\xhookrightarrow{inf}F_d$ almost surely.
\end{conj}

We now take a closer look at actions of $\Gamma$ that factor through the action of a compact group. Let $G$ be a compact group with the Haar probability measure $\mu$. Let $\rho\colon \Gamma\to G$ be a homomorphism. Group $\Gamma$ acts on $(G,\mu)$ via $\gamma g:=\rho(\gamma)g.$

Thanks to Example \ref{ex-compactactions}, we know that $\Gamma/\rho^{-1}(H)\xhookrightarrow{proj}(G,\mu)$ for any closed subgroup $H\subset G$ with empty interior. This provides an obstruction to infinitesimal containment in $\Gamma$. Let $\Sigma$ be any countable set with $\Gamma$ action and $(Z,\tau)$ any measure preserving action. By Lemma \ref{lem-transitivity}, $(G,\mu)\xhookrightarrow{inf}\Sigma  \times (Z,\tau)$ implies $\Gamma/\rho^{-1}(H)\xhookrightarrow{proj}\Sigma \times (Z,\tau).$ In particular, $\Sigma\times (Z,\tau)$ should have $\rho^{-1}(H)$ almost invariant subsets. This is impossible unless $\Sigma$ has $\rho^{-1}(H)$ almost invariant sets. 

It follows, for example, that $(G,\mu)$ is not infinitesimally contained in $\Gamma$ unless $\rho^{-1}(H)$ is amenable for all proper closed subgroups $H\subset G$. We state the following as a Conjecture, but they are more like questions. The only evidence is our inability to find counter-examples. 

\begin{conj}\label{conj-compact}
Let $\rho\colon \Gamma\to  G$ be as above. Then 
$$(G,\mu)\xhookrightarrow{inf}\bigsqcup_{H\subset G} \left[\Gamma/\rho^{-1}(H)\times (Z_H,\tau_H)\right],$$ where $H$ runs over conjugacy classes of closed subgroups of $G$ with empty interior and $(Z_H,\tau_H)$ is an auxiliary probability measure preserving action.
\end{conj}
Note that when $\rho^{-1}(H)$ is amenable, then $\Gamma/\rho^{-1}(H)\times (Z_H,\tau_H)\xhookrightarrow{proj}\Gamma$, by Proposition \ref{prop-amenableproj}.
A special case of Conjecture \ref{conj-compact} is:
\begin{conj}\label{conj-compactnonam}
Suppose $\rho(\Lambda)$ is open for all non-amenable subgroups $\Lambda$ of $\Gamma$. Then 
$(G,\mu)\xhookrightarrow{inf}\Gamma.$
\end{conj}
We can specialize even further 
\begin{conj}\label{conj-lattices}
\begin{enumerate}
    \item Let $\Gamma=\SL_2(\mathbb Z).$ For any prime $p$, $(\SL_2(\mathbb Z_p),\mu)\xhookrightarrow{inf}\SL_2(\mathbb Z)$.
    \item Let $\Gamma$ be a cocompact arithmetic lattice in $\SL_2(\mathbb C)$ of the second type (so that $\Gamma$ does not intersect any conjugate $H$ of $\SL_2(\mathbb R)$ as a lattice in $H$). Let $\widehat{\Gamma}^c$ be the congruence completion of $\Gamma$. Then $(\widehat{\Gamma}^c,\mu)\xhookrightarrow{inf}\Gamma.$
\end{enumerate}   
\end{conj}

\begin{question}
    Conjecture \ref{conj-compactnonam} together with Theorem \ref{thm-QCIinf} suggests that Chifan-Ioana's theorem might extend to certain compact actions - ones that are infinitesimally contained in $\Gamma$. Could one prove it independently of infinitesimal containment?
\end{question}

We conclude with a question on expander subgraphs of congruence quotients of $\SL_2(\mathbb Z).$

\begin{question}\label{q-subexpanders}
Let $S\subset \SL_2(\mathbb Z)$ be a finite symmetric subset generating a Zariski sense subgroup and let $p$ be a prime. Let $G_n$ be the graph $\Cay(\SL_2(\mathbb Z/p^n\mathbb Z),S).$ Does there exist a sequence of subsets $Q_n\subset G_n$ such that the induced subgraphs form an expander family and $\lim_{n\to\infty}|Q_n|/|G_n|=0$?
\end{question}

A positive answer to Conjecture \ref{conj-lattices} (1) together with Theorem \ref{mthm-approxhyp} would imply a negative answer to Question \ref{q-subexpanders}.

%One can ask for an analogue for general compact action. 
% \begin{defn}
% Let $\Gamma\curvearrowright (X,\mu)$ be a p.m.p. action. Let $\Lambda\subset \Gamma$ be a subgroup. A sub equivalence relation $\cS$of the orbit equivalence relation on $X$ is called hyper-$\Lambda$ finite, if it is contained in an ascending union of equivalence relations $\cQ_n$, where each $\cQ_n$ contains the orbit equivalence of $\Lambda$ as a finite index subrelation.  
% \end{defn}

% If we were to believe in Conjecture \ref{conj-compact}, then there is hope that an extension of our proof of Theorem \ref{thm-QCIinf} would yield.
% \begin{conj}
% Suppose that $\Gamma$ is exact and let $\rho\colon \Gamma\to G$ be as before. Let $\cS$ be an equivalence subrelation of the orbit equivalence relation on $G$. Then $X=\sqcup_{i=0}X_i$ where $X_i$ are $\cS$-invariant, $\mu(X_i)>0$ and $\cS|_{X_i}$ is strongly ergodic for $i\geq 1$. The subset $X_0$ further decomposes as 
% $$X_0=\bigsqcup_{H\subset G}X_0^H,$$
% where $H$ runs over conjugacy classes of closed subgroups with empty interior, $X_0^H$ is $\cS$-invariant and the restriction $\cS|_{X_0^H|}$ is hyper-$\rho^{-1}(H)$-finite. \end{conj}

%For any closed subgroup $H\subset G$ we consider the action $\Gamma\curvearrowright (G/H,\mu_{G/H})$ where $\mu_{G/H}$ is the projection of $\mu$ and $\Gamma$ acts via $\gamma gH=\rho(\gamma)gH.$

\bibliographystyle{abbrv}
\bibliography{ref}

\begin{thebibliography}{10}

\bibitem{aaserud2018approximate}
A.~N. Aaserud and S.~Popa.
\newblock Approximate equivalence of group actions.
\newblock {\em Ergodic Theory and Dynamical Systems}, 38(4):1201--1237, 2018.

\bibitem{abert2014kesten}
M.~Ab{\'e}rt, Y.~Glasner, and B.~Vir{\'a}g.
\newblock Kesten's theorem for invariant random subgroups.
\newblock {\em Duke Mathematical Journal}, 163(3):465--488, 2014.

\bibitem{abert2012rank}
M.~Ab{\'e}rt and N.~Nikolov.
\newblock Rank gradient, cost of groups and the rank versus heegaard genus
  problem.
\newblock {\em Journal of the European Mathematical Society}, 14(5):1657--1677,
  2012.

\bibitem{AbertWeiss2013BernoulliWeaklyContained}
M.~Ab{\'e}rt and B.~Weiss.
\newblock Bernoulli actions are weakly contained in any free action.
\newblock {\em Ergodic Theory and Dynamical Systems}, 33:323--333, 2013.

\bibitem{agol2008criteria}
I.~Agol.
\newblock Criteria for virtual fibering.
\newblock {\em Journal of Topology}, 1(2):269--284, 2008.

\bibitem{AldousLyons}
D.~Aldous and R.~Lyons.
\newblock Processes on unimodular random networks.
\newblock {\em Electronic Journal of Probability}, 12:1454--1508, 2007.

\bibitem{baccelli2021unimodular}
F.~Baccelli, M.-O. Haji-Mirsadeghi, and A.~Khezeli.
\newblock Unimodular hausdorff and minkowski dimensions.
\newblock {\em Electronic Journal of Probability}, 26:1--64, 2021.

\bibitem{BackhauszGerencserHarangi2019EntropyFIID}
A.~Backhausz, B.~Gerencs{'e}r, and V.~Harangi.
\newblock Entropy inequalities for factors of iid.
\newblock {\em Groups, Geometry, and Dynamics}, 13(2):389--414, 2019.

\bibitem{bevilacqua2025metric}
E.~Bevilacqua and L.~Bowen.
\newblock Metric criteria for fixed price of countable groups.
\newblock {\em arXiv preprint arXiv:2510.05459}, 2025.

\bibitem{bjorklund2025intersection}
M.~Bj{\"o}rklund, T.~Hartnick, and Y.~Karasik.
\newblock Intersection spaces and multiple transverse recurrence.
\newblock {\em Journal d'Analyse Math{\'e}matique}, pages 1--54, 2025.

\bibitem{BNSWW13}
J.~Brodzki, G.~A. Niblo, J.~Špakula, R.~Willett, and N.~Wright.
\newblock Uniform local amenability.
\newblock {\em Journal of Noncommutative Geometry}, 7(3):583--603, 2013.

\bibitem{BurtonKechris2020WeakContainment}
P.~J. Burton and A.~S. Kechris.
\newblock Weak containment of measure-preserving group actions.
\newblock {\em Ergodic Theory and Dynamical Systems}, 40(10):2681--2733, 2020.

\bibitem{carderi2021non}
A.~Carderi, D.~Gaboriau, and M.~de~La~Salle.
\newblock Non-standard limits of graphs and some orbit equivalence invariants.
\newblock {\em Annales Henri Lebesgue}, 4:1235--1293, 2021.

\bibitem{ChifanIoana2010BernoulliSubequivalence}
I.~Chifan and A.~Ioana.
\newblock Ergodic subequivalence relations induced by a bernoulli action.
\newblock {\em Geometric and Functional Analysis}, 20(1):53--67, 2010.

\bibitem{connes1981amenable}
A.~Connes, J.~Feldman, and B.~Weiss.
\newblock An amenable equivalence relation is generated by a single
  transformation.
\newblock {\em Ergodic theory and dynamical systems}, 1(4):431--450, 1981.

\bibitem{csoka2020entropy}
E.~Cs{\'o}ka, V.~Harangi, and B.~Vir{\'a}g.
\newblock Entropy and expansion.
\newblock {\em Annales de L'Institut Henri Poincare Section (B) Probability and
  Statistics}, 56(4):2428--2444, 2020.

\bibitem{csokapetemester}
E.~{Cs{\'o}ka}, P.~{Mester}, and G.~{Pete}.
\newblock {Quantitative indistinguishability and sparse and dense clusters in
  factor of IID percolations}.
\newblock {\em arXiv e-prints}, page arXiv:2512.07740, Dec. 2025.

\bibitem{elek2021uniform}
G.~Elek.
\newblock Uniform local amenability implies property a.
\newblock {\em Proceedings of the American Mathematical Society},
  149(6):2573--2577, 2021.

\bibitem{FeldmanMoore1977}
J.~Feldman and C.~C. Moore.
\newblock Ergodic equivalence relations, cohomology, and von neumann algebras.
\newblock {\em Transactions of the American Mathematical Society},
  234(2):289--324, 1977.

\bibitem{fraczyk2023poisson}
M.~Fraczyk, S.~Mellick, and A.~Wilkens.
\newblock Poisson-voronoi tessellations and fixed price in higher rank.
\newblock {\em arXiv preprint arXiv:2307.01194}, 2023.

\bibitem{Furstenberg2008}
H.~Furstenberg.
\newblock Ergodic fractal measures and dimension conservation.
\newblock {\em Ergodic Theory and Dynamical Systems}, 28(2):405--422, 2008.

\bibitem{Furstenberg2014}
H.~Furstenberg.
\newblock {\em Ergodic Theory and Fractal Geometry}, volume 120 of {\em CBMS
  Regional Conference Series in Mathematics}.
\newblock American Mathematical Society, 2014.

\bibitem{Gab1}
D.~Gaboriau.
\newblock Co\^{u}t des relations d'\'{e}quivalence et des groupes.
\newblock {\em Invent. Math.}, 139(1):41--98, 2000.

\bibitem{GaboriauSurvey}
D.~Gaboriau.
\newblock Orbit equivalence and measured group theory.
\newblock In {\em Proceedings of the {I}nternational {C}ongress of
  {M}athematicians. {V}olume {III}}, pages 1501--1527. Hindustan Book Agency,
  New Delhi, 2010.

\bibitem{GamarnikSudan2017Limits}
D.~Gamarnik and M.~Sudan.
\newblock Limits of local algorithms over sparse random graphs.
\newblock {\em The Annals of Probability}, 45(4):2353--2376, 2017.

\bibitem{glasner2015uniformly}
E.~Glasner and B.~Weiss.
\newblock Uniformly recurrent subgroups.
\newblock {\em Recent trends in ergodic theory and dynamical systems},
  631:63--75, 2015.

\bibitem{HatamiLovaszSzegedy2014LocalGlobal}
H.~Hatami, L.~Lov{'a}sz, and B.~Szegedy.
\newblock Limits of local-global convergent graph sequences.
\newblock {\em Geometric and Functional Analysis}, 24(1):269--296, 2014.

\bibitem{hayes2019weak}
B.~Hayes.
\newblock Weak equivalence to bernoulli shifts for some algebraic actions.
\newblock {\em Proceedings of the American Mathematical Society},
  147(5):2021--2032, 2019.

\bibitem{hayes2021harmonic}
B.~Hayes.
\newblock Harmonic models and bernoullicity.
\newblock {\em Compositio Mathematica}, 157(10):2160--2198, 2021.

\bibitem{hayes2021max}
B.~Hayes.
\newblock Max-min theorems for weak containment, square summable homoclinic
  points, and completely positive entropy.
\newblock {\em Indiana University Mathematics Journal}, 70(4):1221--1266, 2021.

\bibitem{HochmanShmerkin2012}
M.~Hochman and P.~Shmerkin.
\newblock Local entropy averages and projections of fractal measures.
\newblock {\em Annals of Mathematics}, 175(3):1001--1059, 2012.

\bibitem{hutchcroft_2020}
T.~Hutchcroft.
\newblock Non-intersection of transient branching random walks.
\newblock {\em Probability Theory and Related Fields}, 178(1--2):1--23, Feb
  2020.

\bibitem{hutchcroft2024percolation}
T.~Hutchcroft and M.~Pan.
\newblock Percolation at the uniqueness threshold via subgroup relativization.
\newblock {\em arXiv preprint arXiv:2409.12283}, 2024.

\bibitem{jardon2025exactness}
H.~Jard{\'o}n-S{\'a}nchez, S.~Mellick, A.~Poulin, and K.~Wr{\'o}bel.
\newblock Exactness and the topology of the space of invariant random
  equivalence relations.
\newblock {\em arXiv preprint arXiv:2502.11622}, 2025.

\bibitem{KaimanovichAmenability}
V.~A. Kaimanovich.
\newblock Amenability, hyperfiniteness, and isoperimetric inequalities.
\newblock {\em C. R. Acad. Sci. Paris S\'er. I Math.}, 325(9):999--1004, 1997.

\bibitem{Kechris2010GlobalAspects}
A.~S. Kechris.
\newblock {\em Global Aspects of Ergodic Group Actions}, volume 160 of {\em
  Mathematical Surveys and Monographs}.
\newblock American Mathematical Society, Providence, RI, 2010.

\bibitem{khezeli2025products}
A.~Khezeli.
\newblock Products of infinite countable groups have fixed price one.
\newblock {\em arXiv preprint arXiv:2509.08325}, 2025.

\bibitem{klenke2020martingale}
A.~Klenke.
\newblock Martingale convergence theorems and their applications.
\newblock In {\em Probability Theory}, Universitext, chapter~11, pages
  241--273. Springer, 2020.

\bibitem{lackenby2002heegaard}
M.~Lackenby.
\newblock Heegaard splittings, the virtually haken conjecture and property tau.
\newblock {\em arXiv preprint math/0205327}, 2002.

\bibitem{lackenby2005expanders}
M.~Lackenby.
\newblock Expanders, rank and graphs of groups.
\newblock {\em Israel Journal of Mathematics}, 146(1):357--370, 2005.

\bibitem{miller2017edge}
B.~D. Miller and A.~Tserunyan.
\newblock Edge sliding and ergodic hyperfinite decomposition.
\newblock {\em arXiv preprint arXiv:1704.06019}, 2017.

\bibitem{OrnsteinWeiss1987}
D.~S. Ornstein and B.~Weiss.
\newblock Entropy and isomorphism theorems for actions of amenable groups.
\newblock {\em Journal d'Analyse Mathématique}, 48:1--141, 1987.

\bibitem{Osajda2018}
D.~Osajda.
\newblock Residually finite non-exact groups.
\newblock {\em Geometric and Functional Analysis}, 28(2):509--517, 2018.

\bibitem{Oza00}
N.~Ozawa.
\newblock Amenable actions and exactness for discrete groups.
\newblock {\em C. R. Acad. Sci. Paris Sér. I Math.}, 330(8):691--695, 2000.

\bibitem{pete2025nonamenable}
G.~Pete and S.~Rokob.
\newblock Nonamenable poisson zoo.
\newblock {\em arXiv preprint arXiv:2505.07145}, 2025.

\bibitem{schramm2011hyperfinite}
O.~Schramm.
\newblock Hyperfinite graph limits.
\newblock In {\em Selected Works of Oded Schramm}, pages 639--645. Springer,
  2011.

\bibitem{thorisson1999point}
H.~Thorisson.
\newblock Point-stationarity in d dimensions and palm theory.
\newblock {\em Bernoulli}, pages 797--831, 1999.

\bibitem{Zimmer3}
R.~J. Zimmer.
\newblock Amenable ergodic group actions and an application to poisson
  boundaries of random walks.
\newblock {\em Journal of Functional Analysis}, 27(3):350--372, 1978.

\bibitem{Zimmer1984}
R.~J. Zimmer.
\newblock {\em Ergodic Theory and Semisimple Groups}, volume~81 of {\em
  Monographs in Mathematics}.
\newblock Birkhäuser, 1984.

\end{thebibliography}
\end{document}